\documentclass[11pt,oneside,reqno]{amsart}
\usepackage{geometry}
\usepackage[utf8]{inputenc}
\usepackage{amstext,latexsym,amsbsy,amsmath,amssymb,amsthm,mathtools,relsize,geometry,enumerate}
\usepackage{hyperref}
\hypersetup{
  colorlinks   = true, 
  urlcolor     = blue, 
  linkcolor    = red, 
  citecolor   = red 
}
\usepackage{mathtools,enumerate,enumitem}

\usepackage{mathrsfs}

\usepackage{comment}
\usepackage{graphicx}
\usepackage{epstopdf}
\setkeys{Gin}{width=\linewidth,totalheight=\textheight,keepaspectratio}
\graphicspath{{./images/}}

\usepackage{multicol}
\usepackage{booktabs}
\usepackage{mathrsfs}
\usepackage{color}

\numberwithin{equation}{section}

\title{The inviscid limit for long time statistics of the one-dimensional stochastic Ginzburg-Landau equation}

 \author{Hung D.~Nguyen$^1$}

\address{$^1$ Department of Mathematics, University of Tennessee, Knoxville, Tennessee, USA}

\newcommand{\nbb}{\mathbb{N}}
\newcommand{\rbb}{\mathbb{R}}

\newcommand{\cbb}{\mathbb{C}}

\newcommand{\vbb}{\mathbb{V}}

\newcommand{\W}{\mathcal{W}}

\newcommand{\Pcal}{\mathcal{P}}

\newcommand{\la}{\langle}
\newcommand{\ra}{\rangle}

\newcommand{\util}{\tilde{u}}

\newcommand{\taut}{\tilde{\tau}}

\newcommand{\nt}{\notag}

\newcommand{\vbar}{\overline{v}}
\newcommand{\zbar}{\overline{z}}

\newcommand{\uhat}{\widehat{u}}

\newcommand{\ubar}{\overline{u}}
\newcommand{\etabar}{\overline{\eta}}

\newcommand{\ug}{u^\gamma}

\newcommand{\vg}{v^\gamma}

\newcommand{\Xg}{X^\gamma}
\newcommand{\Yg}{Y^\gamma}

\newcommand{\Xhat}{\widehat{X}}
\newcommand{\Yhat}{\widehat{Y}}
\newcommand{\What}{\widehat{\mathcal{W}}}

\renewcommand{\i}{\textup{i}}

\newcommand{\mi}{\wedge}

\newcommand{\Tr}{\text{Tr}}

\renewcommand{\d}{\textup{d}}

\newcommand{\f}{\varphi}

\newcommand{\grad}{\nabla}

\newcommand{\Fcal}{\mathcal{F}}
\newcommand{\E}{\mathbb{E}}

\renewcommand{\P}{\mathbb{P}}

\newcommand{\Law}{\textup{Law}}

\renewcommand{\Re}{\text{Re}}

\renewcommand{\H}{\mathcal{H}}

\newcommand{\close}{\!\!\!}

\newcommand{\TV}{\textup{TV}}
\newcommand{\KL}{\textup{KL}}

\theoremstyle{plain}
\newtheorem{theorem}{Theorem}[section]

\newtheorem{lemma}[theorem]{Lemma}
\newtheorem{assumption}[theorem]{Assumption}
\newtheorem{proposition}[theorem]{Proposition}
\newtheorem{definition}[theorem]{Definition}

\theoremstyle{definition}

\newtheorem{remark}[theorem]{Remark}

\numberwithin{equation}{section}

\begin{document}

\begin{abstract}
    
We study the long time statistics of a one-dimensional Ginzburg-Landau equation with cubic nonlinearity while being subjected to random perturbations via an additive Gaussian noise. Under the assumption that sufficiently many directions of the phase space are stochastically forced, we find that in the inviscid limit, the solutions are well approximated by the stochastic Schr\"odinger equation on the infinite time horizon. Moreover, the convergence is established to exhibit a quantitative order of log log rate with respect to the vanishing viscosity. Our proof relies on the uniform polynomial mixing of the Ginzburg-Landau equation as well as the convergence of the sequence of invariant probability measures in suitable Wasserstein distances.
\end{abstract}

\maketitle

\section{Introduction} \label{sec:intro}

For each $\gamma\in(0,1)$, we consider the following stochastic system in the unknown variable $\ug(t)=\ug(x,t):[0,1]\times[0,\infty)\to\cbb$
\begin{align} \label{eqn:Ginzburg_Landau:original}
\d u^\gamma(t) & = (\gamma +\i) \triangle u^\gamma(t)\d t + \i |u^\gamma(t)|^2u^\gamma (t) \d t-\alpha u^\gamma(t)\d t +Q\d W(t),\\
u^\gamma(x,t) & = 0,\quad x\in \{0,1\},\,t>0, \nt \\
u^\gamma(x,0) & = u_0(x),\quad x\in [0,1]. \nt 
\end{align}
System \eqref{eqn:Ginzburg_Landau:original} is the complex Ginzburg-Landau (CGL) equation, arising in the modeling of traveling waves in dissipative dynamics. On the right-hand side of \eqref{eqn:Ginzburg_Landau:original}, $\gamma\in(0,1)$ represents the viscosity, $\alpha>0$ denotes the damping constant, $Q\d W(t)$ is a Gaussian process which is white in time and whose spatial correlation is characterized by the symmetric bounded operator $Q$ that satisfies certain conditions. Historically, equation \eqref{eqn:Ginzburg_Landau:original} was developed in the seminal work of \cite{ginzburg1950theory} describing superconductivity. Since then, it has found many applications in various areas of physics, e.g.,  chemical reaction \cite{huber1993universal,kuramoto1975formation}, hydrodynamic instability theory \cite{blennerhassett1980generation,moon1982three,
moon1983transitions,newell1969finite,stuart1978frs}, and waves on deep water \cite{dias1999nonlinear}.

In the absence of viscosity, that is when setting $\gamma=0$, equation \eqref{eqn:Ginzburg_Landau:original} is formally reduced to the following focusing nonlinear Schr\"odinger (NLS) equation
\begin{align} \label{eqn:Schrodinger:original}
\d  u(t) & = \i \triangle u(t)\d t + \i |u(t)|^2u (t)\d t -\alpha u(t) \d t+Q\d  W(t),\\
u(x,t) & = 0,\quad x\in \{0,1\},\,t>0, \nt \\
u(x,0) & = u_0(x),\quad x\in [0,1]. \nt 
\end{align}
Not only can equation \eqref{eqn:Schrodinger:original} be derived in similar phenomena as in the case of the CGL equation, it is also employed to study nonlinear optics \cite{agrawal2011nonlinear,arecchi1993transition}, the propagation of laser beams \cite{fibich2015nonlinear}, and the propagation of solitons \cite{ablowitz1981solitons,newell1985solitons}. Under suitable assumptions on the noise structure, the large time behaviors of \eqref{eqn:Ginzburg_Landau:original} and \eqref{eqn:Schrodinger:original} are well-known. More specifically, for fixed $\gamma>0$, equation \eqref{eqn:Ginzburg_Landau:original} possesses unique ergodicity \cite{hairer2002exponential,odasso2006ergodicity} and that the equilibrium relaxation rate is exponentially fast. On the other hand, equation \eqref{eqn:Schrodinger:original} was established to be only polynomially attractive toward the unique invariant probability measure, owing to the lack of strong dissipation \cite{debussche2005ergodicity}. It is therefore a matter of interest to compare the statistically steady states of \eqref{eqn:Ginzburg_Landau:original} and those of \eqref{eqn:Schrodinger:original} in the inviscid limit, i.e., as $\gamma\to 0$. Notably, when the random force is proportional to the square root of the viscosity, there is a literature in this direction on the following equation 
\begin{align} \label{eqn:Ginzburg-Landau:original:gamma.QdW(t)}
\partial_t u^\gamma(t) & = (\gamma +\i) \triangle u^\gamma(t) \pm \i |u^\gamma(t)|^2u^\gamma (t) -\alpha u^\gamma(t) +\sqrt{\gamma} Q\partial_t W(t).
\end{align}
It has been shown that the stationary solutions of \eqref{eqn:Ginzburg-Landau:original:gamma.QdW(t)} converge toward those of the deterministic NLS equation
\begin{align} \label{eqn:Schrodinger:deterministic}
\partial_t u(t) & = \i \triangle u(t) \pm \i |u(t)|^2u (t) -\alpha u(t).
\end{align}
See for examples the work of \cite{kuksin2004randomly,kuksin2013weakly,
shirikyan2011local, zine2022inviscid}. However, much less is known about the approximation of \eqref{eqn:Ginzburg_Landau:original} by \eqref{eqn:Schrodinger:original} in the vanishing viscosity, and particularly whether the mixing rates and invariant measures of \eqref{eqn:Ginzburg_Landau:original} resemble those of \eqref{eqn:Schrodinger:original}. In light of recent works dealing with similar issues for singular parameter limits \cite{glatt2022short,
glatt2021mixing,nguyen2023small}, our main goal of the article is two-fold. Firstly, we seek to identify a general set of conditions on the random forcing so as to establish ergodic properties for \eqref{eqn:Ginzburg_Landau:original} via a spectral gap that is uniform with respect to the viscosity constant $\gamma$. Secondly, we demonstrate that the statistically steady states of \eqref{eqn:Ginzburg_Landau:original} can be related to \eqref{eqn:Schrodinger:original}, allowing for verifying the validity of the inviscid limit for the solutions of \eqref{eqn:Ginzburg_Landau:original} on the infinite time horizon. In what follows, we provide an overview of our main theorems and refer the reader to Section \ref{sec:results} for a more rigorous description.

\subsection{Summary of main results} \label{sec:intro:result} 

Our first result giving the uniform ergodicity of \eqref{eqn:Ginzburg_Landau:original} is summarized as follows:

\begin{theorem} \label{thm:ergodicity:CGL:heuristic}
Under general conditions on the noise structure $Q\d W(t)$, for all $\gamma$ sufficiently small, equation \eqref{eqn:Ginzburg_Landau:original} admits a unique invariant probability measure $\nu^\gamma$. Furthermore, for all $q\ge 2$, initial data $u_1^0,u_2^0$ with sufficient spatial regularity, and suitable Lipschitz bounded function $f$, the following holds
\begin{align} \label{ineq:ergodicity:CGL:heuristic}
\big|\E f\big(\ug(t;u_1^0)\big) -\E f \big(\ug(t;u_2^0)\big) \big|\le \frac{C}{(1+t)^q},\quad t\ge 0,
\end{align}
for some positive constant $C$ independent of $t$ and $\gamma$.
\end{theorem}

We refer the reader to Theorem \ref{thm:ergodicity:Ginzburg_Landau} for a precise statement of Theorem \ref{thm:ergodicity:CGL:heuristic}. As mentioned above, we remark that the problem of large time behaviors for \eqref{eqn:Ginzburg_Landau:original} is not new as it was previously shown to possess geometric ergodicity in $L^2(0,1)$  \cite{hairer2002exponential,odasso2006ergodicity}. The novelty of Theorem \ref{thm:ergodicity:CGL:heuristic} is the polynomial mixing rate that does not depend on the parameter $\gamma$ provided that one starts the dynamics from initial data with higher spatial regularity.

Turning to equation \eqref{eqn:Schrodinger:original}, we recall from \cite{debussche2005ergodicity} that \eqref{eqn:Schrodinger:original} also admits a unique invariant probability measure $\nu^0$. Our second main result below asserts that as $\gamma\to 0$, $\nu^\gamma$ is well approximated by $\nu^0$ with respect to suitable Wasserstein distances. 
\begin{theorem} \label{thm:gamma->0:nu^gamma-nu^0:heuristic} Under the same hypotheses of Theorem \ref{thm:ergodicity:CGL:heuristic}, there exist suitable Wasserstein distances $\W$ such that the following holds for all $q\ge 2$
\begin{align} \label{lim:gamma->0:nu^gamma-nu^0:heuristic}
\W\big(\nu^\gamma,\nu^0\big) = O\Big(\big(\log |\log \gamma|\big)^{-q}\Big),\quad \gamma\to 0.
\end{align}

\end{theorem}

The rigorous statement of the above result is presented in Theorem \ref{thm:gamma->0:Wasserstein:nu^gamma-nu^0}. We remark that the distance $\W$ appearing in Theorem \ref{thm:gamma->0:nu^gamma-nu^0:heuristic} is neither the usual total variation nor a weighted variation. While they are typically the case for finite dimensional systems \cite{mattingly2002ergodicity}, solutions starting from distinct initial data in infinite dimensional dynamics tend to be singular with each other \cite{hairer2006ergodicity,hairer2008spectral,
hairer2011theory,hairer2011asymptotic}, requiring a modification of the measuring Wasserstein metric. See Section \ref{sec:main-result:gamma->0} for a precise description of $\W$. 

Finally, as a consequence of Theorem \ref{thm:ergodicity:CGL:heuristic} and Theorem \ref{thm:gamma->0:nu^gamma-nu^0:heuristic}, we obtain the validity of the inviscid limit on the infinite time horizon.

\begin{theorem} \label{thm:gamma->0:[0,infty):heuristic}
Let $\ug(t;u_0)$ and $u(t;u_0)$ respectively be the solutions of \eqref{eqn:Ginzburg_Landau:original} and \eqref{eqn:Schrodinger:original} with the same initial data $u_0$. Then, for all $q\ge 2$, $u_0$ with sufficient spatial regularity, and suitable observable $f$, the following holds
\begin{align} \label{lim:gamma->0:[0,infty):heuristic}
\sup_{t\ge 0}\big|\E f\big(\ug(t;u_0)\big) -\E f \big(u(t;u_0)\big) \big| \le \frac{C}{(\log |\log \gamma|)^q},\quad \gamma\to 0,
\end{align}
for some positive constant $C$ independent $t$ and $\gamma$.
\end{theorem}

We refer the reader to Theorem \ref{thm:gamma->0:phi} for a rigorous statement of the estimate \eqref{lim:gamma->0:[0,infty):heuristic}.

\subsection{Previous related literature} \label{sec:intro:literature} Before diving into the methodology employed for proving the main theorems, we briefly review the literature on asymptotic behaviors of the CGL and the NLS equations. Concerning deterministic settings, that is when $Q\equiv 0$, the inviscid limit on finite time windows was investigated as early as in the work of \cite{wu1998inviscid}. It was also studied in a variety of related systems \cite{wang2004inviscid,bechouche2000inviscid,huang2008inviscid,
machihara2003inviscid,
wang2002limit} justifying the approximation of the CGL equation by the NLS equation. With regard to the impact of stochastic forcing, as mentioned above, the stationary solutions of \eqref{eqn:Ginzburg-Landau:original:gamma.QdW(t)} are known to weakly converge to those of \eqref{eqn:Schrodinger:deterministic} \cite{kuksin2004randomly}. The argument relies on a tightness property, which in turn is deduced from the uniformity of suitable energy estimates with respect to $\gamma$. In a subsequent work \cite{shirikyan2011local}, these processes are shown to satisfy certain smoothness properties in the sense that their moments in $L^p$ and $H^1_0$ are absolutely continuous with respect to Lebesgue measures. Analogous results for more generalized settings of \eqref{eqn:Ginzburg-Landau:original:gamma.QdW(t)}-\eqref{eqn:Schrodinger:deterministic} were established in \cite{kuksin2013weakly} making use of the same strategy from \cite{kuksin2004randomly}. More recently, in \cite{zine2022inviscid}, the inviscid limit of \eqref{eqn:Ginzburg-Landau:original:gamma.QdW(t)} was explored for the space-time white noise with Gibbs measure initial data.  

Concerning the large time asymptotic of \eqref{eqn:Ginzburg_Landau:original}, on the one hand, ergodicity was studied in \cite{barton2004invariant} with the hypothesis that noise is invertible. While the existence of an invariant measure was obtained via the construction of stationary solutions, the uniqueness was a consequence of an irreducibility together with the strong Feller property. On the other hand, in the case of degenerate additive noise, \eqref{eqn:Ginzburg_Landau:original} was demonstrated to possess geometric ergodicity in \cite{hairer2002exponential} via the method of asymptotic coupling, which had previously been employed in \cite{boulvard2023controllability,bricmont2002exponential,weinan2001gibbsian,
kuksin2002coupling,
kuksin2000stochastic,
kuksin2001coupling,
kuksin2002couplingb,
mattingly2002exponential,
nersesyan2008polynomial,nersesyan2024exponential,
shirikyan2004exponential}. In the same spirit of \cite{hairer2002exponential} with sufficiently rich random perturbations, exponential mixing was established for the CGL equation in \cite{odasso2006ergodicity} under a variety of assumptions on the nonlinearity structures. It is important to point out that in \cite{hairer2002exponential,odasso2006ergodicity}, the exponential convergence rate is valid thanks to the advantage of strong dissipation, i.e., when $\gamma>0$. In particular, this allows for leveraging the exponential Martingale inequality as well as deriving a path-wise Foias-Prodi estimate, showing that when the low frequencies of the solutions agree, the high frequencies must be close too. In contrast, the work of \cite{debussche2005ergodicity} studied the NLS equation \eqref{eqn:Schrodinger:original} and proved that it only admitted a polynomial rate of any order, owing to the unavailability of critical ingredients that are otherwise present in the CGL equation. It is worthwhile to mention that for \eqref{eqn:Schrodinger:original}, the damping constant $\alpha>0$ is necessary so as to ensure at least the existence of statistically steady states. It turns out that $\alpha$'s positivity is also crucial to establish a mixing rate. More specifically, it was exploited in the approach of \cite{debussche2005ergodicity} together with the coupling argument developed in \cite{odasso2006ergodicity} to produce a Foias-Prodi estimate in expectation for \eqref{eqn:Schrodinger:original}. In turn, this is sufficient to dominate the nonlinearity to the extent that one can still extract a polynomial rate. See also related work in \cite{brzezniak2024global,brzezniak2023ergodic,brzezniak2023invariant}. We remark that in comparison with the more popular geometric ergodicity in the literature of SPDEs, it is more difficult to handle sub-exponential situations due to the disadvantage of weak dissipation. This is found to be the case for a conservation law with multiplicative noise \cite{dong2023ergodicity}, the stochastic wave equations \cite{gao2024polynomialWAVE, nguyen2024polynomial}, the Navier-Stokes equation in $\rbb^2$ \cite{nersesyan2024polynomial} and the Kuramoto-Sivashinsky equation in $\rbb$ \cite{gao2024polynomial}.

\subsection{Methodology of the proofs} \label{sec:intro:method-of-proof}
Turning back to Theorem \ref{thm:ergodicity:CGL:heuristic}, as $\gamma$ is close to zero, the dissipation nature is essentially inactive and does not help deducing a mixing property. To circumvent the problem, we shift our analysis to focusing on the constant $\alpha$ as employed in the framework of \cite{debussche2005ergodicity} dealing with the same issue for the NLS equation \eqref{eqn:Schrodinger:original}. More specifically, the first ingredient of the proof of \eqref{ineq:ergodicity:CGL:heuristic} is the existence of a Lyapunov function, cf. Lemma \ref{lem:moment:H_1}, proving the returning time to the origin is exponentially fast. This can be constructed by appealing to the damping effect $\alpha>0$ while performing a priori estimates. We note that the Lyapunov condition is quite popular and can be found in many works for SPDEs \cite{boritchev2013sharp,butkovsky2014subgeometric,
butkovsky2020generalized}. In literature, the second ingredient is typically an asymptotic strong Feller property \cite{hairer2008spectral,hairer2011theory} showing a large time smoothing behavior of the dynamics, resulting in a contraction property for the associated Markov semigroup. This relies on delicate bounds on the derivatives of the solutions with respect to both the initial conditions and the noise directions. Unfortunately, in case of equation \eqref{eqn:Ginzburg_Landau:original}, this approach is not applicable owing to the complication from estimating the cubic nonlinearity in higher regularity. To overcome the issue, we resort to the coupling technique developed in \cite{debussche2005ergodicity,
mattingly2002exponential,odasso2006ergodicity} by carefully facilitating the Foias-Prodi structure of \eqref{eqn:Ginzburg_Landau:original}, cf. Lemma \ref{lem:Foias-Prodi:Ginzburg-Landau}. In particular, we demonstrate that starting from distinct initial data, the solutions can be coupled in a way that a certain number of low modes agree with each other whereas the total energy can be effectively controlled. It is important to point out that the chosen number of low frequencies is dictated by the number of directions of the phase space that are directly excited by the stochastic forcing. On the one hand, this coupling behavior can be shown to occur for a long time with very high probability, cf. Proposition \ref{prop:ergodicity:P(ell(k+1)=k+1|ell(k)=infinity)}. On the other hand, the likelihood that they decouple immediately is established to be low with a power law order, cf. Proposition \ref{prop:ergodicity:P(ell(k+1)=l|ell(k)=l)}, which ultimately produces the polynomial mixing rate. We note that the restriction on the one dimensional setting stems from the fact that the stochastic Foias-Prodi estimate is actually not available in higher dimensions, cf. Remark \ref{rem:dimension-one}. As an analytic trade-off, we are able to leverage several Sobolev inequalities pertaining to dimension one, and successfully derive the coupling argument to ultimately obtain \eqref{ineq:ergodicity:CGL:heuristic}. It is worth mentioning that the approach of \cite{odasso2006ergodicity} for \eqref{eqn:Ginzburg_Landau:original} while making use of $\gamma>0$ implies that the randomness does depend on $\gamma$. See Remark \ref{rem:Foias-Prodi} for a further discussion of this point. On the other hand, the Foias-Prodi argument and Lyapunov bounds in \cite{debussche2005ergodicity} were derived for \eqref{eqn:Schrodinger:original}, which is \eqref{eqn:Ginzburg_Landau:original} with $\gamma=0$. In the present article, we are able to demonstrate that provided the noise structure is sufficiently non degenerate depending only on the damping $\alpha>0$, all the delicate estimates are carefully performed so as to keep track of the uniformity with respect to sufficiently small $\gamma$. This allows us to establish the quantitative result \eqref{ineq:ergodicity:CGL:heuristic} in the regime of vanishing viscosity. We also note that in literature, the uniform mixing rate is typically explored in the settings that possess geometric ergodicity, e.g., the heat equation \cite{glatt2022short}, the Navier-Stokes equation \cite{glatt2021mixing}, and the wave equation \cite{nguyen2023small}. To the best of the author's knowledge, \eqref{ineq:ergodicity:CGL:heuristic} seems to be the first to establish a non-exponential rate independent of the singular parameter limit of interest. We refer the reader to Section \ref{sec:main-result:polynomial-mixing} for the rigorous statement of \eqref{ineq:ergodicity:CGL:heuristic}, and to Section \ref{sec:ergodicity:Ginzburg-Landau} where its proof is supplied.

The second main result of the paper, as captured in Theorem \ref{thm:gamma->0:nu^gamma-nu^0:heuristic}, concerns the approximation of the invariant measure $\nu^\gamma$ by $\nu^0$ in the vanishing viscosity regime. This result can be regarded as an analogue of the convergence of stationary solutions in the settings of \eqref{eqn:Ginzburg-Landau:original:gamma.QdW(t)}-\eqref{eqn:Schrodinger:deterministic} \cite{kuksin2004randomly,zine2022inviscid}. We remark that in literature, there are situations where the stationary solutions can be explicitly computed. More surprisingly, they turn out to coincide with that of the targeting equation, resulting in limit \eqref{lim:gamma->0:nu^gamma-nu^0:heuristic} becoming a trivial identity. To name a few examples, we refer the reader to the work of \cite{cerrai2006smoluchowski} on a gradient system and to \cite{nguyen2018small} on a generalized Langevin equation. In general, such a phenomenon is rare and particularly is not expected for \eqref{eqn:Ginzburg_Landau:original}-\eqref{eqn:Schrodinger:original}. Nevertheless, motivated by recent works in \cite{cerrai2020convergence,cerrai2023small,
foldes2017asymptotic,
foldes2015ergodic,foldes2016ergodicity,
foldes2019large,kuksin2012mathematics} studying different asymptotic behaviors of statistically steady states, the argument for \eqref{lim:gamma->0:nu^gamma-nu^0:heuristic} can be reduced to proving the convergence of $\ug(t)$ toward $u(t)$ on any finite time window. In other words, it holds that
\begin{align*}
\W\big(\nu^\gamma,\nu^0\big)\lesssim \E\|\ug(t)-u(t)\|_{L^2(0,1)} + \W\big(\textup{Law}(u(t)),\nu^0\big),
\end{align*} 
where $\ug(0)=u(0)\sim \nu^\gamma$ and Law$(X)$ denotes the distribution of the random variable $X$. We note that by invoking uniform energy estimates, the first term on the above right hand side can be controlled by $e^{ct}|\log \gamma|^{-1}$, cf. Proposition \ref{prop:gamma->0:|u^gamma-u|:[0,T]}, whereas in accordance with the mixing property of equation \eqref{eqn:Schrodinger:original}, the second term has the order of $(1+t)^{-q}$, cf. Theorem \ref{thm:ergodicity:Schrodinger} and Lemma \ref{lem:ergodicity:Schrodinger}. Combining the two estimates allows for deducing the limit \eqref{lim:gamma->0:nu^gamma-nu^0:heuristic} by optimizing the powers on $t$. In comparison with previous work of singular parameter limits for dispersive settings, e.g., the convergence of \eqref{eqn:Ginzburg-Landau:original:gamma.QdW(t)} toward \eqref{eqn:Schrodinger:deterministic} \cite{kuksin2004randomly,kuksin2013weakly} and the small mass limit for the stochastic wave equations \cite{cerrai2020convergence,nguyen2023small}, the main novelty and challenge of \eqref{lim:gamma->0:nu^gamma-nu^0:heuristic} is the appearance of the quantitative rate in term of $\gamma$. The precise statement of Theorem \ref{thm:gamma->0:nu^gamma-nu^0:heuristic} is given in Section \ref{sec:main-result:gamma->0}, whose detailed argument will be carried out in Section \ref{sec:gamma->0:nu^gamma->nu^0}.

Lastly, making use of  the results from \eqref{ineq:ergodicity:CGL:heuristic} and \eqref{lim:gamma->0:nu^gamma-nu^0:heuristic}, we obtain, via Theorem \ref{thm:gamma->0:[0,infty):heuristic}, the global in time validity of  the solutions of \eqref{eqn:Schrodinger:original} as an approximation for those of \eqref{eqn:Ginzburg_Landau:original} in the inviscid regime. Indeed, we invoke the fact that $\W$ can be related to suitable observables by virtue of the dual Kantorovich formula \cite{villani2008optimal} and proceed in a similar fashion as in the proof of \eqref{lim:gamma->0:nu^gamma-nu^0:heuristic}, namely,
\begin{align*}
&\big|\E f\big(\ug(t;u_0)\big) -\E f \big(u(t;u_0)\big) \big|\\
& \lesssim \E\|\ug(t)-u(t)\|_{L^2(0,1)} +\W\big(\textup{Law}(\ug(t)),\nu^\gamma\big)+ \W\big(\textup{Law}(u(t)),\nu^0\big).
\end{align*}
While the first and the last terms are bounded as mentioned in the previous paragraph, the second term is a consequence of \eqref{ineq:ergodicity:CGL:heuristic}, cf. Lemma \ref{lem:ergodicity:Ginzburg-Landau:W_(d_1)}. Altogether, one can once again optimize the powers on time $t$ so as to produce \eqref{lim:gamma->0:[0,infty):heuristic}. We note that although limit \eqref{lim:gamma->0:[0,infty):heuristic} might give the impression of a Lipschitz condition on test functions, it can actually be applied to a large class of observables, including functions with polynomial growth. We refer the reader to Remark \ref{rem:phi} for a further discussion of this point. Finally, for the sake of simplicity, we restrict our attention to the focusing settings, but expect that analogous results should also hold for the defocusing case, i.e., when the term $\i |u|^2u$ in \eqref{eqn:Ginzburg_Landau:original}-\eqref{eqn:Schrodinger:original} is replaced by $-\i |u|^2u$ . See Remark \ref{rem:Foias-Prodi}, part 4, for a further discussion of this point. In Section \ref{sec:main-result:gamma->0}, we provide the rigorous statement of \eqref{lim:gamma->0:[0,infty):heuristic} through Theorem \ref{thm:gamma->0:phi} whereas its detailed proof will be supplied in Section \ref{sec:gamma->0:phi}.

The rest of the paper is organized as follows: in Section~\ref{sec:results}, we introduce all the functional settings as well as the main assumptions on the noise structures. We also formulate our results in this section, including Theorem~\ref{thm:ergodicity:Ginzburg_Landau} on the uniform polynomial mixing for the CGL equation \eqref{eqn:Ginzburg_Landau:original}, and the main results concerning the inviscid limit stated in Theorem~\ref{thm:gamma->0:Wasserstein:nu^gamma-nu^0} established for invariant measures and Theorem~\ref{thm:gamma->0:phi} for the infinite time horizon. In Section~\ref{sec:moment}, we perform a priori moment bounds on the solutions of~\eqref{eqn:Ginzburg_Landau:original} that will be employed to prove the main results. We then discuss the asymptotic coupling and prove the polynomial ergodicity in Section~\ref{sec:ergodicity:Ginzburg-Landau}. In Section~\ref{sec:gamma->0}, we provide the detailed proofs of the inviscid limits making use of the previous sections. In Appendix~\ref{sec:Schrodinger}, we briefly review several auxiliary estimates on the NLS equation \eqref{eqn:Schrodinger:original} that were employed to prove the inviscid results. In Appendix \ref{sec:auxiliary-results}, we collect an irreducibility condition for a linear version of the NLS equation \eqref{eqn:Schrodinger:original}, that was invoked to deduce the mixing property presented in Section \ref{sec:ergodicity:Ginzburg-Landau}.

\section{Assumptions and main results} \label{sec:results}
\subsection{Functional settings and assumptions} Considering the bounded interval $[0,1]$, we denote by $L^2(0,1)$, $H^1_0(0,1)$ and $H^2(0,1)$ the usual Sobolev spaces of real-valued functions on $(0,1)$. Let $A$ denote the negative Dirichlet Laplacian operator $-\triangle$ in $L^2(0,1)$ endowed with the Dirichlet boundary condition and the domain $\text{Dom}(A)=H^1_0(0,1)\cap H^2(0,1)$. From now on, we will fix an orthonormal basis $\{e_k\}_{k\ge 1}$ in $L^2(0,1)$ that diagonalizes $A$, i.e.,
\begin{align} \label{cond:Ae_k=alpha_k.e_k}
Ae_k=\alpha_k e_k,\quad k\ge 1,
\end{align}
where
\begin{align} \label{form:e_k}
e_k=\sqrt{2}\sin(k\pi x),\quad x\in[0,1],\quad \text{and}\quad \alpha_k = (k\pi)^2,\quad k\ge 1.
\end{align}

Next, we denote by $H$ the complexified Sobolev space  of $L^2(0,1)$. More specifically, for $u=u_1+\i u_2$, $v=v_1+\i v_2$, $u_i,v_i\in L^2(0,1), i=1,2$, let $\la u,v\ra_H$ denote the inner product given by 
\begin{align*}
\la u,v\ra_H = \int_0^1 (u_1+\i u_2)  (v_1+\i v_2)\d x ,
\end{align*}
and the norm is defined as
\begin{align*}
\|u\|^2_H= \la u,\ubar\ra_H = \|u_1\|^2_{L^2(0,1)} + \|u_2\|^2_{L^2(0,1)}.
\end{align*}
More generally, for each $r\in\rbb$, we denote by $H^r$ the complex domain of $A^{r/2}$ together with the inner product defined as
\begin{align*}
\la u,v\ra_{H^r} = \sum_{k\ge 1}\alpha_k^r\la u,e_k\ra_H \la v,e_k\ra_H.
\end{align*}
Particularly, the corresponding norm is given by
\begin{align*}
\|u\|^2_{H^r}= \la u,\ubar\ra_{H^r}=\sum_{k\ge 1}\alpha_k^r|\la u,e_k\ra_H|^2.
\end{align*}
In addition to $H^r$, we will work with the complex $L^p$ spaces, $p\ge 1$ namely
\begin{align*}
L^p=\Big\{u:[0,1]\to\cbb\big|\|u\|_{L^p} := \int_0^1 |u|^p\d x<\infty\Big\}.
\end{align*}

Next, for each integer $N\ge 1$, we let $P_N$ be the projection of $u\in H$ on span$\{e_1,\dots,e_N\}$, i.e.,
\begin{align*}
P_N u = \sum_{k=1}^N \la u,e_k\ra_H e_k.
\end{align*}
The complement of $P_N$ is denoted as $Q_N=I-P_N$, i.e.,
\begin{align*}
Q_Nu = \sum_{k\ge N+1}^\infty\la u,e_k\ra_H e_k.
\end{align*}

We may now recast the Ginzburg-Landau equation \eqref{eqn:Ginzburg_Landau:original} as 
\begin{align} \label{eqn:Ginzburg_Landau}
\d u^\gamma(t) & = -(\gamma +\i) A u^\gamma(t)\d t + \i |u^\gamma(t)|^2u^\gamma (t)\d t -\alpha u^\gamma(t)\d t +Q\d W(t), \notag \\
\ug(0)&=u_0\in H.
\end{align}
Regarding the noise term in \eqref{eqn:Ginzburg_Landau}, we assume that $W(t)$ is a cylindrical Wiener process on $H$, whose decomposition is given by
\begin{align*}
W(t)=\sum_{k\ge 1} e_k\big(B_k^1(t)+\i B_k^2(t)\big),
\end{align*}
where $\{(B_k^1,B_k^2)\}_{k\ge 1}$ is a sequence of i.i.d two-dimensional Brownian motions, each defined on the same stochastic basis $\mathcal{S}=(\Omega, \Fcal,\{\Fcal_t\})$ \cite{karatzas2012brownian}. Concerning the operator $Q$, we impose the following assumption \cite{barton2004invariant,
debussche2005ergodicity,hairer2002exponential,
odasso2006ergodicity}

\begin{assumption} \label{cond:Q} $Q:H\to H$ satisfies
\begin{align} \label{cond:Q:Qe_k=lambda_k.e_k}
Qe_k=\lambda_k e_k,\, k=1,2,\dots,N,\quad\text{and}\quad Qe_k=0,\,k\ge N+1.
\end{align}
for some integer $N\ge 1$ and positive constants $\lambda_k$, $k=1,\dots,N$.
Furthermore, there exists a positive constant $C_Q$ independent of $N$ such that
\begin{align} \label{cond:Q:Tr(A^(3/2)QQ)<infinity}
\emph{Tr}(A^{\frac{3}{2}}QQ^*)= \sum_{k=1}^N \alpha_k^{\frac{3}{2}}\lambda_k^2 < C_Q,
\end{align}
where $\{\alpha_k\}_{k\ge 1}$ is the sequence of eigenvalues of $A$ as in \eqref{cond:Ae_k=alpha_k.e_k}-\eqref{form:e_k}.

\end{assumption}

In the above, we recall the definition $\Tr(g)=\sum_{k\ge 1}\la ge_k,e_k\ra_H$ for $g\in L(H)$.

\begin{remark} \label{rem:Q} 1. We remark that condition \eqref{cond:Q:Qe_k=lambda_k.e_k} states that noise is directly excited only in the first $N$ directions of the phase space, resulting in $Q$ being invertible on span$\{e_1,\dots,e_N\}.$ As shown below in Section \ref{sec:ergodicity:Ginzburg-Landau}, cf. Proposition \ref{prop:ergodicity:P(ell(k+1)=l|ell(k)=l)}, $N$ will be chosen to be sufficiently large so as to ensure the validity of the Foias-Prodi estimate. On the other hand, condition \eqref{cond:Q:Tr(A^(3/2)QQ)<infinity} requires that the noise's regularity is at least $H^3$ regardless of the size of $N$. It will be exploited to establish an irreducible condition for \eqref{eqn:Ginzburg_Landau} as well as energy estimates that do not depend on $N$. In turn, all of these properties will be employed to conclude the uniformity of the polynomial mixing rate  with respect to the inviscid parameter $\gamma$.

2. We also remark that in this work, we opt for the choice of $Q$ being diagonalized by the same basis $\{e_k\}_{k\ge 1}$, so as to simplify computations \cite{debussche2005ergodicity,hairer2002exponential}. Analogous results should hold for more general additive noise structures as long as $Q$ is invertible on span$\{e_1,\dots,e_N\}$ and that condition \eqref{cond:Q:Tr(A^(3/2)QQ)<infinity} is satisfied.
\end{remark}

Under Assumption \ref{cond:Q}, the existence and uniqueness of mild solutions in $H$ of \eqref{eqn:Ginzburg_Landau} can be derived using the Galerkin approximation by exploiting the Lyapunov structures in Section \ref{sec:moment}. The argument is classical and can be found in many previous works for SPDE, e.g., \cite{albeverio2008spde,
glatt2008stochastic}. See also \cite{barton2004global,barton2004invariant}. 

As a consequence of the well-posedness, for each $\gamma>0$, we can thus introduce the Markov transition probabilities of the solutions $\ug(t;u_0)$ by
\begin{align*}
P_t^\gamma(u_0,A):= \P(\ug(t;u_0)\in A),
\end{align*}
which are well-defined for $t\ge 0$, initial states $u_0\in H$ and Borel sets $A\subset H$. Letting $\mathcal{B}_B(H)$ denote the set of bounded Borel measurable functions $\f:H\to \rbb$, the associated Markov semigroup $P_t^\gamma:\mathcal{B}_b(H)\to \mathcal{B}_b(H)$ is defined and denoted by
\begin{align*}
P_t^\gamma \f(u_0)=\E\big[\f\big(\ug(t;u_0\big)\big],\quad \f\in \mathcal{B}_b(H).
\end{align*}
Let $\Pcal r(H)$ be the space of probability measures in $H$. For each $\nu\in \Pcal r(H)$, we denote by $P_t^\gamma\nu$ the measure given by
\begin{align*}
P_t^\gamma\nu(A)=\int_H P_t^\gamma(u_0,A)\nu(\d u_0).
\end{align*}
Recall that a probability measure $\nu \in \Pcal r(H)$ is said to be invariant for $P_t^\gamma$ if 
\begin{align*}
P_t^\gamma \nu=\nu.
\end{align*}
Alternatively, this is equivalent to
\begin{align*}
\int_H \f(u) P_t^\gamma\nu(\d u) = \int_H \f(u)\nu(\d u), \quad \f\in \mathcal{B}_b(H).
\end{align*}

\subsection{Uniform polynomial mixing} \label{sec:main-result:polynomial-mixing}

We now turn to the topic of uniform mixing rate for \eqref{eqn:Ginzburg_Landau}. Following the framework of~\cite{butkovsky2014subgeometric, butkovsky2020generalized, 
hairer2008spectral,
hairer2011asymptotic}, we recall that a function $d:H^k\to[0,\infty)$ is called \emph{distance-like} if it is symmetric, lower semi-continuous, and $d(u,v)=0$ if and only if $ u=v$; see \cite[Definition 4.3]{hairer2011asymptotic}. Let $\W_{d}$ denote the corresponding coupling distance or Wasserstein-type distance in $\Pcal r(H^k)$ associated with $d$, given by
\begin{equation} \label{form:W_d}
\W_{d}(\nu_1,\nu_2) := \inf \E\, d(X,Y),
\end{equation}
where the infimum is taken over all pairs of random variables $(X,Y)$ such that $X\sim \nu_1$ and $Y\sim\nu_2$. When $d$ is a distance in $H^k$, by the dual Kantorovich Theorem, $\W_d$ is equivalently defined as \cite[Theorem 5.10]{villani2008optimal}
\begin{align} \label{form:W_d:dual-Kantorovich}
\W_{d}(\nu_1,\nu_2)=\sup_{[f]_{\text{Lip},d}\leq1}\Big|\int_{\H}f(u)\nu_1(\d u)-\int_{\H}f(u)\nu_2(\d u)\Big|,
\end{align}
where
\begin{align} \label{form:Lipschitz}
[f]_{\text{Lip},{d}}=\sup_{u\neq v}\frac{|f(u)-f(v)|}{d(u,v)}.
\end{align}
On the other hand, if $d$ is only distance-like, then the following one-sided inequality holds
\begin{equation} \label{ineq:W_d(nu_1,nu_2):dual}
\W_{d}(\nu_1,\nu_2) \ge \sup_{[f]_{\text{Lip},{d}}\le 1}\Big|\int_H f(U)\nu_1(\d U)-\int_H f(U)\nu_2(\d U)\Big|.
\end{equation}
We refer the reader to \cite[Proposition A.3]{glatt2021mixing} for a further discussion of this point.

In our settings, we will particularly pay attention to the following two distances in $H$: the former is  the discrete metric, i.e.,
$d(u,v)=1$ when $u\neq v$ and $d(u,v)=0$ otherwise. The corresponding $\W_d$ is the usual total variation distance, denoted by $\W_{\TV}$. The latter is the distance $d_{k}$, $k=0,1$, given by
\begin{equation} \label{form:d_k}
d_k(u,v):=\|u-v\|_{H^k}\mi 1,\quad u,v\in H^k.
\end{equation}
Particularly, the corresponding $W_{d_1}$ (with $k=1$) will be used to measure the convergence rate of \eqref{eqn:Ginzburg_Landau} toward equilibrium. 

In order to tackle the inviscid regime $\gamma\to 0$, it is essential to establish suitable a priori bounds on the solution $\ug(t)$. To this end, we recall the Gagliardo-Nirenberg inequality 
\begin{align} \label{ineq:Gagliardo-Nirenberg}
\|u\|^4_{L^4} \le \frac{1}{4}\|u\|^2_{H^1}+\frac{1}{2}\kappa\|u\|^6_H, \quad u \in H^1,
\end{align}
holds for some positive constant $\kappa$. We introduce the function $\Psi:H^1\to\rbb$ defined as
\begin{align} \label{form:Psi}
\Psi(u)= \|u\|^2_{H^1}-\frac{1}{2}\|u\|^4_{L^4}+\kappa\|u\|_{H}^6.
\end{align}
It is well-known that $\Psi$ is the conserved Hamiltonian for the focusing NLS equation in the absence of noise and damping, i.e., $Q\equiv 0$ and $\alpha=0$ \cite{debussche2005ergodicity,temam2012infinite}. The auxiliary results concerning $\Psi$ and the NLS equation are collected in Appendix \ref{sec:Schrodinger} and will be employed to study the approximation of \eqref{eqn:Ginzburg_Landau} by the NLS equation. Next, let $\Phi$ be the functional given by
\begin{align} \label{form:Phi}
\Phi(u)=\Psi(u)+\kappa \|u\|^{18}_H= \|u\|^2_{H^1}-\frac{1}{2}\|u\|^4_{L^4}+\kappa\|u\|_{H}^6+\kappa \|u\|^{18}_H.
\end{align}
In view of \eqref{ineq:Gagliardo-Nirenberg}, we note that $\Phi$ satisfies the lower bound
\begin{align} \label{cond:Phi}
\Phi(u)\ge \frac{3}{4}\|u\|^2_{H^1}+\frac{1}{2}\|u\|^4_{L^4}+\frac{1}{2}\kappa\|u\|_{H}^6+\kappa \|u\|^{18}_H,
\end{align}
and that
\begin{align*}
 \Phi\ge \Psi\ge \frac{1}{2}\|u\|^2_{H^1}.
\end{align*}
The former will be useful to establish energy estimates for \eqref{eqn:Ginzburg_Landau} in Section \ref{sec:moment} whereas the latter will be invoked to investigate the inviscid limit $\gamma\to 0$ in Section \ref{sec:gamma->0}.

Letting $N$ be the constant as in Assumption \ref{cond:Q}, we recall from \cite{odasso2006ergodicity} that for fixed $\gamma>0$ and for all $N$ sufficiently large possibly depending on $\gamma$, it can be shown that $P_t^\gamma$ admits a unique invariant probability measure $\nu^\gamma$ and that $P_t^\gamma$ is exponentially attractive toward $\nu^\gamma$ with respect to the Wasserstein distance $\W_{d_0}$ \cite[Theorem 2.1]{odasso2006ergodicity}.

We now state our first main result, concerning the unique ergodicity of $P_t^\gamma$ as well as the polynomial mixing rate independent of the parameter $\gamma$.

\begin{theorem} \label{thm:ergodicity:Ginzburg_Landau} 1. Let $N$ be the parameter as in Assumption \ref{cond:Q}. Then, there exists a positive integer $N_1$ sufficiently large such that for all $N\ge N_1$,  $P^\gamma_t$ admits a unique invariant probability measure $\nu^\gamma$ in $H$.

2. Furthermore, let $\f:H^1\to\rbb$ be such that $\|\f\|_{\textup{Lip},{d_1}}<\infty$ where $d_1$ is as in \eqref{form:d_k} and $\|\f\|_{\textup{Lip},{d_1}}$ is defined in \eqref{form:Lipschitz}. Then, the following holds for all $q\ge 2$
\begin{align} \label{ineq:ergodicity:Ginzburg_landau}
\limsup_{\gamma\to 0}\sup_{u_1,u_2\in H^1}\frac{\big| \E \f\big(\ug(t;u_1)\big) - \E \f\big(\ug(t;u_2)\big) \big| }{1+\Phi(u_1)+\Phi(u_2)} \le C_q(1+t)^{-q}\|\f\|_{\textup{Lip},{d_1}},\quad t\ge 0,
\end{align}
for some positive constant $C_q$ independent of $u_1,u_2,t,\gamma$ and $\f$. In the above, $\Phi$ is defined in \eqref{form:Phi}.

\end{theorem}

We note that Theorem \ref{thm:ergodicity:Ginzburg_Landau}, part 1, is not entirely new. A similar unique ergodicity result was established in \cite[Theorem 2.1]{odasso2006ergodicity}, albeit for an equation in the presence of a coercive term $-\eta |u|^2u$ ($\eta>0$) and without linear damping, i.e., with $\alpha=0$ on the right-hand side of \eqref{eqn:Schrodinger}. Although part 1 of Theorem \ref{thm:ergodicity:Ginzburg_Landau} does not follow directly from \cite[Theorem 2.1]{odasso2006ergodicity}, for each $\gamma>0$, the linear damping term $(\alpha>0)$ provides sufficient dissipation to implement the coupling argument developed in \cite{odasso2006ergodicity}, thereby yielding Part 1 of Theorem \ref{thm:ergodicity:Ginzburg_Landau}. The main novelty of Theorem \ref{thm:ergodicity:Ginzburg_Landau} is the mixing rate \eqref{ineq:ergodicity:Ginzburg_landau} that is uniform with respect to $\gamma$ sufficiently small. It is also important to point out that \eqref{ineq:ergodicity:Ginzburg_landau} does not hold for any initial data in $H$, owing to the fact that the estimates in Section \ref{sec:ergodicity:Ginzburg-Landau} require sufficient regularity on the solutions whose moment bounds are independent of $\gamma$. In contrast, while the exponential mixing result in \cite{odasso2006ergodicity} can be applied to any $u_1,u_2\in H$, it imposes a correlation between $N$ and $\gamma$. In particular, \cite[Theorem 2.1]{odasso2006ergodicity} implicitly implies that as $\gamma\to 0$, the parameter $N$ must tend to infinity. See Remark \ref{rem:Foias-Prodi} below.

The proof of Theorem \ref{thm:ergodicity:Ginzburg_Landau} makes use of the coupling framework in \cite{debussche2005ergodicity, mattingly2002exponential,odasso2006ergodicity} tailored to the Ginzburg-Landau equation \eqref{eqn:Ginzburg_Landau}. The technique relies on three crucial ingredients: a stochastic version of the classical Foias-Prodi argument  allowing for coupling the solution trajectories, the high probability of coupling inside a ball and the small probability of decoupling. Together with suitable moment bounds, we are able to deduce the uniformity of a polynomial rate for every $\gamma\in(0,1)$. All of this will be carried out in Section \ref{sec:ergodicity:Ginzburg-Landau}. In turn, the result of Theorem \ref{thm:ergodicity:Ginzburg_Landau} will be exploited to study the inviscid regime, which we describe next.

\subsection{Inviscid limit $\gamma\to 0$} \label{sec:main-result:gamma->0} Having established the uniform mixing rate for \eqref{eqn:Ginzburg_Landau}, we turn to the main topic of the paper concerning the limit $\gamma\to 0$. We first recast the Schr\"odinger equation \eqref{eqn:Schrodinger:original} as
\begin{align} \label{eqn:Schrodinger}
\d u(t) & = -\i A u(t)\d t + \i |u(t)|^2u (t)\d t -\alpha u(t)\d t +Q\d W(t),\quad u(0)=u_0\in H^1.
\end{align}
Similar to \eqref{eqn:Ginzburg_Landau}, under Assumption \ref{cond:Q}, it is not difficult to derive the well-posedness of \eqref{eqn:Schrodinger}. Indeed, the local solutions of \eqref{eqn:Schrodinger} in $H^1$ follow from the fact that the semigroup $S(t)=e^{-\i A t-\alpha t}$ is contracting in $H^1$ \cite{debussche2005ergodicity}. Then, the argument can be extended to deduce the global well-posedness by appealing to suitable a priori estimates supplied in Lemma \ref{lem:moment:H_1:Schrodinger}. As a result, the corresponding Markov semigroup $P_t^0$ is well-defined as in the case of $P_t^\gamma$. In the work of \cite{debussche2005ergodicity}, it can be shown that $P_t^0$ admits a unique invariant probability measure $\nu^0$ in $\Pcal r(H^1)$, and that the convergence rate is polynomial of any order, owing to the lack of strong dissipation. For the convenience of the reader, the precise statement is presented in Theorem \ref{thm:ergodicity:Schrodinger}, whose proof can be found in \cite{debussche2005ergodicity}. 

Our next goal is to rigorously describe in what sense we are comparing $\nu^\gamma$ obtained from Theorem \ref{thm:ergodicity:Ginzburg_Landau} with $\nu^0$. To this end, following the framework of \cite{ glatt2022short,glatt2021mixing,hairer2008spectral,hairer2011theory,hairer2011asymptotic,
nguyen2023small}, we introduce the following distance-like function
\begin{align} \label{form:d_0^xi}
d_0^\xi(u_1,u_2)=\sqrt{d_0(u_1,u_2)\big(1+e^{\xi\|u_1\|^2_H}+e^{\xi\|u_2\|^2_H}\big)},\quad u_1,u_2\in H.
\end{align}
As mentioned in the introduction, there are singular limit regimes where the limiting measure is shown to agree with those of the approximating dynamics \cite{cerrai2006smoluchowski,nguyen2018small}. This remarkable observation stems from the fact that those measures can be explicitly computed. Such an approach is not available in our settings and in many other situations \cite{cerrai2020convergence,
cerrai2022smoluchowski,cerrai2023small,foldes2017asymptotic,foldes2015ergodic,
foldes2016ergodicity,foldes2019large, nguyen2023small}. Nevertheless, our second main result of the paper states that $\nu^\gamma$ resembles $\nu^0$ in the regime of vanishing viscosity.

\begin{theorem} \label{thm:gamma->0:Wasserstein:nu^gamma-nu^0}
Under the same hypotheses of Theorem \ref{thm:ergodicity:Ginzburg_Landau}, let $\nu^\gamma$ and $\nu^0$ be the unique invariant probability measures of \eqref{eqn:Ginzburg_Landau} and \eqref{eqn:Schrodinger}, respectively. Then, for all $q\ge 2$ and $\xi\in(0,1)$ sufficiently small, it holds that
\begin{align} \label{lim:gamma->0:Wasserstein:nu^gamma-nu^0}
\W_{d_0^\xi}(\nu^\gamma,\nu^0) \le \frac{C_{q,\xi}}{(\log|\log \gamma|)^q},\quad \text{as}\,\,\,\gamma\to 0.
\end{align}
In the above, $d_0^\xi$ is the distance-like function defined in \eqref{form:d_0^xi}, $\W_{d_0^\xi}$ is the corresponding Wasserstein distance defined in \eqref{form:W_d}, and $C_{q,\xi}$ is a positive constant independent of $\gamma$.

\end{theorem}

The proof of Theorem \ref{thm:gamma->0:Wasserstein:nu^gamma-nu^0} relies on two crucial ingredients: the approximation of \eqref{eqn:Ginzburg_Landau} by \eqref{eqn:Schrodinger} on any finite time window, cf. Proposition \ref{prop:gamma->0:|u^gamma-u|:[0,T]}, and the polynomial mixing rate of \eqref{eqn:Schrodinger} toward $\nu^0$, cf. Theorem \ref{thm:ergodicity:Schrodinger}. In particular, the former estimate draws upon the framework typically found in the settings of dispersive equations, e.g., the wave equation \cite{cerrai2006smoluchowski,cerrai2006smoluchowski2,
cerrai2022smoluchowski,cerrai2022small,
cerrai2023small}. Then, we combine with the fact that $\nu^\gamma$ satisfies uniform exponential moment bounds independent of $\gamma$ to ultimately deduce the convergence \eqref{lim:gamma->0:Wasserstein:nu^gamma-nu^0}. All of this will be clearer in Section \ref{sec:gamma->0}, where the proof of Theorem \ref{thm:gamma->0:Wasserstein:nu^gamma-nu^0} is supplied.

Finally, as a consequence of Theorem \ref{thm:ergodicity:Ginzburg_Landau} and Theorem \ref{thm:gamma->0:Wasserstein:nu^gamma-nu^0}, we obtain the following result.

\begin{theorem} \label{thm:gamma->0:phi}
Under the same hypotheses of Theorem \ref{thm:ergodicity:Ginzburg_Landau}, let $\{\ug_0\}_{\gamma\in(0,1)}$ be a sequence of deterministic initial conditions such that 
\begin{align*}
\sup_{\gamma\in(0,1)} \|\ug_0\|_{H^1} < R.
\end{align*}
Then, for all $R>0$, $q\ge 2$, $\xi\in(0,1)$ sufficiently small, and $\f:H\to\rbb$ such that $\|\f\|_{\textup{Lip},d_0^\xi}<\infty$ where $d_0^\xi$ is as in \eqref{form:d_0^xi} and $\|\f\|_{\textup{Lip},d_0^\xi}$ is defined in \eqref{form:Lipschitz}, the following holds 
\begin{align} \label{lim:gamma->0:phi}
\sup_{t\ge 0}\big|\E \f\big(\ug(t;\ug_0)\big)-\E \f\big(u(t;\ug_0)\big)\big| \le C_{R,q,\xi}\frac{\|\f\|_{\textup{Lip},d^\xi_0}}{(\log |\log \gamma|)^{q}}  ,\quad\textup{as} \,\gamma\to 0,
\end{align}
for some positive constant $C_{R,q,\xi}$ independent of $\f$, $u_0^\gamma$ and $\gamma$.

\end{theorem}

In order to prove Theorem \ref{thm:gamma->0:phi}, we will draw upon the framework of \cite{glatt2021mixing} tailored to our settings by employing the uniform polynomial rate in Theorem \ref{thm:ergodicity:Ginzburg_Landau} and the convergence of $\nu^\gamma$ toward $\nu^0$ in Theorem \ref{thm:gamma->0:Wasserstein:nu^gamma-nu^0}. We note that our argument is slightly different from those in \cite{glatt2022short,glatt2021mixing,nguyen2023small}, which typically require a generalized triangle inequality for $\W_{d_0^\xi}$. In our setting, instead of dealing directly with $d_0^\xi$, we prove an analogue of \eqref{lim:gamma->0:phi} for the distance $d_0$. Then, we will upgrade to $d_0^\xi$ making use of the fact that both $\ug(t)$ and $u(t)$ satisfy appropriate exponential moment bounds independent of time $t$ and $\gamma$. The detailed proof of Theorem \ref{thm:gamma->0:phi} will be carried out in Section \ref{sec:gamma->0}.

\begin{remark} \label{rem:phi}
 Recalling \cite[Proposition 46]{glatt2021mixing}, let $(\vbb,\|\cdot\|)$ be a Banach space and consider distance-like functions
of the form
\begin{align*}
    \ell(u,\util) = \sqrt{\Big(\frac{\|u-\util\|}{\varepsilon}\mi 1\Big)\big(1+V(u)+V(\util)\big)}
\end{align*}
where $\varepsilon>0$ and $V:\vbb\to[0,\infty)$ is convex. Given any continuously
differentiable function $\f:\vbb\to\rbb$, suppose that
\begin{align*}
\sup_{u\in \vbb} \frac{\max\{2|\f(u)|,\sqrt{\varepsilon}\|D \f(u)\|_{\vbb} \} }{\sqrt{1+V(u)}}<\infty.
\end{align*}
Then, $\f$ is Lipschitz with respect to $\ell$. As a consequence, a sufficient condition for $\f$ appearing in Theorem \ref{thm:gamma->0:phi} to satisfy $\|\f\|_{\text{Lip},d_0^\xi}<\infty$ is that
\begin{align*}
\sup_{u\in H}\frac{\max\{2|\f(u)|,\|D \f(u)\|_H\}}{\sqrt{1+e^{\xi\|u\|^2_H}}}<\infty.
\end{align*}
In particular, it is not difficult to check that the class of polynomial functions satisfies the above condition.

\end{remark}

\section{A priori moment estimates} \label{sec:moment}

Throughout the rest of the paper, $c$ and $C$ denote generic positive constants that may change from line to line. The main parameters that they depend on will appear between parenthesis, e.g., $c(T,q)$ is a function of $T$ and $q$.

In this section, we consider the solutions $\ug$ of  \eqref{eqn:Ginzburg_Landau} and establish several useful energy estimates that will be employed to prove the main results. We start with the bounds in $H$ norm in Lemma \ref{lem:moment:L^2} below. In particular, Lemma \ref{lem:moment:L^2}, part 1, will be invoked to obtain higher regularity in Lemma \ref{lem:moment:H_1} whereas the result of Lemma \ref{lem:moment:L^2}, part 2, appears in the proof of Lemma \ref{lem:moment:nu_gamma} giving the uniform exponential moment of $\nu^\gamma$.

\begin{lemma} \label{lem:moment:L^2}
1. Let $u_0\in L^2(\Omega;H)$ be given and $u^\gamma(t)$ be the solution of \eqref{eqn:Ginzburg_Landau}. Then, for all $n\ge 1$, it holds that 
\begin{align} \label{ineq:moment:|u|^n_H+int.|u|^n_H<|u_0|^n_H+Ct}
&\E\big[\|\ug(t)\|_H^{2n}\big]+\alpha n \int_0^t \E\big[\|\ug(r)\|_H^{2n}\big]\emph{d}r \notag  \\
&\le\E\big[\|u_0\|_H^{2n}\big]+2^n\Big(\frac{n-1}{\alpha}\Big)^{n-1} |\emph{Tr}(QQ^*)|^{n}t,\quad t\ge 0 ,
\end{align}
and that
\begin{align} \label{ineq:moment:|u|^n_H<e^(-ct)|u_0|^n_H}
\E\big[\|\ug(t)\|_H^{2n}\big] \le e^{-\alpha n t}\E\big[\|u_0\|_H^{2n}\big]+2^n\frac{(n-1)^{n-1}}{n \alpha^n } |\emph{Tr}(QQ^*)|^{n},\quad t\ge 0.
\end{align}

2. For all $\xi>0$ sufficiently small independent of $\gamma$, the following holds
\begin{align} \label{ineq:moment:exp(|u|^2_H)}
\E \big[ e^{ \xi \|\ug(t)\|^2_H  }\big] \le e^{-c_\xi t} \E \big[ e^{ \xi \|u_0\|^2_H  }\big]  +C_\xi ,\quad t\ge 0,
\end{align}
for some positive constants $c_\xi$ and $C_\xi$ independent of $t,\gamma$ and $u_0$.
\end{lemma}

\begin{remark} \label{rem:Tr(QQ)} In light of Assumption \ref{cond:Q}, cf. \eqref{cond:Q:Tr(A^(3/2)QQ)<infinity}, it is important to note that
\begin{align*}
\sup_{N\ge 1}\Tr(QQ^*)=\sup_{N\ge 1}\sum_{k=1}^N\lambda_k^2<C_Q,
\end{align*}
which is independent of the size of $N$.
\end{remark}

\begin{proof}[Proof of Lemma \ref{lem:moment:L^2}]
For notational convenience, throughout the proofs in this section, we will drop the superscript $\gamma$ in $\ug$.

 We first compute the derivatives of $\|u\|_H^2$ as follows.
\begin{align*}
\big\la D \big(\|u\|_H^2\big),v\big\ra_{H}&= \la u,\vbar\ra_H+\la \ubar, v\ra_{H},\\
\big\la D^2\big(\|u\|_H^2\big)(z),v\big\ra_H & =   \la z,\vbar\ra_H+\la \zbar, v\ra_H.
\end{align*}
Applying It\^o's formula to $\|u\|_H^2$ gives
\begin{align} \label{eqn:d.|u|^2_H}
\d \|u(t)\|^2_H & = -2\gamma\|A^{1/2}u(t)\|^2_H\d t-2\alpha\|u(t)\|^2_H\d t +2\Tr(QQ^*)\d t+\d M_0(t),
\end{align}
where
\begin{align} \label{form:M_0(t)}
\d M_0(t)=\la u(t),\overline {Q\d W(t)}\ra_H+\la \ubar(t),Q\d W(t)\ra_H,
\end{align}
whose quadratic variation is given by
\begin{align} \label{form:M_0(t):quadratic_variation}
\d\la M_0\ra(t) = 4\|Qu(t)\|^2_H\d t.
\end{align}
Next, for $n\ge 2$, it holds that
\begin{align} \label{eqn:d.|u|^n_H}
\d \|u(t)\|^{2n}_H & = n\|u(t)\|^{2n-2}_H\Big(-2\gamma\|A^{1/2}u(t)\|^2_H\d t-2\alpha\|u(t)\|^2_H\d t +2\Tr(QQ^*)\d t+\d M_0(t)\Big) \nt \\
&\qquad +n(2n-2)\|u(t)\|^{2n-4}_H \|Qu(t)\|^2_H.
\end{align}

1. Turning back to \eqref{ineq:moment:|u|^n_H+int.|u|^n_H<|u_0|^n_H+Ct} and \eqref{ineq:moment:|u|^n_H<e^(-ct)|u_0|^n_H}, we employ Young inequality to estimate
\begin{align*}
&2\|u(t)\|^{2n-2}_H\Tr(QQ^*)+ (2n-2)\|u(t)\|^{2n-4}_H \|Qu(t)\|^2_H\\
&\le 2 \|u(t)\|^{2n-2}_H\Tr(QQ^*)+ (2n-2)\|u(t)\|^{2n-4}_H \cdot \|Q\|_{L(H)}^2\|u(t)\|^2_H\\
&\le 2n \|u(t)\|^{2n-2}_H\Tr(QQ^*)\\
&\le  \alpha\|u\|^{2n}_H+2^n\Big(\frac{n-1}{\alpha}\Big)^{n-1} |\Tr(QQ^*)|^{n}.
\end{align*}
From \eqref{eqn:d.|u|^n_H}, we deduce the bound
\begin{align}\label{ineq:d|u|^n_H}
\d \|u(t)\|^{2n}_H & \le  -\alpha n \|u(t)\|^{2n}_H\d t +2^n\Big(\frac{n-1}{\alpha}\Big)^{n-1} |\Tr(QQ^*)|^{n}\d t + n\|u(t)\|^{2n-2}_H \d M_0(t).
\end{align}
Taking expectation on both sides yields
\begin{align*}
\frac{\d}{\d t} \E\|u(t)\|^{2n}_H & \le  -\alpha n \E \|u(t)\|^{2n}_H +2^n\Big(\frac{n-1}{\alpha}\Big)^{n-1} |\Tr(QQ^*)|^{n}.
\end{align*}
On the one hand, integrating the above estimate with respect to time $t$ produces \eqref{ineq:moment:|u|^n_H+int.|u|^n_H<|u_0|^n_H+Ct}. On the other hand, from Gronwall's inequality, we establish \eqref{ineq:moment:|u|^n_H<e^(-ct)|u_0|^n_H}, thereby finishing the proof.

2. With regard to \eqref{ineq:moment:exp(|u|^2_H)}, let $\xi>0$ be given and be chosen later. From \eqref{eqn:d.|u|^2_H}, we readily have
\begin{align} \label{ineq:d.beta|u|^2_H}
\d \xi\|u(t)\|^2_H & \le -2\alpha\xi \|u(t)\|^2_H\d t +2\xi\Tr(QQ^*)\d t+\d\xi M_0(t). 
\end{align}
Note also that the quadratic process $\la M_0\ra(t)$ given by \eqref{form:M_0(t):quadratic_variation} satisfies the bound
\begin{align*}
\d \la M_0\ra (t) \le 4\|Q\|^2_{L(H)}\|u(t)\|^2_H \d t\le 4\Tr(QQ^*)\|u(t)\|^2_H\d t.
\end{align*}
It follows that
\begin{align*}
\d  e^{\xi\|u(t)\|^2_H}& \le \xi e^{\xi\|u(t)\|^2_H}\big(\d \|u(t)\|^2_H+\frac{1}{2}\xi \d \la M_0\ra(t)\big) \\
&\le \xi e^{\xi\|u(t)\|^2_H}\big( -2\alpha \|u(t)\|^2_H\d t +2\Tr(QQ^*)\d t+\d M_0(t)+2\xi \Tr(QQ^*)\|u(t)\|^2_H\d t\big).
\end{align*}
Picking $\xi$ sufficiently small, e.g., 
\begin{align*}
\xi<\frac{\alpha}{2\Tr(QQ^*)},
\end{align*}
produces the bound in expectation
\begin{align*}
\frac{\d}{\d t}\E e^{\xi\|u(t)\|^2_H}\le  \xi \E \big[ e^{\xi\|u(t)\|^2_H}(-\alpha\|u(t)\|^2_H+2\Tr(QQ^*))\big].
\end{align*}
We employ the elementary inequality
\begin{align*}
e^{\xi x}(-\alpha x+2\Tr(QQ^*)) \le - ce^{\xi x}+C,\quad x\ge 0,
\end{align*}
to infer the existence of positive constants $c=c(\xi)$ and $C=C(\xi)$ independent of $\gamma$ such that
\begin{align*}
\frac{\d}{\d t}\E e^{\xi\|u(t)\|^2_H}\le -c \E e^{\xi\|u(t)\|^2_H}+C,\quad t\ge 0.
\end{align*}
In turn, this implies estimate \eqref{ineq:moment:exp(|u|^2_H)} by virtue of Gronwall's inequality. The proof is thus complete.

\end{proof}

Next, in Lemma \ref{lem:moment:H_1} stated and proven below, we establish moment bounds in $H^1$. The result of Lemma \ref{lem:moment:H_1}, part 1, will be particularly employed to obtain the uniform polynomial rate in Section \ref{sec:ergodicity:Ginzburg-Landau} whereas part 2 will be useful in proving the convergence of $\ug(t)$ toward $u(t)$ in Proposition \ref{prop:gamma->0:|u^gamma-u|:[0,T]}.

\begin{lemma} \label{lem:moment:H_1}
1. Let $u_0\in L^2(\Omega;H^1)$ be given and $u^\gamma(t)$ be the solution of \eqref{eqn:Ginzburg_Landau} with initial condition $u_0$. Then, for all $n\ge 1$, it holds that 
\begin{align} \label{ineq:moment:|u|_H1+int.|u|_H1<|u_0|_H1+Ct}
\E\big[ \Phi(u^\gamma(t))^n\big]+\frac{1}{2}\alpha n \int_0^t \E\big[\Phi(u^\gamma(r))^n\big] \emph{d} r \le \E \big[\Phi(u_0)^n\big]+ C_{1,n}t,\quad t\ge 0,
\end{align}
and that
\begin{align} \label{ineq:moment:|u|_H1<e^(-ct)|u_0|_H1}
\E\big[ \Phi(u^\gamma(t))^n\big] 
&\le e^{-\frac{1}{2}\alpha n t}\E\big[ \Phi(u_0)^n\big]+C_{1,n} ,\quad t\ge 0,
\end{align}
for some positive constant $C_{1,n}$ independent of $u_0,t,\gamma$ and $N$.

2. For all $T>0$, the following holds 
\begin{align} \label{ineq:moment:sup_[0,T]Phi(t)^n}
\E \big[\sup_{[0,T]}\Phi(\ug(t))^n\big] &\le \E\big[\Phi(u_0)^{n}\big]+ \E \big[\Phi(u_0)^{2n}\big] + C_{1,n} T  .
\end{align}
\end{lemma}

\begin{proof}
Recalling $\Phi$ from \eqref{form:Phi}, we first consider $\|u\|^4_{L^4}$ and compute the derivatives as follows:
\begin{align*}
\big\la D \big(\|u\|_{L^4}^4\big),v\big\ra_{H}&= 2\la |u|^2u,\vbar\ra_H+2\la |u|^2\ubar, v\ra_{H},\\
\big\la D^2\big(\|u\|^4_{L^4}\big)(z),v\big\ra_H & =  4 \la |u|^2,z\vbar+\zbar v\ra_H+2 \la u^2,\zbar \vbar\ra_H+2\la \ubar^2,zv\ra_H.
\end{align*}
We apply It\^o's formula to $\|u\|_{L^4}^4$ and obtain
\begin{align} \label{eqn:d.|u|^4_L4}
&\d \|u(t)\|_{L^4}^4  \notag \\
& =  -8\gamma \big\|\Re\big(u(t)\grad \ubar(t)\big)\big\|^2_H\d t -4\gamma\la |u(t)|^2,|\grad u(t)|^2 \big\ra_H\d t+2\Re\Big( \la |u(t)|^2u(t),\i A\ubar(t)\ra_H\Big)\d t   \nt   \\
&\qquad-4\alpha\|u(t)\|_{L^4}^4\d t+8\sum_{k=1}^N \lambda_k^2\la |u(t)|^2,e_k^2\ra_H\d t \nt \\
&\qquad +2\big\la |u(t)|^2u(t),\overline{Q\d W(t)}\big\ra_H+2\big\la |u(t)|^2\ubar(t), Q\d W(t)\big\ra_{H}.
\end{align}

Next, regarding $\|u\|^2_{H^1}$, we have
\begin{align} \label{eqn:d|u|^2_H1}
\d \|u(t)\|^2_{H^1} & = -2\gamma \|u(t)\|^2_{H^2}\d t -2\alpha\|u(t)\|^2_{H^1}\d t+\Re\Big( \la |u(t)|^2u(t),\i A\ubar(t)\ra_H\Big)\d t \nt \\
&\qquad  +\Tr(AQQ^*)\d t+ \big\la u(t),\overline {Q\d W(t)}\big\ra_{H^1}+\big\la \ubar(t),Q\d W(t)\big\ra_{H^1}.
\end{align}
Combining two above identities produces
\begin{align*}
&\d \|u(t)\|^2_{H^1} -\frac{1}{2}\d \|u(t)\|_{L^4}^4 \nt \\
&= -2\gamma \|u(t)\|^2_{H^2}\d t  +4\gamma \big\|\Re\big(u(t)\grad \ubar(t)\big)\big\|^2_H\d t +2\gamma\la |u(t)|^2,|\grad u(t)|^2 \big\ra_H\d t \nt \\
&\qquad -2\alpha\|u(t)\|^2_{H^1}\d t +2\alpha\|u(t)\|_{L^4}^4\d t +\Tr(AQQ^*)\d t  \nt  \\
&\qquad -4\sum_{k=1}^N \lambda_k^2\la |u(t)|^2,e_k^2\ra_H\d t +\d M_{1,1}(t),
\end{align*}
where
\begin{align*}
\d M_{1,1}(t) & =  \big\la u(t),\overline {Q\d W(t)}\big\ra_{H^1}+\big\la \ubar(t),Q\d W(t)\big\ra_{H^1}\\
&\qquad -\big\la |u(t)|^2u(t),\overline{Q\d W(t)}\big\ra_H-\big\la |u(t)|^2\ubar(t), Q\d W(t)\big\ra_{H}.
\end{align*}
To estimate the terms involving $\grad u(t)$ on the above right-hand side, we note that $\gamma>0$ is taken to be sufficiently small. Also, using $H^{3/4}\subset L^\infty$ (in dimension $d=1$) and interpolation inequalities, we have the following chain of estimates
\begin{align*}
 \la |u|^2,|\grad u |^2\ra_H \le \|\grad u\|^2_{L^\infty}\|u\|^2_H\le c\|\grad u\|_{H^{3/4}}^2\|u\|^2_H&\le c\|u\|_{H^2}^{7/4}\|u\|_{H}^{9/4}.
\end{align*}
It follows that
\begin{align*} 
&4\gamma \big\|\Re\big(u\grad \ubar\big)\big\|^2_H +2\gamma\la |u|^2,|\grad u|^2 \big\ra_H \nt \\
&\le 6  \gamma\la |u|^2,|\grad u|^2 \big\ra_H \le c\,\gamma^{\frac{8}{7}}\|u\|^2_{H^2}+\kappa\alpha\|u\|^{18}_{H}.
\end{align*}
In the above, the positive constant $c=c(\kappa,\alpha)$ is independent of $\gamma$. Hence, by taking $\gamma$ sufficiently small, we obtain the bound
\begin{align} \label{ineq:d|u|^2_H1-d|u|^4_L4}
&\d \|u(t)\|^2_{H^1} -\frac{1}{2}\d \|u(t)\|_{L^4}^4 \nt \\
&\le  \kappa\alpha\|u(t)\|^{18}_{H} -2\alpha\|u(t)\|^2_{H^1}\d t +2\alpha\|u(t)\|_{L^4}^4\d t +\Tr(AQQ^*)\d t   +\d M_{1,1}(t),
\end{align}

Turning back to $\Phi$ defined in \eqref{form:Phi}, from \eqref{ineq:d|u|^n_H} (with $2n=6$ and $2n=18$) and \eqref{ineq:d|u|^2_H1-d|u|^4_L4}, we set $\Phi(t):=\Phi(u(t))$ and deduce that
\begin{align*}
\d \Phi(t) &\le    -2\alpha\|u(t)\|^2_{H^1}\d t +2\alpha\|u(t)\|_{L^4}^4\d t -3 \alpha \kappa \|u(t)\|^{6}_H\d t-8\alpha \kappa \|u(t)\|^{18}_H\d t \nt \\ 
&\qquad +\Tr(AQQ^*)\d t +2^5\frac{|\Tr(QQ^*)|^{3}}{\alpha^{2}}\d t +2^{33}\frac{|\Tr(QQ^*)|^{9}}{\alpha^{8}}\d t \nt  \\
&\qquad +\d M_{1,1}(t)+ 3\kappa\|u(t)\|^{4}_H \d M_0(t)  + 9\kappa\|u(t)\|^{16}_H \d M_0(t).
\end{align*}
In view of \eqref{ineq:Gagliardo-Nirenberg}, it holds that
\begin{align*}
2\alpha\|u\|_{L^4}^4 \le \frac{1}{2}\alpha\|u\|^2_{H^1}+\alpha\kappa\|u\|^{6}_H.
\end{align*}
We combine the above two estimates to obtain
\begin{align} \label{ineq:d.Phi(t)}
\d \Phi(t) 
 &\le    -\alpha\Phi(u(t))\d t   +C_{1,1}\d t +\d M_1(t),
\end{align}
where
\begin{align*}
C_{1,1}=\Tr(AQQ^*) +2^5\frac{|\Tr(QQ^*)|^{3}}{\alpha^{2}} +2^{33}\frac{|\Tr(QQ^*)|^{9}}{\alpha^{8}},
\end{align*}
and the semi Martingale process $M_1(t)$ is defined as
\begin{align} \label{form:M_1(t)}
\d M_1(t) & =   \big\la Au(t)-|u(t)|^2u(t)+\big(3\kappa\|u(t)\|^{4}_H+9\kappa\|u(t)\|^{16}_H\big)u(t) ,\overline{Q\d W(t)}\big\ra_H \nonumber \\
&\qquad +\big\la A\ubar(t) -|u(t)|^2\ubar(t)+\big(3\kappa\|u(t)\|^{4}_H+9\kappa\|u(t)\|^{16}_H\big)\ubar(t), Q\d W(t)\big\ra_{H},
\end{align}
whose quadratic variation is given by
\begin{align} \label{form:M_1(t):quadratic_variation}
\d \la M_1\ra(t) &= 4 \sum _{k=1}^N \big|\big\la  Au(t)-|u(t)|^2u(t)+\big(3\kappa\|u(t)\|^{4}_H+9\kappa\|u(t)\|^{16}_H\big)u(t) , Qe_k\big\ra \big|^2\d t.
\end{align}
In particular, we have the bound
\begin{align*} 
& 4 \sum _{k=1}^N \big|\big\la  Au-|u|^2u+\big(3\kappa\|u\|^{4}_H+9\kappa\|u\|^{16}_H\big)u , Qe_k\big\ra \big|^2\\
&\le 12 \Tr(AQQ^*)\|u\|^2_{H^1}+12 \Tr(QQ^*)\Big[ \|u\|^6_{L^6}+ \big(3\kappa\|u\|^{4}_H+9\kappa\|u\|^{16}_H\big)^2\|u\|^2_H    \Big]\\
& \le 12 \Tr(AQQ^*)\|u\|^2_{H^1}+12 \Tr(QQ^*)\Big[ \|u\|^6_{L^6}+ 18\kappa^2\|u\|^{10}_H+162\kappa^2\|u\|^{34}_H    \Big].
\end{align*}
With regard to $\|u\|^6_{L^6}$ on the above right-hand side, for each $\varepsilon>0$, we invoke Gagliardo-Nirenberg interpolation to see that the following holds
\begin{align*}
\|u\|_{L^6}^6 \le c\|u\|^{2}_{H^1}\|u\|^{4}_H,
\end{align*}
whence
\begin{align*} 
\d \la M_1\ra(t) & \le 12 \Tr(AQQ^*)\|u(t)\|^2_{H^1}\d t+12 \Tr(QQ^*)\Big[ c\|u(t)\|^{2}_{H^1}\|u(t)\|^{4}_H \nt \\
&\qquad\qquad\qquad\qquad+ 18\kappa^2\|u(t)\|^{10}_H+162\kappa^2\|u(t)\|^{34}_H    \Big].
\end{align*}
In light of \eqref{cond:Phi}, since $\Tr(QQ^*)<C<\infty$ and $\|u\|^2_{H^1}\le \frac{4}{3}\Phi(u)$, we deduce further that
\begin{align} \label{ineq:d<M_1>(t)}
\d \la M_1\ra(t) & \le 16 \Tr(AQQ^*)\Phi(t)\d t+C\Big[ \|u(t)\|^{2}_{H^1}\|u(t)\|^{4}_H+\|u(t)\|^{10}_H+\|u(t)\|^{34}_H    \Big] \nonumber \\
&\le \Phi(t)^2\d t+ C\d t.
\end{align}

Next, for each $n\ge 2$, by It\^o's formula, we compute
\begin{align*}
\d \Phi(t)^n & = n \Phi(t)^{n-1}\d H(t)+\frac{1}{2}n(n-1)\Phi(t)^{n-2}\d \la M_1\ra(t).
\end{align*}
From \eqref{ineq:d.Phi(t)} and \eqref{ineq:d<M_1>(t)}, we obtain
\begin{align*}
\d \Phi(t)^n &\le -n\alpha \Phi(t)^n\d t+ n(8n-7)\Tr(AQQ^*)\Phi(t)^{n-1}\d t+n\Phi(t)^{n-1}\d M_1(t)\\
&\qquad +c\,\Phi(t)^{n-1}\d t+c\, \Phi(t)^{n-2}\Big(\|u(t)\|^{2}_{H^1}\|u(t)\|^{4}_H + \|u(t)\|^{10}_H+\|u(t)\|^{34}_H  \Big)\d t.  
\end{align*}
Observe that by virtue of Young's inequality,
\begin{align*}
& n(8n-7)\Tr(AQQ^*)\Phi(t)^{n-1}
\\ &\le \frac{1}{4}n\alpha \Phi(t)^n + (8n-7)^{n}\Big(\frac{n}{4(n-1)}\Big)^{n-1}\cdot \frac{|\Tr(AQQ^*)|^n}{\alpha^{n-1}}.
\end{align*}
Likewise,
\begin{align*}
& c\,\Phi(t)^{n-1}+c\, \Phi(t)^{n-2}\Big(\|u(t)\|^{2}_{H^1}\|u(t)\|^{4}_H + \|u(t)\|^{10}_H+\|u(t)\|^{34}_H  \Big)\\
&\qquad\qquad\le \frac{1}{4}n\alpha \Phi(t)^n +C\d t.
\end{align*}
Altogether, we arrive at the bound
\begin{align} \label{ineq:d.Phi(t)^n}
\d \Phi(t)^n &\le -\frac{1}{2}n\alpha \Phi(t)^n \d t + (8n-7)^{n}\Big(\frac{n}{4(n-1)}\Big)^{n-1}\cdot \frac{|\Tr(AQQ^*)|^n}{\alpha^{n-1}}\d t \nt \\
&\qquad +C_{n,\alpha}\d t + n\Phi(t)^{n-1}\d M_1(t).
\end{align}
In the above right-hand side, we emphasize that the positive constant $C_{n,\alpha}>0$ is independent of $\gamma$, $u$ and $Q$.

1. Turning back to \eqref{ineq:moment:|u|_H1+int.|u|_H1<|u_0|_H1+Ct} and \eqref{ineq:moment:|u|_H1<e^(-ct)|u_0|_H1}, we note that \eqref{ineq:d.Phi(t)^n} implies the bound in expectation
\begin{align*}
\frac{\d }{\d t}\E\Phi(t)^n \le -\frac{1}{2}n\alpha \E \Phi(t)^n \d t + (8n-7)^{n}\Big(\frac{n}{4(n-1)}\Big)^{n-1}\cdot \frac{|\Tr(AQQ^*)|^n}{\alpha^{n-1}}+C,\quad t\ge 0.
\end{align*}
On the one hand, we integrate the above estimate on both sides with respect to time $t$ and immediately obtain \eqref{ineq:moment:|u|_H1+int.|u|_H1<|u_0|_H1+Ct}. On the other hand, by virtue of Gronwall's inequality, we establish \eqref{ineq:moment:|u|_H1<e^(-ct)|u_0|_H1}, as claimed.

2. With regard to the sup norm estimate \eqref{ineq:moment:sup_[0,T]Phi(t)^n}, we note that \eqref{ineq:d.Phi(t)^n} implies the almost surely bound
\begin{align} \label{ineq:sup_[0,T]Phi(t)^n}
\sup_{[0,T]}\Phi(t)^n &\le \Phi(0)^{n} + (8n-7)^{n}\Big(\frac{n}{4(n-1)}\Big)^{n-1}\cdot \frac{|\Tr(AQQ^*)|^n}{\alpha^{n-1}} T +C\, T \nt \\
&\qquad\qquad + \sup_{t\in[0,T]} \int_0^t n\Phi(r)^{n-1}\d M_1(r).
\end{align}
Concerning the semi Martingale on the above right-hand side, we invoke \eqref{ineq:d<M_1>(t)} while making use of Burkholder's inequality to infer
\begin{align*}
&\E  \sup_{t\in[0,T]} \Big|\int_0^t n\Phi(r)^{n-1}\d M_1(r)\Big|\\
& \le n \Big|\E\int_0^T \Phi(r)^{2n-1} \cdot 16 \Tr(AQQ^*)\d r\\
&\qquad+C\E\int_0^T \Phi(r)^{2n-2}\Big[ \|u(r)\|^{2}_{H^1}\|u(r)\|^{4}_H+\|u(r)\|^{10}_H+\|u(r)\|^{34}_H    \Big] \d r \Big|^{1/2}\\
&\le  \frac{1}{2} n+ n \E\int_0^T \Phi(r)^{2n-1} \cdot 8 \Tr(AQQ^*)\d r\\
&\qquad+C\E\int_0^T \Phi(r)^{2n-2}\Big[ \|u(r)\|^{2}_{H^1}\|u(r)\|^{4}_H+\|u(r)\|^{10}_H+\|u(r)\|^{34}_H    \Big] \d r .
\end{align*}
where $C=C(n,\alpha,\Tr(QQ^*))$ is a positive constant independent of $T,\gamma$ and $u_0$. We employ Young's inequality to see that
\begin{align*}
& C\E\int_0^T \Phi(r)^{2n-2}\Big[ \|u(r)\|^{2}_{H^1}\|u(r)\|^{4}_H+\|u(r)\|^{10}_H+\|u(r)\|^{34}_H    \Big] \d r\\
&\le  \frac{1}{4}\alpha n \E\int_0^T \Phi(r)^{2n}\d r+C T,
\end{align*}
and that
\begin{align*}
&n \E\int_0^T \Phi(r)^{2n-1} \cdot 8 \Tr(AQQ^*)\d r\\
&\le  \frac{1}{4}\alpha n \E\int_0^T \Phi(r)^{2n}\d r + 4\Big(\frac{2n-1}{n}\Big)^{2n-1}\frac{|\Tr(AQQ^*)|^{2n}}{\alpha^{2n-1}} T.
\end{align*}
As a consequence, we deduce
\begin{align*}
&\E  \sup_{t\in[0,T]} \Big|\int_0^t n\Phi(r)^{n-1}\d M_1(r)\Big|\\ 
&\le \frac{1}{2}\alpha n \E\int_0^T \Phi(r)^{2n}\d r + 4\Big(\frac{2n-1}{n}\Big)^{2n-1}\frac{|\Tr(AQQ^*)|^{2n}}{\alpha^{2n-1}} T +C T,
\end{align*}
for some positive constant $C=C(n,\alpha,\Tr(QQ^*))$ is a positive constant independent of $T,\gamma$ and $u_0$. In view of \eqref{ineq:moment:|u|_H1+int.|u|_H1<|u_0|_H1+Ct}, we further obtain
\begin{align*}
&\E  \sup_{t\in[0,T]} \Big|\int_0^t n\Phi(r)^{n-1}\d M_1(r)\Big|\\ 
&\le \E \big[\Phi(0)^{2n}\big]+ \Big[(16n-7)^{2n}\Big(\frac{2n}{4(2n-1)}\Big)^{2n-1}+ 4\Big(\frac{2n-1}{n}\Big)^{2n-1}\Big]\frac{|\Tr(AQQ^*)|^{2n}}{\alpha^{2n-1}} T +C T.
\end{align*}
This together with \eqref{ineq:sup_[0,T]Phi(t)^n} implies the bound
\begin{align*}
\E \big[\sup_{[0,T]}\Phi(t)^n\big] &\le \E\big[\Phi(0)^{n}\big]+ \E \big[\Phi(0)^{2n}\big] + CT \\
&\qquad + (8n-7)^{n}\Big(\frac{n}{4(n-1)}\Big)^{n-1}\cdot \frac{|\Tr(AQQ^*)|^n}{\alpha^{n-1}}T\\
&\qquad + \Big[(16n-7)^{2n}\Big(\frac{2n}{4(2n-1)}\Big)^{2n-1}+ 4\Big(\frac{2n-1}{n}\Big)^{2n-1}\Big]\frac{|\Tr(AQQ^*)|^{2n}}{\alpha^{2n-1}} T,
\end{align*}
which is the desired estimate \eqref{ineq:moment:sup_[0,T]Phi(t)^n}. The proof is thus finished.
\end{proof}

Having established moment bounds in $H$ and $H^1$, we turn to controlling the energy growth in probability, which is crucial in proving uniform ergodicity in Section \ref{sec:ergodicity:Ginzburg-Landau}. For this purpose, we introduce the following functional for $n\ge 1$
\begin{align} \label{form:E_n(t)}
E_n(t;u_0)= \Phi(\ug(t;u_0))^n +\frac{1}{2}n\alpha \int_0^t\Phi(\ug(s;u_0))^n\d s,
\end{align}
where we recall $\Phi$ defined in \eqref{form:Phi}. In Lemma \ref{lem:moment:P[E_k(t)>Phi(0)^n+rho(Phi(0)^2n+t)]} below, we provide two tail probabilities on $E_n$. The former appears in the proof of Lemma \ref{lem:ergodicity:P(ell(k+1)=k+1|ell(k)=infinity):Girsanov} whereas the latter is an ingredient for proving Proposition \ref{prop:ergodicity:P(ell(k+1)=l|ell(k)=l)}.

\begin{lemma} \label{lem:moment:P[E_k(t)>Phi(0)^n+rho(Phi(0)^2n+t)]}
For all $n\ge 1,p\ge 1$ and $\rho\ge 0$, the followings hold
\begin{align}
&\P\Big(\sup_{t\in[0,T]}\big(E_n(t;u_0) - (C_n-1) t\big)  \ge \Phi(u_0)^n+\rho\sqrt{T} \Big)\notag \\
&\qquad\le K_{n,p}\frac{\E[\Phi(u_0)^{np}]+1 }{\rho
^p },\label{ineq:moment:P[sup_[0,T]E_k(t)>Phi(0)^n+rho(Phi(0)^2n+t)} \quad T>0,\\
&\P\Big(\sup_{t\in[T,\infty)}\big(E_n(t;u_0) - C_n t\big)  \ge \Phi(u_0)^n+\rho \Big)\notag \\
&\qquad\le K_{n,p}\frac{1}{ (\rho+T)^{2p-1}}\big(\E\Phi(u_0)^{2np}+1\big), \quad T>1\label{ineq:moment:P[sup_[T,infty)E_k(t)>Phi(0)^n+Phi(0)^2n+1+rho)},
\end{align}
for some positive constants $C_n,K_{n,p}$ independent of $u_0,T,\rho$ and $\gamma$. In the above, $E_n(t;u_0)$ is defined in \eqref{form:E_n(t)}.
\end{lemma}
\begin{proof}
With regard to \eqref{ineq:moment:P[sup_[0,T]E_k(t)>Phi(0)^n+rho(Phi(0)^2n+t)}, we recall from \eqref{ineq:d.Phi(t)^n} that
\begin{align*}
 \d \Phi(t)^n  +\frac{1}{2}n\alpha \Phi(t)^n \d t &\le  (8n-7)^{n}\Big(\frac{n}{4(n-1)}\Big)^{n-1}\cdot \frac{|\Tr(AQQ^*)|^n}{\alpha^{n-1}}\d t  +C_{n,\alpha}\d t \\
 &\qquad+ n\Phi(t)^{n-1}\d M_1(t),
\end{align*}
where $M_1(t)$ is the semi Martingale process given by \eqref{form:M_1(t)} and $C_{n,\alpha}$ is a positive constant independent of $u_0,\gamma,T$. So, recalling $E_n$ defined in \eqref{form:E_n(t)} and setting 
\begin{align*}
    C_n=(8n-7)^{n}\Big(\frac{n}{4(n-1)}\Big)^{n-1}\cdot \frac{|\Tr(AQQ^*)|^n}{\alpha^{n-1}}  +C_{n,\alpha}+1,
\end{align*}
we infer that 
\begin{align*}
E_n(t;u_0)-(C_n-1) t \le \Phi(u_0)^n +\int_0^t n\Phi(s)^{n-1}\d M_1(s).
\end{align*}
It follows that for all $\rho>0$, Markov's inequality implies
\begin{align*}
&\P\Big(\sup_{t\in[0,T]}\big(E_n(t;u_0) - (C_n-1) t\big) \ge \Phi(u_0)^n+\rho\sqrt{T}\Big)\\
&\le \P\Big(\sup_{t\in[0,T]}\Big|\int_0^t n\Phi(s)^{n-1}\d M_1(s)\Big| \ge \rho\sqrt{T}\Big)\le \frac{1}{\rho^{2p} T^{p} }\E\bigg[\sup_{t\in[0,T]}\Big|\int_0^t n\Phi(s)^{n-1}\d M_1(s)\Big|^{2p}\bigg].
\end{align*}
Furthermore, Burkholder's inequality together with estimate \eqref{ineq:d<M_1>(t)} produces the bound
\begin{align*}
\E \bigg[\sup_{t\in[0,T]}\Big|\int_0^t n\Phi(s)^{n-1}\d M_1(s)\Big|^{2p}\bigg] &\le c\, \E \Big[\Big|\int_0^T  \Phi(s)^{2n-2}\d \la M_1\ra(s)\Big|^{p}\Big]\\
&\le c \E\Big[\Big|\int_0^T \big(\Phi(s)^{2n}+1\big)\d s\Big|^p\Big].
\end{align*}
We invoke Holder's inequality while making use of \eqref{ineq:moment:|u|_H1<e^(-ct)|u_0|_H1} to infer
\begin{align*}
 \E\Big[\Big|\int_0^T \big(\Phi(s)^{2n}+1\big)\d s\Big|^p\Big]
&\le c\, T^{p-1}\E\int_0^T \big(\Phi(s)^{2np} +1\big) \d s\le c\, T^{p}\big(\E\Phi(u_0)^{2np}+1\big),
\end{align*}
whence
\begin{align*} 
\E \bigg[\sup_{t\in[0,T]}\Big|\int_0^t n\Phi(s)^{n-1}\d M_1(s)\Big|^{2p}\bigg] \le c\, T^{p}\big(\E\Phi(u_0)^{2np}+1\big).
\end{align*}
Altogether, we obtain
\begin{align*}
\P\Big(\sup_{t\in[0,T]}\big(E_n(t;u_0) - (C_n-1) t\big) \ge  \Phi(u_0)^n+\rho\sqrt{T}\Big) \le \frac{c}{\rho^{2p} T^{p} }\cdot T^{p}\big(\E\Phi(u_0)^{2np}+1\big),
\end{align*}
which establishes \eqref{ineq:moment:P[sup_[0,T]E_k(t)>Phi(0)^n+rho(Phi(0)^2n+t)}, as claimed.

Turning to \eqref{ineq:moment:P[sup_[T,infty)E_k(t)>Phi(0)^n+Phi(0)^2n+1+rho)}, for each $k=0,1,2,\dots$, observe that
\begin{align*}
&\P\Big(\sup_{t\in[T+k,T+k+1]}\big(E_n(t;u_0) - C_n t\big) \ge \Phi(u_0)^n+\rho\Big)\\
& \le  \P\Big(\sup_{t\in[T+k,T+k+1]}\big(E_n(t;u_0) - (C_n -1)t\big) \ge \Phi(u_0)^n+\rho+T+k\Big)\\
&\le  \P\Big(\sup_{t\in[T+k,T+k+1]}\Big|\int_{0}^t n\Phi(s)^{n-1}\d M_1(s)\Big| \ge \rho+T+k\Big)\\
&\le \frac{c}{(\rho+T+k)^{2p}}\E\bigg[\sup_{t\in[T+k,T+k+1]}\Big|\int_{0}^{t} n\Phi(s)^{n-1}\d M_1(s)\Big|^{2p}\bigg].
\end{align*}
Similar to the proof of \eqref{ineq:moment:P[sup_[0,T]E_k(t)>Phi(0)^n+rho(Phi(0)^2n+t)}, we employ Holder's inequality and \eqref{ineq:moment:|u|_H1<e^(-ct)|u_0|_H1} to infer
\begin{align*}
&\E\bigg[\sup_{t\in[T+k,T+k+1]}\Big|\int_{0}^{t} n\Phi(s)^{n-1}\d M_1(s)\Big|^{2p}\bigg]\\
&\le c\,\E\Big[\Big|\int_{0}^{T+k+1} (\Phi(s)^{2n}+1)\d s\Big|^{p}\Big]\le c\,(T+k+1)^{p-1}\big(\E\Phi(u_0)^{2np}+1\big),
\end{align*}
implying
\begin{align*}
\P\Big(\sup_{t\in[T+k,T+k+1]}\big(E_n(t;u_0) - C_n t\big) \ge \Phi(u_0)^n+\rho\Big) \le c\frac{(T+k+1)^{p-1}}{(\rho+T+k)^{2p}}\big(\E\Phi(u_0)^{2np}+1\big).
\end{align*}
As a consequence, (recalling $T>1$)
\begin{align*}
&\P\Big(\sup_{t\in[T,\infty]}\big(E_n(t;u_0) - C_n t\big) \ge \Phi(u_0)^n+\rho\Big)\\
& \le c\,\big(\E\Phi(u_0)^{2np}+1\big)\sum_{k\ge 0}\frac{(T+k+1)^{p-1}}{(\rho+T+k)^{2p}}\le   \frac{c}{ (\rho+T)^{p}}\big(\E\Phi(u_0)^{2np}+1\big).
\end{align*}
The proof is thus complete.
\end{proof}

Finally, we turn to the stochastic Foias-Prodi estimate, stating that if the low modes are close, then so are the high modes. In order to precisely state the result, we introduce the function
\begin{align} \label{form:J}
J(u_1,u_2)
&=\|u_1-u_2\|^2_{H^1}-\Re\big\{\la u_1 u_2,(\ubar_1-\ubar_2)^2\ra_H\big\}\notag \\
&\qquad+ \kappa_2\big(\Phi(u_1)+\Phi(u_2)\big)\|u_1-u_2\|^2_H.
\end{align}
By Sobolev embedding $H^1\subset L^\infty$, we may pick $\kappa_2$ sufficiently large such that
\begin{align*}
\frac{1}{2}\kappa_2\big(\Phi(u_1)+\Phi(u_2)\big)\|u_1-u_2\|^2_H \ge \big|\Re\big\{\la u_1 u_2,(\ubar_1-\ubar_2)^2\ra_H\big\}\big|.
\end{align*}
In other words,
\begin{align} \label{cond:J}
J(u_1,u_2) \ge \|u_1-u_2\|^2_{H^1}+ \frac{1}{2}\kappa_2\big(\Phi(u_1)+\Phi(u_2)\big)\|u_1-u_2\|^2_H.
\end{align}

\begin{lemma} \label{lem:Foias-Prodi:Ginzburg-Landau}
For all $N\ge 1$, let $u_1^0,u_2^0\in L^2(\Omega;H^1)$ be such that $P_Nu_1^0=P_Nu_2^0$ where $P_N$ is the projection on \textup{span}$\{e_1,\dots,e_N\}$ and let $\tau_N$ be the stopping time defined as 
\begin{align} \label{form:tau:Foias-Prodi}
\tau_N=\inf\{t\ge 0: P_N\ug(t;u_1^0)\neq P_N\ug(t;u_2^0)\}.
\end{align}
Then, for all $\gamma$ sufficiently small, the following holds
\begin{align} \label{ineq:Foias-Prodi}
\E\Big[ \exp \Big\{ \alpha (t\mi\tau_N)& -\frac{c_*}{\alpha_N^{1/8}}\int_0^{t\mi \tau_N}\hspace{-0.5cm}\big[2\alpha\Phi(\ug(s;u_1^0))^4+2\alpha\Phi(\ug(s;u_2^0))^4+1\big] \d s \Big\}  J(t\mi \tau_N)\Big]\notag \\
&\le \E [J(u_1^0,u_2^0)],\quad t\ge 0,
\end{align}
for some positive constant $c_*$ independent of $u_1^0,\,u_2^0,\,t,\,N$ and $\gamma$. In the above, $J$ is given by \eqref{form:J} and $\alpha_N$ is the eigenvalue of $A$ as in \eqref{cond:Ae_k=alpha_k.e_k}.
\end{lemma}

\begin{remark} \label{rem:Foias-Prodi}
1. In the case of fixed $\gamma>0$, in view of \cite[Proposition 2.5]{odasso2006ergodicity}, we note that one can also deduce the following path-wise estimate
\begin{align*}
&\|\ug(t\mi \tau_N;u_1^0) -\ug(t\mi \tau_N;u_2^0)   \|_H\\
&\le \|u_1^0-u_2^0\|_H \exp\Big\{ -\frac{1}{2}\gamma\alpha_{N+1}(t\mi \tau_N)+c\big[ F(t\mi \tau_N;u_1^0)+ F(t\mi \tau_N;u_2^0)\big]   \Big\},
\end{align*}
where 
\begin{align*}
F(t;u_i^0) = \|\ug(t;u_i^0)\|_H^2+\gamma\int_0^t \|\ug(s;u_i^0)\|_H^2\d s,\quad i=1,2.
\end{align*}
In particular, it is clear that as $\gamma \to 0$, $N$ must be taken arbitrarily large accordingly, allowing for a dissipative effect strong enough to dominate the $F$ terms.

Here, in the vanishing viscosity regime, we are only able to derive an estimate in expectation, owing to the lack of dissipation as $\gamma\to 0$. Nevertheless, the result of Lemma \ref{lem:Foias-Prodi:Ginzburg-Landau} is sufficient to carry out the coupling argument presented in Section \ref{sec:ergodicity:Ginzburg-Landau} so as to establish ergodicity. Lemma \ref{lem:Foias-Prodi:Ginzburg-Landau} can also be considered as an analogue of \cite[Proposition 2.1]{debussche2005ergodicity} adapted to the setting of the Ginzburg-Landau equation.

2. Following closely the proof of Lemma \ref{lem:Foias-Prodi:Ginzburg-Landau}, the assumption that $\gamma$ is sufficiently small depending on $N$ is employed in estimate \eqref{cond:gamma:small}. This smallness condition on $\gamma$ is actually consistent with the analysis of the inviscid limit $\gamma\to 0$ that we will present in Section \ref{sec:gamma->0}.

3. We note that it is possible that the stopping time $\tau_N$ defined in \eqref{form:tau:Foias-Prodi} equals 0. Nevertheless, in Proposition \ref{prop:ergodicity:P(ell(k+1)=l|ell(k)=l)} below, we will demonstrate that, once the low modes of two solutions agree, there is a high probability they will stay synchronized in any large time interval, implying that $\tau_N>0$. 

4. Lastly, the choice of $\Phi$ and $J$ respectively defined in \eqref{form:Phi} and \eqref{form:J} are specific to the focusing equation \eqref{eqn:Schrodinger}. For the defocusing equation, obtained by replacing the nonlinear term 
$\i |u|^2u$ in \eqref{eqn:Schrodinger} by $-\i |u|^2u$, we may instead define
\begin{align*}
     \Phi_{\text{defocusing}}(u)&= \|u\|^2_{H^1}+\frac{1}{2}\|u\|^4_{L^4}+\kappa\|u\|_{H}^6+\kappa \|u\|^{18}_H,\\
    J_{\text{defocusing}}(u_1,u_2)
&=\|u_1-u_2\|^2_{H^1}+\textup{Re}\big\{\la u_1 u_2,(\ubar_1-\ubar_2)^2\ra_H\big\}\\
&\qquad+ \kappa_2\big(\Phi(u_1)+\Phi(u_2)\big)\|u_1-u_2\|^2_H.
\end{align*}
With these modifications, the uniform-in-$\gamma$
 results established in this section extend to the defocusing equation with only minor changes.
\end{remark}

\begin{proof}[Proof of Lemma \ref{lem:Foias-Prodi:Ginzburg-Landau}]
For notational convenience, we denote
\begin{align*}
u_1(t)=\ug(t;u_1^0),\quad \text{and} \quad u_2=\ug(t;u_2^0).
\end{align*}
Letting $v=u_1-u_2$, observe that $v$ satisfies the equation
\begin{align*}
\frac{\d}{\d t} v= -\gamma Av-\i Av -\alpha v+\i\big[ |u_1|^2 u_1-|u_2|^2u_2\big].
\end{align*}
We note that the last term on the above right hand side can be recast as
\begin{align*}
|u_1|^2 u_1-|u_2|^2 u_2 = (|u_1|^2+|u_2|^2)v+u_1u_2\vbar.
 \end{align*}
A routine calculation produces
\begin{align*}
\frac{\d}{\d t}\|v\|^2_{H}& = -2\gamma\| v\|^2_{H^1}-2\alpha\|v\|^2_{H}+2\Re\big\{ \i \la u_1u_2,(\vbar)^2\ra_H \big)\big\},
\end{align*}
Using Holder's inequality and $H^1\subset H^{3/4}\subset L^\infty$, we have
\begin{align} \label{cond:dimension-one}
\big|\Re\big\{ \i \la u_1u_2,(\vbar)^2\ra_H \big)\big\}\big|& \le \|u_1\|_H\|u_2\|_H \|v\|^2_{H^{3/4}}\le c\big(\Phi(u_1)+\Phi(u_2)\big)\|v\|^2_{H^{3/4}}.
\end{align}
Under the condition that $P_Nu_1=P_Nu_2$, it holds that
\begin{align*}
\|v\|_{H^{3/4}}= \|Q_Nv\|_{H^{3/4}}  \le \frac{1}{\alpha_N^{1/8}} \|v\|_{H^1},
\end{align*}
whence
\begin{align*}
\big|\Re\big\{ \i \la u_1u_2,(\vbar)^2\ra_H \big)\big\}\big|& \le  \frac{c}{\alpha_N^{1/4}}\big(\Phi(u_1)+\Phi(u_2)\big)\|v\|^2_{H^{1}}.
\end{align*}
It follows that
\begin{align} \label{ineq:d|v|^2_H}
\frac{\d}{\d t}\|v\|^2_{H}& \le -2\gamma\| v\|^2_{H^1}-2\alpha\|v\|^2_{H}+\frac{c}{\alpha_N^{1/4}}\big(\Phi(u_1)+\Phi(u_2)\big)\|v\|^2_{H^{1}}.
\end{align}

Next, we turn to $\|v\|^2_{H^1}$ and compute
\begin{align*}
\frac{\d}{\d t}\|v\|^2_{H^1}& = -2\gamma\|A v\|^2_H-2\alpha\|v\|^2_{H^1}+2\Re\big\{ \i\big( G_1 +  \la u_1u_2,(\grad \vbar)^2\ra_H \big)\big\},
\end{align*}
where
\begin{align*}
G_1&=\la \ubar_1\grad u_1,v\grad \vbar\ra_H +\la u_1\grad \ubar_1,v\grad \vbar\ra_H+  \la \ubar_2\grad u_2,v\grad \vbar\ra_H\\
&\hspace{2cm}+\la u_2\grad \ubar_2,v\grad \vbar\ra_H +  \la u_2\grad u_1,\vbar\grad \vbar\ra_H+  \la u_1\grad u_2,\vbar\grad \vbar\ra_H  .
\end{align*}
Similarly to the argument for \eqref{ineq:d|v|^2_H}, we have
\begin{align*}
\la \ubar_1\grad u_1,v\grad \vbar\ra_H \le \|u_1\|_{L^\infty}\|v\|_{L^\infty}\|u_1\|_{H^1}\|v\|_{H^1}& \le c\|v\|_{H^1}\|v\|_{H^{3/4}}\|u_1\|^2_{H^1} \\
&\le c \|v\|_{H^1}\|v\|_{H^{3/4}}\big(\Phi(u_1)+\Phi(u_2)\big).
\end{align*}
We note that the same argument can also be carried out for other terms on the right-hand side of $G_1$ equation. Hence, $G_1$ can be estimated as
\begin{align} \label{ineq:G_1}
|G_1| & \le c \|v\|_{H^1}\|v\|_{H^{3/4}}\big(\Phi(u_1)+\Phi(u_2)\big)\le \frac{c}{\alpha^{1/8}_N} \|v\|^2_{H^1}\big(\Phi(u_1)+\Phi(u_2)\big).
\end{align}
where in the last estimate, we once again employed the fact that $P_Nu_1=P_Nu_2$. It follows that
\begin{align} \label{ineq:d.|v|^2_H1}
\frac{\d}{\d t}\|v\|^2_{H^1}& \le -2\gamma\|A v\|^2_H-2\alpha\|v\|^2_{H^1}+\frac{c}{\alpha^{1/8}_N} \|v\|^2_{H^1}\big(\Phi(u_1)+\Phi(u_2)\big) \nonumber \\
&\qquad+2\Re\big\{ \i \la u_1u_2,(\grad \vbar)^2\ra_H \big\}.
\end{align}
In the above, we emphasize that $c$ is a positive constant independent of $\gamma,N$ and $t$.

Next, from \eqref{form:J}, we consider the term $\la u_1u_2,(\vbar)^2\ra_H$ and compute
\begin{align*}
&\d \la u_1u_2,(\vbar)^2\ra_H\\
 & = \la u_2(\vbar)^2,-\gamma A u_1-\i A u_1-\alpha u_1+\i |u_1|^2u_1\ra_H\d t+ \la u_2(\vbar)^2,Q\d W\ra_H\\
&\qquad +\la u_1(\vbar)^2,-\gamma A u_2-\i A u_2-\alpha u_2+\i |u_2|^2u_2\ra_H\d t+ \la u_1(\vbar)^2,Q\d W\ra_H\\
&\qquad + 2\la u_1u_2 \vbar, -\gamma A\vbar +\i A\vbar-\alpha\vbar-\i \big[(|u_1|^2+|u_2|^2)\vbar+\ubar_1\ubar_2v\big]\ra_H\d t\\
& =: G_2\d t+ \la (u_1+u_2)(\vbar)^2,Q\d W\ra_H- 2\gamma\la u_1 u_2, (\grad\vbar)^2\ra_H\d t +2\i \la u_1u_2,(\grad\vbar)^2\ra_H\d t.
\end{align*}
Similar to \eqref{ineq:G_1}, we employ $H^1\subset H^{3/4}\subset L^\infty$ to infer
\begin{align*}
|G_2|\le \frac{c}{\alpha^{1/8}_N} \|v\|^2_{H^1}\big(\Phi(u_1)+\Phi(u_2)\big).
\end{align*}
To estimate $- 2\gamma\la u_1 u_2, (\grad\vbar)^2\ra_H$, we note that up to this point we have not imposed the smallness of $\gamma$. So, by taking $\gamma\in (0,1)$ sufficiently small depending on $N$, we have
\begin{align} \label{cond:gamma:small}
\big| 2\gamma\la u_1 u_2, (\grad\vbar)^2\ra_H\big| \le c\gamma \big(\|u_1\|^2_{H^1}+\|u_2\|^2_{H^1}\big)\|v\|^2_{H^1}\le \frac{c}{\alpha^{1/8}_N} \|v\|^2_{H^1}\big(\Phi(u_1)+\Phi(u_2)\big).
\end{align}
It follows that
\begin{align} \label{ineq:d<u_1.u_2.(vbar)^2>_H}
-\d\, \Re\big\{ \la u_1u_2,(\vbar)^2\ra_H\big\} 
&\le \frac{c}{\alpha^{1/8}_N} \|v\|^2_{H^1}\big(\Phi(u_1)+\Phi(u_2)\big)\d t-2\Re\big\{ \i \la u_1u_2,(\grad \vbar)^2\ra_H \big\}\d t \nonumber \\
&\qquad +\Re\big\{\la (u_1+u_2)(\vbar)^2,Q\d W\ra_H\big\}.
\end{align}

With regard to the last term involving $\Phi$ on the right-hand side of \eqref{form:J}, on the one hand, we employ \eqref{ineq:d.Phi(t)} to infer
\begin{align*}
 \|v\|^2_H \d  \big(\Phi(u_1)+\Phi(u_2)\big)
&\le  -\alpha \big(\Phi(u_1)+\Phi(u_2)\big)\|v\|^2_H \d t+ C\|v\|^2_H \d t\\
&\qquad+ \|v\|^2_H \big[\d M_1(t;u_1^0)+\d M_1(t;u_2^0)\big],
\end{align*}
where $M_1(t;u_i^0)$ is the semi Martingale process $M_1$ as in \eqref{form:M_1(t)} with the initial condition $u_i^0$, $i=1,2$. We invoke the fact that $\|v\|_H^2\le \|v\|_{H^1}^2/\alpha_N $ (provided $P_Nu_1=P_Nu_2$) to deduce 
\begin{align*}
 \|v\|^2_H \d  \big(\Phi(u_1)+\Phi(u_2)\big)
&\le  -\alpha \big(\Phi(u_1)+\Phi(u_2)\big)\|v\|^2_H \d t+ \frac{c}{\alpha_N}\|v\|^2_{H^1} \d t\\
&\qquad+ \|v\|^2_H \big[\d M_1(t;u_1^0)+\d M_1(t;u_2^0)\big].
\end{align*}
On the other hand, from \eqref{ineq:d|v|^2_H}, we readily have
\begin{align*}
 \big(\Phi(u_1)+\Phi(u_2)\big)\frac{\d}{\d t} \|v\|^2_H \le -2\alpha \big(\Phi(u_1)+\Phi(u_2)\big)\|v\|^2_H + \frac{c}{\alpha_N^{1/4}}\big(\Phi(u_1)^2+\Phi(u_2)^2\big)\|v\|^2_{H^1}.
\end{align*}
Altogether, we obtain
\begin{align} \label{ineq:d(Phi(u_1)+Phi(u_2))|v|^2_H}
&\d \big[ \big(\Phi(u_1)+\Phi(u_2)\big)\|v\|^2_H\big] \nonumber \\
& \le -3\alpha \big(\Phi(u_1)+\Phi(u_2)\big)\|v\|^2_H\d t + \frac{c}{\alpha_N^{1/4}}\big[\Phi(u_1)^2+\Phi(u_2)^2+1\big]\|v\|^2_{H^1}\d t \nonumber \\
&\qquad + \|v\|^2_H \big[\d M_1(t;u_1^0)+\d M_1(t;u_2^0)\big].
\end{align}

Turning back to $J$ as in \eqref{form:J}, we collect \eqref{ineq:d.|v|^2_H1}, \eqref{ineq:d<u_1.u_2.(vbar)^2>_H} and \eqref{ineq:d(Phi(u_1)+Phi(u_2))|v|^2_H} to deduce the bound
\begin{align*}
&\d\, J \\
& \le -2\alpha\|v\|^2_{H^1}\d t-3\alpha \kappa_2 \big(\Phi(u_1)+\Phi(u_2)\big)\|v\|^2_H\d t + \frac{c}{\alpha_N^{1/8}}\big[2\alpha\Phi(u_1)^4+2\alpha\Phi(u_2)^4+1\big]\|v\|^2_{H^1}\d t \nonumber \\
&\qquad + \|v\|^2_H \big[\d M_1(t;u_1^0)+\d M_1(t;u_2^0)\big]+\Re\big\{\la (u_1+u_2)(\vbar)^2,Q\d W\ra_H\big\}.
\end{align*}
It follows that for all $0\le t\le \tau_N$
\begin{align*}
&\d\Big[ \exp \Big\{ \alpha t -\frac{c}{\alpha_N^{1/8}}\int_0^t\big[2\alpha\Phi(u_1(s))^4+2\alpha\Phi(u_2(s))^4+1\big] \d s \Big\}J(t)\Big]\\&\le \exp \Big\{ \alpha t -\frac{c}{\alpha_N^{1/8}}\int_0^t\big[2\alpha\Phi(u_1(s))^4+2\alpha\Phi(u_2(s))^4+1\big] \d s \Big\} \\
&\qquad\times \Big(\|v(t)\|^2_H \big[\d M_1(t;u_1^0)+\d M_1(t;u_2^0)\big] +\Re\big\{\la (u_1+u_2)(\vbar)^2,Q\d W\ra_H\big\}\Big),
\end{align*}
whence
\begin{align*}
&\E\Big[ \exp \Big\{ \alpha (t\mi \tau_N) -\frac{c}{\alpha_N^{1/8}}\int_0^{t\mi\tau_N}\big[2\alpha\Phi(u_1(s))^4+2\alpha\Phi(u_2(s))^4+1\big] \d s \Big\}J(t\mi\tau_N)\Big]\\
&\qquad\le \E [J(0)].
\end{align*}
This establishes \eqref{ineq:Foias-Prodi}, thereby finishing the proof.
\end{proof}

\begin{remark} \label{rem:dimension-one}
The dimensional restriction of the Foias-Prodi estimate \eqref{ineq:Foias-Prodi} follows from the bound \eqref{cond:dimension-one} where we invoke Sobolev embedding $H^{1/2+}\subset L^\infty$, which is only available in dimension one, but not in higher dimensions. See also the proof of \cite[Proposition 2.1]{debussche2005ergodicity} for the Schr\"odinger equation \eqref{eqn:Schrodinger}. 

2. To the best of the author's knowledge, while there are ergodic results established for the Schr\"odinger equation \eqref{eqn:Schrodinger} in the whole space in higher dimensions, e.g., \cite{brzezniak2023ergodic,ekren2017existence, nguyen2025polynomial}, the issue of unique ergodicity and mixing for \eqref{eqn:Schrodinger} in bounded domains when $d\ge 2$ still remains an open problem. 
\end{remark}

\section{Uniform polynomial mixing of \eqref{eqn:Ginzburg_Landau}} \label{sec:ergodicity:Ginzburg-Landau}

\subsection{Proof of Theorem \ref{thm:ergodicity:Ginzburg_Landau}} \label{sec:ergodicity:Ginzburg-Landau:Proof}

In this section, we establish the uniform polynomial convergence rate of $P^\gamma_t$ toward $\nu^\gamma$. Before discussing the proof of Theorem \ref{thm:ergodicity:Ginzburg_Landau}, it is illuminating to recapitulate the framework of \cite{debussche2005ergodicity,odasso2006ergodicity} on the coupling argument, which is the key tool for proving ergodicity. For this purpose, we set
\begin{align} \label{form:X_N.Y_N}
X^\gamma_N(t)&= X^\gamma_N(t;u_0) = P_N\ug(t;u_0),\notag \\
 Y^\gamma_N(t)&= Y^\gamma_N(t;u_0) = Q_N \ug(t;u_0)=(I-P_N)\ug(t;u_0),
\end{align}
where we recall $P_Nu$ is the projection of $u$ on span$\{e_1,\dots,e_N\}$. Observe that 
\begin{align} \label{eqn:Ginzburg-Landau:X_N}
\d \Xg_N(t)&=  \big(-\gamma A -\i A- \alpha\big)\Xg_N(t)  \d t+Q\d W(t) \nonumber \\
&\qquad+\i P_N\big[ \big|\Xg_N(t)+\Yg_N(t)\big|^2\big(\Xg_N(t)+\Yg_N(t)\big)\big]\d t,
\end{align}
and that
\begin{align} \label{eqn:Ginzburg-Landau:Y_N}
\d \Yg_N(t)&=  \big(-\gamma A -\i A- \alpha\big)\Yg_N(t)  \d t \nonumber \\
&\qquad+\i Q_N\big[ \big|\Xg_N(t)+\Yg_N(t)\big|^2\big(\Xg_N(t)+\Yg_N(t)\big)\big]\d t,
\end{align}
with the initial conditions $\Xg_N(0)=P_Nu_0$ and $\Yg_N(0)=Q_Nu_0$. Next, we introduce the following discrete random variables:
\begin{definition}{\cite{debussche2005ergodicity}} \label{def:ell_beta} 
1. For every $u_1^0,u_2^0\in L^2(\Omega;H^1)$, $N\ge 1$, $\theta>0$, $\beta>0$, $T>0$ and $k\in\nbb$, define 
\begin{align*}
\ell_{\theta,\beta}^N(k)=\min\{l\in\{0,\dots,k\}: P_{l,k} \textup{ holds} \},
\end{align*}
where $\min \emptyset =\infty$ and 
\begin{align*}
P_{l,k} =\begin{cases}
X^\gamma_N(t;u_1^0)=X^\gamma_N(t;u_2^0),& \forall t\in [lT,kT],\\
\Phi(\ug(lT;u_1^0))+ \Phi(\ug(lT;u_2^0))\le \beta,\\
E_4(t;u_i^0) \le \theta+ \beta^4+C_4(t-lT),& \forall t\in [lT,kT],\, i=1,2.
\end{cases}
\end{align*}
In the above, $\Phi$ is as in \eqref{form:Phi}, $E_4$ is given by \eqref{form:E_n(t)} and $C_4$ is the constant from Lemma \ref{lem:moment:P[E_k(t)>Phi(0)^n+rho(Phi(0)^2n+t)]} with $n=4$.

2. The pair $\big(\ug(t;u_1^0),\ug(t;u_2^0)\big)$ is said to be coupled on $[lT,kT]$ if $\ell_{\theta,\beta}^N(k)=l$.
\end{definition}

As mentioned in Section \ref{sec:main-result:polynomial-mixing}, the motivation for $\ell_{\theta,\beta}^N$ stems from the Foias-Prodi typed estimates producing an effective control on the solutions' energy provided that a sufficiently large number of low modes of the solutions are close \cite{debussche2005ergodicity,odasso2006ergodicity,
temam2012infinite}. For the convenience of the reader, we summarize the idea of the coupling technique in \cite[Theorem 2.9]{debussche2005ergodicity} and \cite[Theorem 1.8]{odasso2006ergodicity}. The argument essentially consists of three steps as follows.

\textit{Step 1}: Fix $R$ sufficiently larger than $C_{1,1}$ where $C_{1,1}$ is the constant from Lemma \ref{lem:moment:H_1}, e.g., $R>8C_{1,1}$. Starting from two initial conditions $u_1^0,u_2^0$, we let $\tau$ be the first time the solutions enter a ball of radius $R$, i.e., 
\begin{align*}
\tau=\inf\{k\ge 0: \Phi\big(\ug(kT;u_1^0)\big)+ \Phi\big(\ug(kT;u_2^0)\big)\le R\}.
\end{align*}
In view of the Lyapunov estimate presented in Lemma \ref{lem:moment:H_1}, it can be shown that 
\begin{align} \label{ineq:E.e^(xi.tau)<Phi(u_0)}
\E e^{\xi \tau}< c\big(1+\E\Phi(u_1^0)+\E\Phi(u_2^0)\big) ,
\end{align}
for some positive constants $\xi$ and $c$ independent of $\gamma,\, T$ and the initial conditions. See for example \cite[Section 1.4]{odasso2006ergodicity}. In particular, as a consequence, $\tau$ is a.s. finite.

\textit{Step 2}: Once inside the ball, we show that the probability the two solutions will couple is always bounded from below independent of $\gamma$, cf. Proposition \ref{prop:ergodicity:P(ell(k+1)=k+1|ell(k)=infinity)}. The proof of which relies on two auxiliary results collected in Lemma \ref{lem:irreducibility:Ginzburg-Landau} and Lemma \ref{lem:ergodicity:P(ell(k+1)=k+1|ell(k)=infinity):Girsanov}.

\textit{Step 3}: Once the solutions are coupled, on the one hand, the probability that they are close with respect to $d_1$ distance is high and is independent of $\gamma$. This is established in Lemma \ref{lem:ergodicity:P(d_1(u_1,u_2)>t^-q)} exploiting the Foias-Prodi estimate collected in Lemma \ref{lem:Foias-Prodi:Ginzburg-Landau}. On the other hand, for all $\beta>0$ sufficiently small and $T$ sufficiently large, the probability of decoupling is small with a power law order, which is also independent of $\gamma$. The precise statement is discussed in detailed in Proposition \ref{prop:ergodicity:P(ell(k+1)=l|ell(k)=l)}. Altogether, we are able to deduce the mixing rate \eqref{ineq:ergodicity:Ginzburg_landau} uniformly with respect to $\gamma$. 

See \cite[Theorem 2.9]{debussche2005ergodicity} and \cite[Theorem 1.8]{odasso2006ergodicity} for a more detailed explanation of the above steps. We now provide the auxiliary results whose proofs are deferred to the end of this section. We start by stating Lemma \ref{lem:ergodicity:P(d_1(u_1,u_2)>t^-q)} giving the high probability that two coupled solutions are close.

\begin{lemma} \label{lem:ergodicity:P(d_1(u_1,u_2)>t^-q)}
Let $\ell_{\theta,\beta}^N$ be the random variable as in Definition \ref{def:ell_beta}. Then, there exists $N_1=N_1(\alpha)$ sufficiently large such that for all $N\ge N_1$ and $\gamma=\gamma(N)$ sufficiently small, $\theta$, $\beta$, $q$, $T> 0$, and $0\le l\le k$,
\begin{align}
&\P\big(\big\{ d_1\big( \ug(t;u_1^0),\ug(t;u_2^0) \big)\ge c_0 (t-lT)^{-q}\big\} \cap \big\{\ell_{\theta,\beta}^N(k)\le l \big\} \big) \notag \\
&\qquad\qquad\le c_0(t-lT)^{-q},\quad t\in [lT,kT],
\end{align}
holds for some positive constant $c_0=c_0(\theta,\beta,N)$ independent of $\gamma$ and $t$. In the above, $d_1$ is the distance as in \eqref{form:d_k}.
\end{lemma}

Next, we state Proposition \ref{prop:ergodicity:P(ell(k+1)=k+1|ell(k)=infinity)}, establishing a lower bound for the probability of coupling once the solutions enter a ball.

\begin{proposition} \label{prop:ergodicity:P(ell(k+1)=k+1|ell(k)=infinity)}
For all $N\ge 1$, $R>0$, $\theta>0$, and $\beta=\beta(N)$ sufficiently small independent of $\gamma$, the following holds
\begin{align*}
\P\big(\ell_{\theta,\beta}^N(k+1)=k+1|\ell_{\theta,\beta}^N(k)=\infty,\Phi(\ug(kT_1;u_1^0))+ \Phi(\ug(kT_1;u_2^0))\le R\big) \ge \varepsilon_1,
\end{align*}
for some positive constants $T_1=T_1(N,R,\beta)>1$ and $\varepsilon_1=\varepsilon_1(N,R, \beta,T_1)$ independent of $\gamma$. In the above, $\Phi$ is the function defined in \eqref{form:Phi}.
\end{proposition}

Finally, we state the auxiliary result in Proposition \ref{prop:ergodicity:P(ell(k+1)=l|ell(k)=l)} ensuring a low probability of decoupling.

\begin{proposition} \label{prop:ergodicity:P(ell(k+1)=l|ell(k)=l)}
Let $N_1=N_1(\alpha)$ be the constant as in Lemma \ref{lem:ergodicity:P(d_1(u_1,u_2)>t^-q)}. Then, for all $q\ge 2$ and $N\ge N_1$, there exist positive constants $\theta>0$ sufficiently large and $\beta=\beta(q,N)$ sufficiently small both independent of $\gamma$ such that the following holds
\begin{align} \label{ineq:ergodicity:P(ell(k+1)=l|ell(k)=l)}
\P\big(\ell_{\theta,\beta}^N(k+1)\neq l|\ell_{\theta,\beta}^N(k)=l\big)\le  \frac{1}{2} [1+(k-l)T_1]^{-q},\quad 0\le l\le k,
\end{align}
where $T_1$ is the time constant as in Proposition \ref{prop:ergodicity:P(ell(k+1)=k+1|ell(k)=infinity)}.
\end{proposition}

Assuming the above results, let us conclude Theorem \ref{thm:ergodicity:Ginzburg_Landau}, whose argument is the same as that of \cite[Theorem 2.9]{debussche2005ergodicity}. See also the proof of \cite[Theorem 1.8]{odasso2006ergodicity}.

\begin{proof}[Proof of Theorem \ref{thm:ergodicity:Ginzburg_Landau}] Let $N_1$ be the integer constant as in Lemma \ref{lem:ergodicity:P(d_1(u_1,u_2)>t^-q)}, $q\ge 2$ be arbitrary and $R>8C_{1,1}$ where $C_{1,1}$ is the constant from Lemma \ref{lem:moment:H_1}. Then, there exist $\beta=\beta(q,N_1)$ sufficiently small and $T_1=T_1(N_1,R,\beta,q)$ sufficiently large such that Proposition \ref{prop:ergodicity:P(ell(k+1)=k+1|ell(k)=infinity)} and Proposition \ref{prop:ergodicity:P(ell(k+1)=l|ell(k)=l)} hold. Moreover, the choice $R>8C_{1,1}$ is to ensure that \eqref{ineq:E.e^(xi.tau)<Phi(u_0)} holds. See e.g., the proof of \cite[(1.56)]{odasso2006ergodicity}.

 Now, we proceed to prove Theorem \ref{thm:ergodicity:Ginzburg_Landau} by verifying the assumptions of \cite[Theorem 2.9]{debussche2005ergodicity}. First of all, Lemma \ref{lem:ergodicity:P(d_1(u_1,u_2)>t^-q)} establishes \cite[Condition (2.12)]{debussche2005ergodicity}. Next, the results in Proposition \ref{prop:ergodicity:P(ell(k+1)=k+1|ell(k)=infinity)} and Proposition \ref{prop:ergodicity:P(ell(k+1)=l|ell(k)=l)} respectively provides the essential bounds in \cite[Condition (2.14)]{debussche2005ergodicity} and \cite[Condition (2.13)]{debussche2005ergodicity}. Finally, Lemma \ref{lem:moment:H_1}, part 1. supplies the Lyapunov functions required in \cite[Condition (2.15)]{debussche2005ergodicity}. Altogether, by virtue of \cite[Theorem 2.9]{debussche2005ergodicity}, we deduce that there exists a unique invariant probability measure $\nu^\gamma$ such that $\nu^\gamma(H^1)=1$. In particular, the desired convergence rate \eqref{ineq:ergodicity:Ginzburg_landau} holds regardless of $\gamma$.

 Furthermore, if $\nu\in \Pcal r(H)$ is another invariant probability measure, then by the Krylov-Bogoliubov procedure and the estimates in Lemma \ref{lem:moment:L^2}, we note that $\nu$ must be supported in $H^1$. It follows that $\nu$ must be the same as $\nu^\gamma$, hence the uniqueness of $\nu^\gamma$ in $H$.

\end{proof}

\subsection{Proof of the auxiliary results}

We turn to the proof of the auxiliary results. First, we provide the proof of Lemma \ref{lem:ergodicity:P(d_1(u_1,u_2)>t^-q)} while making use of the Foias-Prodi estimate in Lemma \ref{lem:Foias-Prodi:Ginzburg-Landau}.

\begin{proof}[Proof of Lemma \ref{lem:ergodicity:P(d_1(u_1,u_2)>t^-q)}]
For notational convenience, we denote
\begin{align*}
u_1(t)=\ug(t;u_1^0),\quad \text{and} \quad u_2(t)=\ug(t;u_2^0).
\end{align*}
Without loss of generalization, we may assume $l=0$. Let $c_0$ be given and be chosen later. Recalling function $J$ defined in \eqref{form:J} and the stopping time $\tau_N$ from Lemma \ref{lem:Foias-Prodi:Ginzburg-Landau}, observe that for $t\in[0,kT]$,
\begin{align*}
&\big\{ d_1\big( u_1(t), u_2(t) \big)\ge c_0\, t^{-q}\big\} \cap \big\{\ell_{\theta,\beta}^N(k)=0 \big\} \\
&=\big\{ d_1\big( u_1(t),u_2(t) \big)\ge c_0 \, t^{-q}\big\} \cap \big\{\ell_{\theta,\beta}^N(k)=0 \big\}\cap \{t<\tau_N\}\\
&\subset \big\{ \| u_1(t)-u_2(t) \|^2_{H^1}\ge c_0^2 \, t^{-2q}\big\} \cap \big\{\ell_{\theta,\beta}^N(k)=0 \big\}\cap \{t<\tau_N\}\\
&= \big\{  \| u_1(t\mi \tau_N)-u_2(t\mi\tau_N) \|^2_{H^1}\ge c_0^2\, t^{-2q}\big\} \cap \big\{\ell_{\theta,\beta}^N(k)=0 \big\}\cap \{t<\tau_N\}\\
&\subset \Big\{ J(t\mi\tau_N) \ge c_0^2 \, t^{-2q} \Big\} \cap \big\{\ell_{\theta,\beta}^N(k)=0 \big\}.
\end{align*}
In the last inclusion above, we employed \eqref{cond:J}, i.e., $J(u,v)\ge \|u-v\|_{H^1}^2$. Given $\{\ell_{\theta,\beta}^N(k)=0 \}$, it holds that
\begin{align*}
\int_0^{t}2\alpha\Phi(u_1(s))^4+2\alpha\Phi(u_2(s))^4+1\d s \le 2\theta+ 2\beta^4+(2C_4+1)t,  
\end{align*}
whence
\begin{align*}
-\frac{c_*}{\alpha_N^{1/8}}\int_0^{t}2\alpha\Phi(u_1(s))^4+2\alpha\Phi(u_2(s))^4+1\d s\ge -\frac{c_*}{\alpha_N^{1/8}}\big[2\theta+\beta^4+(2C_4+1)t\big].
\end{align*}
In the above, $c_*$ is the constant from Lemma \ref{lem:Foias-Prodi:Ginzburg-Landau}. Since $c_*$ does not depend on $\alpha_N,\theta,\beta$ and $C_4$, we may take $N_1=N_1(\alpha)$ sufficiently large independent of $\beta,\theta$ such that for all $N\ge N_1$ and $\gamma$ sufficiently small
\begin{align} \label{ineq:J(t)>e^(alpha.t)}
 \alpha t-\frac{c_*}{\alpha_N^{1/8}}\int_0^{t}2\alpha\Phi(u_1(s))^4+2\alpha\Phi(u_2(s))^4+1\d s    \ge \frac{1}{2}\alpha t -2c_*\frac{\theta+\beta^4}{\alpha_N^{1/8}}.
\end{align}
It follows that
\begin{align*} 
&\big\{ d_1\big( u_1(t), u_2(t) \big)\ge c_0\, t^{-q}\big\} \cap \big\{\ell_{\theta,\beta}^N(k)=0 \big\} \\
&\subset \bigg\{\exp\Big\{ \alpha t-\frac{c_*}{\alpha_N^{1/8}}\int_0^{t}2\alpha\Phi(u_1(s))^4+2\alpha\Phi(u_2(s))^4+1\d s   \Big\}  J(t\mi\tau_N)\\
&\hspace{2cm} \ge c_0^2 \, t^{-2q} \exp\Big\{\frac{1}{2}\alpha t -2c_*\frac{\theta+\beta^4}{\alpha_N^{1/8}}\Big\} \bigg\}.
\end{align*}
From Lemma \ref{lem:Foias-Prodi:Ginzburg-Landau}, cf. \eqref{ineq:Foias-Prodi}, we invoke Markov inequality to obtain
\begin{align*}
&\P\big(\big\{ d_1\big( u_1(t), u_2(t) \big)\ge c_0\, t^{-q}\big\} \cap \big\{\ell_{\theta,\beta}^N(k)=0 \big\}\big) \\
&\le\E \big[J(0)|\ell_{\theta,\beta}^N(k)=0\big] \frac{t^{2q}}{c_0^2}\exp\Big\{-\frac{1}{2}\alpha t +2c_*\frac{\theta+\beta^4}{\alpha_N^{1/8}}\Big\} .
\end{align*}
Note that 
\begin{align*}
\E \big[J(0)|\ell_{\theta,\beta}^N(k)=0\big] \le \tilde{c}(\beta+\beta^2),\quad\text{and}\quad  t^{2q}e^{-\frac{1}{2}\alpha t}\le \tilde{c}t^{-q},\quad t\ge 0.
\end{align*}
Picking $c_0$ such that 
\begin{align*}
c_0^3= \tilde{c}^2(\beta+\beta^2)\exp\Big\{2c_*\frac{\theta+\beta^4}{\alpha_N^{1/8}}\Big\}, 
\end{align*}
produces
\begin{align*}
\P\big(\big\{ d_1\big( u_1(t), u_2(t) \big)\ge c_0\, t^{-q}\big\} \cap \big\{\ell_{\theta,\beta}^N(k)=0 \big\}\big) \le c_0 t^{-q}.
\end{align*}
The proof is thus finished.
\end{proof}

Next, we turn to Proposition \ref{prop:ergodicity:P(ell(k+1)=k+1|ell(k)=infinity)}. As mentioned in Section \ref{sec:ergodicity:Ginzburg-Landau:Proof}, the argument relies on two ingredients: the former is an irreducibility condition, cf. Lemma \ref{lem:irreducibility:Ginzburg-Landau}, allowing for driving the solutions to any arbitrarily small ball whereas the latter, cf. Lemma \ref{lem:ergodicity:P(ell(k+1)=k+1|ell(k)=infinity):Girsanov}, is a lower bound on the probability of coupling inside the small ball. In particular, the proof of Lemma \ref{lem:irreducibility:Ginzburg-Landau} will employ Lemma \ref{lem:irreducibility:eta} while Lemma \ref{lem:ergodicity:P(ell(k+1)=k+1|ell(k)=infinity):Girsanov} will invoke the result of Lemma \ref{lem:moment:P[E_k(t)>Phi(0)^n+rho(Phi(0)^2n+t)]}, part 1.

\begin{lemma} \label{lem:irreducibility:Ginzburg-Landau}
For all $R,\,r>0$, there exists $T_*=T_*(R,r)>0$ independent of $\gamma$ such that for all $t\ge T_*$ and $u_1^0,u_2^0$ satisfying $\Phi(u_1^0)+\Phi(u_2^0)\le R$, the following holds
\begin{align} \label{ineq:irreducibility:ginzburg-Landau}
\P\big( \Phi(\ug(t;u_1^0)) + \Phi(\ug(t;u_2^0))\le r \big)\ge \varepsilon_*,
\end{align}
for some positive constant $\varepsilon_*=\varepsilon_*(t,R,r)$ independent of $u_1^0,\,u_2^0$ and $\gamma$. 
\end{lemma}
\begin{proof} Letting $\eta(t)$ be the process satisfying \eqref{eqn:eta}, we denote $\vg=\ug-\eta$. From \eqref{eqn:Ginzburg_Landau} and \eqref{eqn:eta}, we observe that $\vg$ obeys the equation 
\begin{align*}
\frac{\d}{\d t} \vg & = -\gamma A(\vg+\eta) -\i A\vg-\alpha \vg+\i |\vg+\eta|^2(\vg+\eta),\\
&= -\gamma A\vg -\i A\vg-\alpha \vg+\i |\vg|^2\vg -\gamma A\eta+\i \big[ |\vg+\eta|^2(\vg+\eta) -|\vg|^2\vg\big],
\end{align*}
with the initial condition $ \vg(0)=\ug(0)$.

Considering $\Phi(\vg)$ where $\Phi$ is given by \eqref{form:Phi}, we employ an argument similar to the proof of \eqref{ineq:d.Phi(t)} and obtain
\begin{align*}
\frac{\d}{\d t}\Phi(\vg) & \le -2\gamma\|A\vg\|^2_H-\alpha\Phi(\vg) +\big\la D\Phi(\vg),  -\gamma A\eta+\i \big[ |\vg+\eta|^2(\vg+\eta) -|\vg|^2\vg\big]  \big\ra_H,
\end{align*}
where
\begin{align*}
&\big\la D\Phi(\vg),  -\gamma A\eta+\i \big[ |\vg+\eta|^2(\vg+\eta) -|\vg|^2\vg\big]  \big\ra_H\\
&= -2\gamma\Re(\la A \vg,A\bar{\eta}\ra_H) +2\Re \big( \big\la \i \big[ |\vg+\eta|^2(\vg+\eta) -|\vg|^2\vg\big],A\vbar ^\gamma \big\ra_H\big)\\
&\qquad -2\gamma \Re\big(\la  |\vg|^2\vg,A\etabar\ra_H  \big)+ 2\Re\big(  \big\la \i \big[ |\vg+\eta|^2(\vg+\eta) -|\vg|^2\vg\big],|\vg|^2\vbar^\gamma\big\ra_H \big)\\
&\qquad -6\kappa \gamma \|\vg\|^4_H\Re\big(\la  \vg,A\etabar\ra_H  \big)+ 6\kappa\|\vg\|^4_H\Re\big(  \big\la \i \big[ |\vg+\eta|^2(\vg+\eta) -|\vg|^2\vg\big],\vbar^\gamma\big\ra_H \big)\\
&\qquad -18\kappa \gamma \|\vg\|^{16}_H\Re\big(\la  \vg,A\etabar\ra_H  \big)+ 18\kappa\|\vg\|^{16}_H\Re\big(  \big\la \i \big[ |\vg+\eta|^2(\vg+\eta) -|\vg|^2\vg\big],\vbar^\gamma\big\ra_H \big)\\
& = I_1+\cdots+I_8.
\end{align*}
We now proceed to estimate $I_k$, $k=1,\dots,8$ on the above right-hand side. Concerning $I_1$, we employ Holder inequality to see that
\begin{align*}
I_1=-2\gamma\Re(\la A \vg,A\bar{\eta}\ra_H) \le \gamma\|A\vg\|^2_H+ \gamma\|A\eta\|^2_H.
\end{align*}
Similarly, recalling $\gamma<1$ and using $H^1\subset L^\infty$ imply
\begin{align*}
&I_3+I_5+I_7 \\
&=-2\gamma \Re\big(\la  |\vg|^2\vg,A\etabar\ra_H  \big)-6\kappa \gamma \|\vg\|^4_H\Re\big(\la  \vg,A\etabar\ra_H  \big)-18\kappa \gamma \|\vg\|^{16}_H\Re\big(\la  \vg,A\etabar\ra_H  \big)\\
&\le c\,\big( \|\vg\|^3_{H^1}+ \|\vg\|^5_{H}  +\|\vg\|^{17}_{H}    \big)\|A\eta\|_H \\
&\le c\,\big(1+\Phi(\vg)^{2}\big)\|A\eta\|_H.
\end{align*}
In the above, $c>0$ is a constant independent of $\gamma$. Next, we employ the inequality
\begin{align} \label{ineq:|u+z|^2(u+z)-|u|^2u<|z|(|u|^2+|z|^2)}
||u+z|^2(u+z)-|u|^2u|\le |z|(|u|^2+|z|^2),\quad u,z\in\cbb,
\end{align}
to infer
\begin{align*}
I_4 & = 2\Re\big(  \big\la \i \big[ |\vg+\eta|^2(\vg+\eta) -|\vg|^2\vg\big],|\vg|^2\vbar^\gamma\big\ra_H \big)\\
&\le 2\la |\eta|(|\vg|^2+|\eta|^2),|\vg|^3\ra_H\\
&\le c\|\eta\|_{L^\infty}\big(\|\vg\|^5_{L^\infty}+\|\eta\|^4_{L^\infty}+\|\vg\|^6_{L^\infty}  \big)\\
&\le c\,\big(1+\Phi(\vg)^{3}+\|A\eta\|^4_H\big)\|A\eta\|_H.
\end{align*}
Likewise,
\begin{align*}
I_6+I_8\le  c\,\big(1+\Phi(\vg)^{4}+\|A\eta\|^8_H\big)\|A\eta\|_H.
\end{align*}
Regarding $I_2$, we compute
\begin{align*}
&\big \la\big[ |\vg+\eta|^2(\vg+\eta) -|\vg|^2\vg\big],A\vbar ^\gamma \big\ra_H\\
&= \big\la \grad  \big(  |\vg|^2\etabar+2|\vg|^2\eta+2\vg|\eta|^2+\vbar^\gamma\eta^2+|\eta|^2\eta\big),\grad \vbar^\gamma \big\ra_H\\
&\le c\big(  \|\vg\|^3_{H^1}\|\eta\|_{H^1}  + \|\vg\|^2_{H^1}\|\eta\|^2_{H^1}+\|\vg\|_{H^1}\|\eta\|^3_{H^1}    \big)\\
&\le c \big(\|\vg\|^4_{H^1}+\|\eta\|^4_{H^1} +1  \big)\|\eta\|_{H^1}\\
&\le c\,\big(1+\Phi(\vg)^{2}+\|A\eta\|^4_H\big)\|A\eta\|_H.
\end{align*}
Altogether, we deduce
\begin{align*}
\frac{\d}{\d t}\Phi(\vg) & \le -\alpha\Phi(\vg) +  c\,\big(1+\Phi(\vg)^{4}+\|A\eta\|^8_H\big)\|A\eta\|_H,
\end{align*}
whence
\begin{align} \label{ineq:Phi(v^gamma)}
\Phi\big(\vg(t)\big) &\le e^{-\alpha t }\Phi\big(\vg(0)\big)+c\int_0^t e^{-\alpha (t-s)}\big(1+\Phi\big(\vg(s)\big)^{4}+\|A\eta(s)\|^8_H\big)\|A\eta(s)\|_H\d s \nonumber \\
&\le e^{-\alpha t }\Phi\big(\ug(0)\big)+c\sup_{s\in[0,t]} \Big[\big(1+\Phi\big(\vg(s)\big)^{4}+\|A\eta(s)\|^8_H\big)\|A\eta(s)\|_H\Big].
\end{align}
Recalling $\ug=\vg+\eta$, we note that
\begin{align*}
\Phi(\ug)\le c\big(\Phi(\vg)+\Phi(\eta)\big)\le c \big(\Phi(\vg)+\|A\eta\|^{18}_H\big).
\end{align*}
This together with \eqref{ineq:Phi(v^gamma)} produces
\begin{align} \label{ineq:Phi(u^gamma)}
\Phi\big(\ug(t)\big) 
&\le c\,e^{-\alpha t }\Phi\big(\ug(0)\big)+c\sup_{s\in[0,t]} \Big[\big(1+\Phi\big(\ug(s)\big)^{4}+\|A\eta(s)\|^{72}_H\big)\|A\eta(s)\|_H\Big].
\end{align}

Turning back to \eqref{ineq:irreducibility:ginzburg-Landau}, let $u_1^0$ and $u_2^0$ be given such that $\Phi(u_1^0)+\Phi(u_2^0)\le R$. From \eqref{ineq:Phi(u^gamma)}, we readily have
\begin{align} \label{ineq:Phi(u^gamma(t;u_1))+Phi(u^gamma(t;u_2))}
&\Phi\big(\ug(t;u_1^0)\big) + \Phi\big(\ug(t;u_2^0)\big) \le c\,e^{-\alpha t }\big[ \Phi(u_1^0)+ \Phi(u_2^0)\big]\nonumber\\
&  +c\sup_{s\in[0,t]} \Big[\big(1+\big[\Phi\big(\ug(s;u_1^0)\big)+\Phi\big(\ug(s;u_2^0)\big]^{4}+\|A\eta(s)\|^{72}_H\big)\|A\eta(s)\|_H\Big].
\end{align}
Denote by $\tau$ the stopping time defined as
\begin{align*}
\tau=\inf\{t\ge 0: \Phi\big(\ug(t;u_1^0)\big)+ \Phi\big(\ug(t;u_2^0)\big) > 3cR\},
\end{align*} 
where $c\ge 1$ is the same constant as in \eqref{ineq:Phi(u^gamma(t;u_1))+Phi(u^gamma(t;u_2))}. For each $t\ge 1$, conditioning on the event
\begin{align*}
B=\Big\{\sup_{s\in[0,t]}\|A\eta(s)\|_H\le \frac{Rr}{2(2+(3cR)^4)(1+R+r+c)^2 }\Big\},
\end{align*}
we claim that $\tau\ge t$. Indeed, suppose by means of contradiction, $\tau<t$. From \eqref{ineq:Phi(u^gamma(t;u_1))+Phi(u^gamma(t;u_2))}, we have
\begin{align*}
&\Phi\big(\ug(\tau;u_1^0)\big) + \Phi\big(\ug(\tau;u_2^0)\big) \\
&\le c e^{-\alpha \tau}R  +c\sup_{s\in[0,\tau]} \Big[\big(1+\big[\Phi\big(\ug(s;u_1^0)\big)+\Phi\big(\ug(s;u_2^0)\big]^{4}+\|A\eta(s)\|^{72}_H\big)\|A\eta(s)\|_H\Big]\\
& \le cR+c(2+(3cR)^4)\cdot\frac{Rr}{2(2+(3cR)^4)(1+R+r+c)^2 }< 2cR.
\end{align*}
This contradicts the fact that $\Phi\big(\ug(\tau;u_1^0)\big) + \Phi\big(\ug(\tau;u_2^0)\big)\ge 3cR$. Thus, $\tau\ge t$, as claimed.

Now, we pick 
\begin{align} \label{cond:T_*}
T_* = \frac{1}{\alpha}\log\Big(\frac{2(cR+r)}{r}\Big)>0.
\end{align} 
For all $t\ge T_*$, conditioning on event $B$ again, it holds that
\begin{align*}
&\Phi\big(\ug(t;u_1^0)\big) + \Phi\big(\ug(t;u_2^0)\big) \\
&\le c e^{-\alpha T_*}R  +c\sup_{s\in[0,\tau]} \Big[\big(1+\big[\Phi\big(\ug(s;u_1^0)\big)+\Phi\big(\ug(s;u_2^0)\big]^{4}+\|A\eta(s)\|^{72}_H\big)\|A\eta(s)\|_H\Big]\\
& \le \frac{cR}{2(cR+r)}\cdot r +c(2+(3cR)^4)\cdot\frac{Rr}{2(2+(3cR)^4)(1+R+r+c)^2 }< r.
\end{align*}
In view of Lemma \ref{lem:irreducibility:eta}, we infer a positive constant $\varepsilon=\varepsilon(t,T,r)$ independent of $\gamma$ such that
\begin{align*}
&\P\big( \Phi(\ug(t;u_1^0)) + \Phi(\ug(t;u_2^0))\le r \big) \\
&\ge \P\big( \Phi(\ug(t;u_1^0)) + \Phi(\ug(t;u_2^0))\le r \big| B\big)\P(B)\ge \P(B)\ge \varepsilon.
\end{align*}
This establishes \eqref{ineq:irreducibility:ginzburg-Landau}, thereby finishing the proof.

\end{proof}

\begin{lemma} \label{lem:ergodicity:P(ell(k+1)=k+1|ell(k)=infinity):Girsanov}
For all $N>0$, there exists a positive constant $\beta_1=\beta_1(N)$ sufficiently small such that for all $\beta\in(0,\beta_1)$,
 the following holds 
\begin{align} \label{ineq:ergodicity:P(ell(k+1)=k+1|ell(k)=infinity):Girsanov}
\P\big( \Xg_N(t_1;u_1)= \Xg_N(t_1;u_2), \Phi\big(\ug(t_1;u_1)\big)+ \Phi\big(\ug(t_1;u_2)\big)\le \beta \big)\ge \frac{1}{2},
\end{align}
for some positive constants
\begin{align*}
t_1=t_1(\beta,N)  \quad \text{and}\quad  r_1=r_1(\beta,N),
\end{align*}
that are both sufficiently small independent of $\gamma$ and
for all $u_1,u_2$ satisfying $\Phi(u_1)+\Phi(u_2)\le r_1$. In \eqref{ineq:ergodicity:P(ell(k+1)=k+1|ell(k)=infinity):Girsanov}, $X^\gamma_N$ is the process defined in \eqref{form:X_N.Y_N}-\eqref{eqn:Ginzburg-Landau:X_N}.
\end{lemma}
\begin{proof}
Let $t_1$, $r_1$ and $\rho$ be given and be chosen later, we aim to prove that there exists a coupling of $\big(\ug(\cdot;u_1),\ug(\cdot;u_2), W )$ such that \eqref{ineq:ergodicity:P(ell(k+1)=k+1|ell(k)=infinity):Girsanov} holds. We introduce the process $\uhat$ defined as
\begin{align*}
\uhat(t)= \ug(t;u_1)+\frac{t_1-t}{t_1}P_N(u_2-u_1),\quad 0\le t\le t_1.
\end{align*}
Setting $\Xhat_N=P_N\uhat$ and $\Yhat_N=Q_N\uhat$, observe that 
\begin{align*}
\Xhat(0)=P_N\uhat(0)= P_Nu_2= X_N^\gamma(0;u_2),\quad\text{and}\quad \Xhat(t_1) = P_N\uhat(t_1)= \Xg_N(t_1;u_1).
\end{align*}
Also, recall $E_4$ defined in \eqref{form:E_n(t)} and the constant $C_4$ from Lemma \ref{lem:moment:P[E_k(t)>Phi(0)^n+rho(Phi(0)^2n+t)]} (with $n=4$). Let $\tau_1,\tau_2$ be stopping times given by
\begin{align*}
\tau_i&=\min\{t\ge 0:E_4(t;u_i)-C_4 t\ge \Phi(u_i)^4+\rho_1\sqrt{t_1} \},\quad i=1,2.
\end{align*}
Then, we have the following chain of implications
\begin{align*}
&\P\Big( \big\{\Xg_N(t_1;u_1)= \Xg_N(t_1;u_2)\big\}\cap \big\{ E_4(t_1;u_1)-C_4 t_1\le \Phi(u_1)^4+\rho_1\sqrt{t_1}\big\}\\
&\hspace{4cm}\cap \big\{E_4(t_1;u_2)-C_4 t_1\le \Phi(u_2)^4+\rho_1\sqrt{t_1}\big\} \Big)\\
&\ge \P\Big( \big\{\forall t\in[0,t_1],\Xhat_N(t)= \Xg_N(t;u_2)\big\} \cap \big\{\sup_{t\in[0,t_1]}\big(E_4(t;u_1)-C_4 t\big)\le \Phi(u_1)^4+\rho_1\sqrt{t_1}\big\}\\
&\hspace{4cm}\cap \big\{\sup_{t\in[0,t_1]}\big(E_4(t;u_2)-C_4 t\big)\le \Phi(u_2)^4+\rho_1\sqrt{t_1}\big\} \Big)\\
& = \P\Big( \big\{\forall t\in[0,t_1],\Xhat_N(t)= \Xg_N(t;u_2)\big\}\cap \big\{ \tau_1\mi\tau_2\ge t_1\big\}\Big)\\
& = \P\Big( \big\{\forall t\in[0,t_1\mi \tau_1\mi\tau_2],\Xhat_N(t)= \Xg_N(t;u_2)\big\}\cap \big\{ \tau_1\mi\tau_2\ge t_1\big\}\Big),
\end{align*}
implying
\begin{align} \label{ineq:P(X^gamma(u_1)=X^gamma(u_2):E_4)}
&\P\Big( \big\{\Xg_N(t_1;u_1)= \Xg_N(t_1;u_2)\big\}\cap \big\{ E_4(t_1;u_1)-C_4 t_1\le \Phi(u_1)^4+\rho_1\sqrt{t_1}\big\} \nonumber\\
&\hspace{4cm}\cap \big\{E_4(t_1;u_2)-C_4 t_1\le \Phi(u_2)^4+\rho_1\sqrt{t_1}\big\} \Big) \nonumber \\
&\ge 1- \P\Big( \exists t\in[0,t_1\mi \tau_1\mi \tau_2],\Xhat_N(t)\neq \Xg_N(t;u_2) \Big)-\P(\tau_1\le t_1)-\P(\tau_2\le t_1).
\end{align}
In view of Lemma \ref{lem:moment:P[E_k(t)>Phi(0)^n+rho(Phi(0)^2n+t)]}, cf. \eqref{ineq:moment:P[sup_[0,T]E_k(t)>Phi(0)^n+rho(Phi(0)^2n+t)} with $n=4$ and $p=1$, it holds that
\begin{align*}
& \P(\tau_1\le t_1)+\P(\tau_2\le t_1)\\
&=\P\Big(\sup_{t\in[0,t_1]}\big(E_4(t_1;u_1)-C_4 t_1\big)\ge \Phi(u_1)^4+\rho_1\sqrt{t_1}\Big)\\
&\qquad+\P\Big(\sup_{t\in[0,t_1]}\big(E_4(t_1;u_2)-C_4 t_1\big)\ge \Phi(u_2)^4+\rho_1\sqrt{t_1} \Big)\\
&\le \P\Big(\sup_{t\in[0,t_1]}\big(E_4(t_1;u_1)-(C_4-1) t_1\big)\ge \Phi(u_1)^4+\rho_1\sqrt{t_1}\Big)\\
&\qquad+\P\Big(\sup_{t\in[0,t_1]}\big(E_4(t_1;u_2)-(C_4-1) t_1\big)\ge \Phi(u_2)^4+\rho_1\sqrt{t_1} \Big)\\
&\le  K_{4,1} \frac{\E[\Phi(u_1)^4+\Phi(u_2)^4]+2}{\rho_1},
\end{align*}
where $K_{4,1}$ is the constant on the right-hand side of \eqref{ineq:moment:P[sup_[0,T]E_k(t)>Phi(0)^n+rho(Phi(0)^2n+t)}. Conditioning on the event $\{\Phi(u_1)+\Phi(u_2)\le r_1\}$, we pick 
\begin{align} \label{cond:t_1_r_1:a}
\rho_1:=8K_{4,1}(r_1^4+1),
\end{align}
so as to deduce
\begin{align*}
 \P(\tau_1\le t_1)+\P(\tau_2\le t_1)\le \frac{1}{4}.
\end{align*}
This together with \eqref{ineq:P(X^gamma(u_1)=X^gamma(u_2):E_4)} produces
\begin{align} \label{ineq:P(X^gamma(u_1)=X^gamma(u_2))>1/2-P(X^gamma(u_1).neq.X^gamma(u_2))}
&\P\Big( \big\{\Xg_N(t_1;u_1)= \Xg_N(t_1;u_2)\big\}\cap \big\{ E_4(t_1;u_1)-C_4 t_1\le \Phi(u_1)^4+\rho_1\sqrt{t_1}\big\} \nonumber\\
&\hspace{4cm}\cap \big\{E_4(t_1;u_2)-C_4 t_1\le \Phi(u_2)^4+\rho_1\sqrt{t_1}\big\} \Big) \nonumber \\
&\ge \frac{1}{2}- \P\Big( \exists t\in[0,t_1\mi \tau_1\mi \tau_2],\Xhat_N(t)\neq \Xg_N(t;u_2) \Big).
\end{align}

Next, we claim that
\begin{align*}
\P\Big( \exists t\in[0,t_1\mi\tau_1\mi\tau_2],\Xhat_N(t)\neq \Xg_N(t;u_2) \Big) \le \frac{1}{4}.
\end{align*}
To this end, since
\begin{align*}
\d \uhat(t)&= -\gamma \ug(t;u_1)\d t -\i A \ug(t;u_1)\d t -\alpha \ug(t;u_1) \d t\\
&\qquad+\i |\ug(t;u_1)|^2\ug(t;u_1)\d t+Q\d W(t)-\frac{1}{t_1}P_N(u_2-u_1)\d t,
\end{align*}
$\Xhat_N$ satisfies the equation
\begin{align*}
\d \Xhat_N(t) & = \big(-\gamma A -\i A- \alpha\big)\Xhat_N(t)  \d t+\i P_N\big[ \big|\Xhat_N(t)+\Yg_N(t;u_2)\big|^2\big(\Xhat_N(t)+\Yg_N(t;u_2)\big)\big]\d t\\
&\qquad+Q\d W(t)+F_1(u_1,u_2,t)\d t,
\end{align*}
where
\begin{align*}
F_1(t;u_1,u_2)&:=\big(-\gamma A -\i A- \alpha\big)\frac{t_1-t}{t_1}P_N(u_2-u_1)-\frac{1}{t_1}P_N(u_2-u_1)\\
&\qquad +\i P_N\big[ |\ug(t;u_1)|^2\ug(t;u_1) - \big|\Xhat_N(t)+\Yg_N(t;u_2)\big|^2\big(\Xhat_N(t)+\Yg_N(t;u_2)\big)\big]. 
\end{align*}
Letting $\What$ be the process given by
\begin{align*}
\d \What= \d W+ Q^{-1}F_1\d t,
\end{align*}
the equation for $\Xhat_N$ can be recast as
\begin{align*}
\d \Xhat_N(t) & = \big(-\gamma A -\i A- \alpha\big)\Xhat_N(t)  \d t+\i P_N\big[ \big|\Xhat_N(t)+\Yg_N(t;u_2)\big|^2\big(\Xhat_N(t)+\Yg_N(t;u_2)\big)\big]\d t\\
&\qquad+Q\d \What(t).
\end{align*}
We note that the law induced by $\big(\Xg_N(\cdot;u_2),W\big)$ on $[0,t_1\mi\tau_1\mi\tau_2]$ is equivalent to that induced by $\big(\Xhat_N(\cdot),\What\big)$. Indeed, thanks to the condition that $Q$ is invertible on span$\{e_1,\dots,e_N\}$ by virtue of Assumption \ref{cond:Q}, it holds that
\begin{align*}
&\|Q^{-1}F_1(u_1,u_2,s)\|^2_{H} \le C(N)\Big[\|u_2-u_1\|^2_H+\frac{1}{t_1^2}\|u_2-u_1\|^2_H\Big]\\
&\quad+C(N)\big\| |\ug(t;u_1)|^2\ug(t;u_1) - \big|\Xhat_N(t)+\Yg_N(t;u_2)\big|^2\big(\Xhat_N(t)+\Yg_N(t;u_2)\big)\big\|^2_H,
\end{align*}
for some positive constant $C(N)$ that may be arbitrarily large as $N$ tends to infinity. Concerning the last term on the above right-hand side, we invoke \eqref{ineq:|u+z|^2(u+z)-|u|^2u<|z|(|u|^2+|z|^2)} while taking into account of the fact that
\begin{align*}
\ug(t;u_1)=\Xhat_N(t)-\frac{t_1-t}{t_1}P_N(u_2-u_1)+\Yg_N(t;u_1),
\end{align*}
to infer for $0\le t\le t_1$
\begin{align*}
&\big\| |\ug(t;u_1)|^2\ug(t;u_1) - \big|\Xhat_N(t)+\Yg_N(t;u_2)\big|^2\big(\Xhat_N(t)+\Yg_N(t;u_2)\big)\big\|^2_H\\
&\le c \big\| \big(|P_N(u_2-u_1)|+|\Yg_N(t;u_1)|+| \Yg_N(t;u_2)|\big)\\
&\hspace{2cm}\times\big(|\ug(t;u_1)|^2+|P_N(u_2-u_1)|^2+ |\Yg_N(t;u_1)|^2+| \Yg_N(t;u_2)|^2 \big)\big\|^2_H\\
&\le c\big[\Phi(u_1)^4+\Phi(u_2)^4+\Phi\big(\ug(t;u_1)\big)^4+\Phi\big(\ug(t;u_2)\big)^4\big].
\end{align*}
As a consequence, conditioning on the event $\{\Phi(u_1)+\Phi(u_2)\le r_1\}$
\begin{align} \label{ineq:int_0^t.|F|^2_H}
&\int_0^{t_1\mi\tau_1\mi\tau_2}\hspace{-0.5cm}\|Q^{-1}F_1(u_1,u_2,s)\|^2_{H}\d s \nonumber \\
&\le C(N)\Big[\Big(t_1+\frac{1}{t_1}+1\Big)\big(\Phi(u_1)^4+\Phi(u_2)^4\big)+C_4t_1+\rho_1\sqrt{t_1}    \Big] \nonumber\\
&\le  C(N)\Big[\Big(t_1+\frac{1}{t_1}+1\Big)r_1^4+C_4t_1+\rho_1\sqrt{t_1}    \Big],
\end{align}
where we emphasize again that the positive constant $C(N)$ may be arbitrarily large depending only on $N$. In particular, this implies
\begin{align*}
\E\exp\Big\{a\int_0^{t_1\mi\tau_1\mi\tau_2}\hspace{-0.5cm}\|Q^{-1}F_1(u_1,u_2,s)\|^2_{H}\d s\Big\}<\infty,\quad a>0,
\end{align*} 
which verifies Novikov's condition (with $a=1/2$), thereby establishing the equivalence in law. As a consequence, employing Pinsker's inequality together with \cite[Theorem A.2]{butkovsky2020generalized} we have the following bound
\begin{align*}
&\W_{\TV}\Big( \Law\big(W_{[0,t_1\mi\tau_1\mi\tau_2]}\big) ,\Law\big( \What_{[0,t_1\mi\tau_1\mi\tau_2]}\big) \Big) \\
&\quad\le \sqrt{\frac{1}{2}D_{\KL}\Big( \Law\big(W_{[0,t_1\mi\tau_1\mi\tau_2]}\big)\big\|\Law\big(\What_{[0,t_1\mi\tau_1\mi\tau_2]}\big) \Big)}\\
 &\quad \le \frac{1}{2}\sqrt{\E\int_0^{t\mi\tau_1\mi\tau_2}\hspace{-0.5cm}\|Q^{-1}F_1(u_1,u_2,s)\|^2_{H}\d s}.
\end{align*}
From \eqref{ineq:int_0^t.|F|^2_H}, conditioning on the event $\{\Phi(u_1)+\Phi(u_2)\le r_1\}$, we obtain
\begin{align*}
&\W_{\TV}\Big( \Law\big(W_{[0,t_1\mi\tau_1\mi\tau_2]}\big) ,\Law\big(\What_{[0,t_1\mi\tau_1\mi\tau_2]}\big) \Big)\\
&\quad \le  C(N)\sqrt{\Big(t_1+\frac{1}{t_1}+1\Big)r_1^4+C_4t_1+\rho_1\sqrt{t_1}  }.
\end{align*}
Also, by the uniqueness of the weak solutions, 
\begin{align*}
&\big\{ \exists t\in[0,t_1\mi \tau_1\mi \tau_2],\Xhat_N(t)\neq \Xg_N(t;u_2) \big\}\subset \big\{W_{[0,t_1\mi\tau_1\mi\tau_2]}\neq \What_{[0,t_1\mi\tau_1\mi\tau_2]} \big\}.
\end{align*}
In turn, this implies
\begin{align*} 
\W_{\TV}\Big( \Law\big(\Xhat_N\big) ,\Law\big(\Xg_N(\,\cdot\,;u_2)\big) \Big) 
&\le \W_{\TV}\Big( \Law\big(W_{[0,t_1\mi\tau_1\mi\tau_2]}\big) ,\Law\big(\What_{[0,t_1\mi\tau_1\mi\tau_2]}\big) \Big)\\
&\le C(N)\sqrt{\Big(t_1+\frac{1}{t_1}+1\Big)r_1^4+C_4t_1+\rho_1\sqrt{t_1}  }.
\end{align*}

We note that up to this point, we have not explicitly chosen a coupling of the triple $\big(\ug(\cdot;u_1),\ug(\cdot;u_2), W )$ so as to satisfy \eqref{ineq:ergodicity:P(ell(k+1)=k+1|ell(k)=infinity):Girsanov}. To this end, we recall that since the discrete metric $\mathbf{1}\{u\neq v\}$ is convex, the infimum in definition \eqref{form:W_d} for $\W_{\TV}$ is achieved \cite{villani2008optimal}. In light of  \cite[Proposition 2.8]{debussche2005ergodicity} (see also \cite[lemma 2.7]{mattingly2002exponential}), there exists a coupling $(W',\What')$ of $(W,\What)$ such that $(\Xg_N(\,\cdot\,;u_2,W'),\Xhat_N')$ is an optimal coupling for $(\Xg_N(\,\cdot\,;u_2,W),\Xhat_N)$. In other words,
\begin{align*}
&\P\Big( \exists t\in[0,t_1\mi\tau_1\mi\tau_2],\Xhat_N'(t)\neq \Xg_N(t;u_2,W') \Big)\\
&\quad = \W_{\TV}\Big( \Law\big(\Xhat_N\big) ,\Law\big(\Xg_N(\,\cdot\,;u_2)\big) \Big) \\
&\quad \le C(N)\sqrt{\Big(t_1+\frac{1}{t_1}+1\Big)r_1^4+C_4t_1+\rho_1\sqrt{t_1}  }.
\end{align*}
Since $\big(\ug(\cdot;u_1),\ug(\cdot;u_2), W' )$ is a coupling of $\big(\ug(\cdot;u_1),\ug(\cdot;u_2), W )$, from \eqref{ineq:P(X^gamma(u_1)=X^gamma(u_2))>1/2-P(X^gamma(u_1).neq.X^gamma(u_2))}, we deduce
\begin{align*}
&\P\Big( \big\{\Xg_N(t_1;u_1)= \Xg_N(t_1;u_2)\big\}\cap \big\{ E_4(t_1;u_1)-C_4 t_1\le \Phi(u_1)^4+\rho_1\sqrt{t_1}\big\} \nonumber\\
&\hspace{4cm}\cap \big\{E_4(t_1;u_2)-C_4 t_1\le \Phi(u_2)^4+\rho_1\sqrt{t_1}\big\} \Big) \nonumber \\
&=\P\Big( \big\{\Xg_N(t_1;u_1,W')= \Xg_N(t_1;u_2,W')\big\}\cap \big\{ E_4(t_1;u_1,W' )-C_4 t_1\le \Phi(u_1)^4+\rho_1\sqrt{t_1}\big\} \nonumber\\
&\hspace{4cm}\cap \big\{E_4(t_1;u_2,W' )-C_4 t_1\le \Phi(u_2)^4+\rho_1\sqrt{t_1}\big\} \Big) \nonumber \\
&\ge \frac{1}{2}- \P\Big( \exists t\in[0,t_1\mi \tau_1\mi \tau_2],\Xhat_N'(t)\neq \Xg_N(t;u_2,W' ) \Big),
\end{align*}
whence
\begin{align*}
&\P\Big( \big\{\Xg_N(t_1;u_1)= \Xg_N(t_1;u_2)\big\}\cap \big\{ E_4(t_1;u_1)-C_4 t_1\le \Phi(u_1)^4+\rho_1\sqrt{t_1}\big\} \nonumber\\
&\hspace{4cm}\cap \big\{E_4(t_1;u_2)-C_4 t_1\le \Phi(u_2)^4+\rho_1\sqrt{t_1}\big\} \Big) \nonumber \\
&\ge 1-C(N)\sqrt{\Big(t_1+\frac{1}{t_1}+1\Big)r_1^4+C_4t_1+\rho_1\sqrt{t_1}  }.
\end{align*}
Now, let $t_1=r_1=r_1(N,\beta)\in (0,1)$ be sufficiently small such that (recalling the choice of $\rho_1$ from \eqref{cond:t_1_r_1:a})
\begin{align}  \label{cond:t_1_r_1:b}
&C(N)\sqrt{\Big(t_1+\frac{1}{t_1}+1\Big)r_1^4+C_4t_1+\rho_1\sqrt{t_1}  }  \notag \\
& = C(N)\sqrt{\Big(t_1+\frac{1}{t_1}+1\Big)r_1^4+C_4t_1+8K_{4,1}(r_1^4+1)\sqrt{t_1}  } \le \frac{1}{4}, 
\end{align}
and that
\begin{align} \label{cond:t_1_r_1:c}
  C_4t_1+r_1^4+\rho_1\sqrt{t_1}=C_4t_1+r_1^4+8K_{4,1}(r_1^4+1)\sqrt{t_1} <\Big(\frac{\beta }{2}\Big)^4.
\end{align}
Since $C_4$ and $K_{4,1}$ do not depend on $C(N)$, for the sake of simplicity, we pick
\begin{align} \label{cond:t_1_r_1:d}
t_1=r_1=\beta^{10},
\end{align}
and take $\beta=\beta(N)$ sufficiently small to ensure \eqref{cond:t_1_r_1:b}-\eqref{cond:t_1_r_1:c} hold. 

With these choices of $t_1$ and $r_1$, recalling from \eqref{form:E_n(t)} that $E_4\ge \Phi^4$, we arrive at the bound
\begin{align*}
&\P\big( \Xg_N(t_1;u_1)= \Xg_N(t_1;u_2), \Phi\big(\ug(t_1;u_1)\big)+ \Phi\big(\ug(t_1;u_2)\big)\le \beta \big)\\
&\ge  \P\Big( \big\{\Xg_N(t_1;u_1)= \Xg_N(t_1;u_2)\big\}\cap \big\{ E_4(t_1;u_1)-C_4 t_1\le \Phi(u_1)^4+\rho_1\sqrt{t_1}\big\} \nonumber\\
&\hspace{4cm}\cap \big\{E_4(t_1;u_2)-C_4 t_1\le \Phi(u_2)^4+\rho_1\sqrt{t_1}\big\} \Big)\\
&\ge \frac{1}{2}.
\end{align*}
This produces \eqref{ineq:ergodicity:P(ell(k+1)=k+1|ell(k)=infinity):Girsanov}, thereby finishing the proof.

\end{proof}

Having established the auxiliary results from Lemma \ref{lem:irreducibility:Ginzburg-Landau} and Lemma \ref{lem:ergodicity:P(ell(k+1)=k+1|ell(k)=infinity):Girsanov}, we now provide the proof of Proposition \ref{prop:ergodicity:P(ell(k+1)=k+1|ell(k)=infinity)}.

\begin{proof}[Proof of Proposition \ref{prop:ergodicity:P(ell(k+1)=k+1|ell(k)=infinity)}] Let $N\ge 1$, $R>0$ be given and let $\beta\in(0,\beta_1)$ where $\beta_1=\beta_1(N)>0$ is the constant from Lemma \ref{lem:ergodicity:P(ell(k+1)=k+1|ell(k)=infinity):Girsanov}. By the Markov property, it suffices to establish the existence of $T_1=T_1(N,R,\beta)$ independent of $\gamma$ such that
\begin{align} \label{ineq:P(ell(k+1)=k+1|ell(k)=infinity)}
&\P\big(  \Xg_N(T_1;u_1)= \Xg_N(T_1;u_2), \Phi(\ug(T_1;u_1))+\Phi(\ug(T;u_2))\le \beta  \big|\Phi(u_1)+ \Phi(u_2)\le R\big) \notag \\
&\hspace{1cm}\ge \varepsilon_1,
\end{align}
holds for some positive $\varepsilon_1=\varepsilon_1(N,R,\beta,T_1)$ independent of $\gamma$.

To see this,  let $t_1=t_1(N,\beta)$ and $r_1=r_1(N,\beta)$ be the constants as in Lemma \ref{lem:ergodicity:P(ell(k+1)=k+1|ell(k)=infinity):Girsanov}, cf. the choice \eqref{cond:t_1_r_1:d}. In view of Lemma \ref{lem:irreducibility:Ginzburg-Landau}, there exists a positive time $T_*=T_*(R,r_1)$ such that for all $t\ge T_*$,
\begin{align*}
\P\big( \Phi(\ug(t;u_1))+ \Phi(\ug(t;u_2))\le r_1 \big| \Phi(u_1)+ \Phi(u_2)\le R \big) \ge \varepsilon_*,
\end{align*} 
where the positive constant $\varepsilon_*=\varepsilon_*(t,R,r_1)$ is independent of $u_1,u_2$ and $\gamma$. In particular, $T_*$ is given by \eqref{cond:T_*}. Now, for $T_1=T_*+t_1$, we employ Lemma \ref{lem:irreducibility:Ginzburg-Landau} and Lemma \ref{lem:ergodicity:P(ell(k+1)=k+1|ell(k)=infinity):Girsanov} to infer that
\begin{align*}
& \P\big(  \Xg_N(T_1;u_1)= \Xg_N(T_1;u_2), \Phi(\ug(T_1;u_1))+\Phi(\ug(T_1;u_2))\le \beta  \big|\Phi(u_1)+ \Phi(u_2)\le R\big) \\
&\ge   \P\Big( \big\{ \Xg_N\big(t_1;\ug(T_*;u_1)\big)= \Xg_N\big(t_1;\ug(T_*;u_2)\big)\big\}\\
&\qquad\qquad\cap\big\{ \Phi\big(t_1;\ug(T_*;u_1)\big)+ \Phi\big(t_1;\ug(T_*;u_2)\big)\le \beta \big\} \\
&\qquad\qquad\qquad\big|\Phi\big(\ug(T_*;u_1)\big)+ \Phi\big(\ug(T_*;u_2)\big)\le r_1,\Phi(u_1)+ \Phi(u_2)\le R  \Big)\\
&\times \P\big( \Phi(\ug(T_*;u_1))+ \Phi(\ug(T_*;u_2))\le r_1 \big| \Phi(u_1)+ \Phi(u_2)\le R \big)\\
& \ge \frac{1}{2}\varepsilon_*=: \varepsilon_1.
\end{align*}
This produces \eqref{ineq:P(ell(k+1)=k+1|ell(k)=infinity)}, as claimed.

\end{proof}

\begin{remark} \label{rem:T_1} From the proof of Proposition \ref{prop:ergodicity:P(ell(k+1)=k+1|ell(k)=infinity)} together with choice of $T_*$ as in \eqref{cond:T_*} and $t_1\in(0,1)$, we see that $T_1=T_*+t_1$ satisfies
\begin{align} \label{cond:T_1}
\frac{1}{\alpha}\log\Big(\frac{2(cR+r_1)}{r_1}\Big) <T_1< \frac{1}{\alpha}\log\Big(\frac{2(cR+r_1)}{r_1}\Big)+1.
\end{align}
In the above, $c$ is a positive constant independent of $\beta,\alpha,R,r_1$. In view of the choice \eqref{cond:t_1_r_1:d}, i.e., $r_1=\beta^{10}$, we deduce that for every fixed $R$, $T_1$ has the order of $|\log\beta|$ as $\beta\to 0$. This will be employed in the argument of Proposition \ref{prop:ergodicity:P(ell(k+1)=l|ell(k)=l)} presented below.

\end{remark}

Finally, we establish Proposition \ref{prop:ergodicity:P(ell(k+1)=l|ell(k)=l)}, which together with the above auxiliary results ultimately concludes Theorem \ref{thm:ergodicity:Ginzburg_Landau}. In particular, the argument will rely on the estimates on tail probabilities from Lemma \ref{lem:moment:P[E_k(t)>Phi(0)^n+rho(Phi(0)^2n+t)]}, part 2.

\begin{proof}[Proof of Proposition \ref{prop:ergodicity:P(ell(k+1)=l|ell(k)=l)}]
By the Markov property, we may assume that $l=0$. Setting $T=T_1$ where $T_1$ is the time constant from Proposition \ref{prop:ergodicity:P(ell(k+1)=k+1|ell(k)=infinity)}, similar to the proof of Proposition \ref{prop:ergodicity:P(ell(k+1)=k+1|ell(k)=infinity)}, we aim to construct a coupling of $\big(\ug(\,\cdot\,;u_1),\ug(\,\cdot\,;u_2),W\big)$ such that \eqref{ineq:ergodicity:P(ell(k+1)=l|ell(k)=l)} holds. To achieve this effect, letting $\rho_2>0$ be given and be chosen later, we introduce the following stopping times
\begin{align*}
\taut_i =\inf\big\{t\ge 0: E_4(k T +t;u_i)\ge \theta+\beta^4+C_4 (kT+t) \big \},\quad i=1,2.
\end{align*}
and
\begin{align*}
\tau_3= \inf\Big\{t \ge 0: \int_{kT}^{T+t\mi\taut_1\mi \taut_2}\hspace{-1cm} \big[1+\Phi(u_1(s))^4+\Phi(u_2(s))^4\big]&\|u_1(s)-u_2(s)\|^2_{H^1}\d s  > \rho_2 e^{-\frac{1}{4}\alpha kT} \Big\}.
\end{align*}
With the above stopping times, for $N\ge N_1$, from Definition \ref{def:ell_beta} of $\ell_{\theta,\beta}^N$, observe that
\begin{align*}
&\big\{\ell_{\theta,\beta}^N(k+1)\neq 0\big\}\\
&\subset\big\{\exists t\in [0,T]: \Xg_N(kT+t;u_1)\neq \Xg_N(kT+t;u_2) \big\}\cup \{\taut_1<T\}\cup \{\taut_2<T\}\\
& =\big\{\exists t\in [0,T\mi\taut_1\mi\taut_2]: \Xg_N(kT+t;u_1)\neq \Xg_N(kT+t;u_2) \big\}\cup \{\taut_1<T\}\cup \{\taut_2<T\}.
\end{align*}
In particular, 
\begin{align*}
&\big\{\exists t\in [0,T\mi\taut_1\mi\taut_2]: \Xg_N(kT+t;u_1)\neq \Xg_N(kT+t;u_2) \big\}\\
&=\Big(\big\{\exists t\in [0,T\mi\taut_1\mi\taut_2]: \Xg_N(kT+t;u_1)\neq \Xg_N(kT+t;u_2) \big\}\cap \{\tau_3>T\mi\taut_1\mi\taut_2\}\Big)\\
&\qquad \cup\Big( \big\{\exists t\in [0,T\mi\taut_1\mi\taut_2]: \Xg_N(kT+t;u_1)\neq \Xg_N(kT+t;u_2) \big\}\cap \{\tau_3\le T\mi\taut_1\mi\taut_2\}\Big)\\
&\subset  \big\{\exists t\in [0,T\mi\taut_1\mi\taut_2\mi\tau_3]: \Xg_N(kT+t;u_1)\neq \Xg_N(kT+t;u_2) \big\}\\
&\qquad \cup \Big( \big\{\forall t\in [0,T\mi\taut_1\mi\taut_2\mi\tau_3]: \Xg_N(kT+t;u_1)= \Xg_N(kT+t;u_2) \big\}\\
&\hspace{5cm}\cap  \{\tau_3\le T\mi\taut_1\mi\taut_2\}\Big).
\end{align*}
Recalling the stopping time $\tau_N$ as in Lemma \ref{lem:Foias-Prodi:Ginzburg-Landau}, we note that
\begin{align*}
&\big\{\forall t\in [0,T\mi\taut_1\mi\taut_2\mi\tau_3]: \Xg_N(kT+t;u_1)= \Xg_N(kT+t;u_2) \big\}\\
&\hspace{1cm}\cap  \{\tau_3\le T\mi\taut_1\mi\taut_2\}\cap \{\ell_{\theta,\beta}^N(k)=0\}\\
&\subset \big\{\forall t\in [0,(kT+T\mi\taut_1\mi\taut_2\mi\tau_3)\mi\tau_N], E_4(t;u_i) \le \theta+\beta^4+C_4 t,\, i=1,2 \big\}\\
&\qquad \cap \Big\{\int_{kT}^{(kT+T\mi\taut_1\mi \taut_2\mi \tau_3)\mi\tau_N}\hspace{-1cm} \big[1+\Phi(u_1(s))^4+\Phi(u_2(s))^4\big]\|u_1(s)-u_2(s)\|^2_{H^1}\d s \ge \rho_2 e^{-\frac{1}{4}\alpha kT} \Big\} \\
& \subset \Big\{ \big[1+2\theta+2\beta^4+2C_4(k+1)T\big]\int_{kT}^{(kT+T\mi\taut_1\mi \taut_2\mi \tau_3)\mi\tau_N}\hspace{-1cm} \|u_1(s)-u_2(s)\|^2_{H^1}\d s \ge \rho_2 e^{-\frac{1}{4}\alpha kT} \Big\}.
\end{align*}
Furthermore, for all $t\in [kT,(kT+T\mi\taut_1\mi \taut_2\mi \tau_3)\mi\tau_N]$, in view of \eqref{ineq:J(t)>e^(alpha.t)}, we have
\begin{align*}
&\int_{kT}^{(k+1)T}\hspace{-0.5cm}\exp\Big\{ \alpha (t\mi\tau_N)-\frac{c_*}{\alpha_N^{1/8}}\int_0^{t\mi\tau_N}\hspace{-0.5cm}2\alpha\Phi(u_1(s))^4+2\alpha\Phi(u_2(s))^4+1 \d s\Big\} J(t\mi\tau_N) \d t\\
&\ge \int_{kT}^{(kT+T\mi\taut_1\mi \taut_2\mi \tau_3)\mi\tau_N}\hspace{-0.5cm}\exp\Big\{\frac{1}{2} \alpha t-2c_*\frac{\theta+\beta^4}{\alpha_N^{1/8}}\Big\}\|u_1(t)-u_2(t)\|^2_{H^1} \d t   \\
& \ge \rho_2\exp\Big\{\frac{1}{4}\alpha kT -2c_*\frac{\theta+\beta^4}{\alpha_N^{1/8}}\Big\}\big[1+2\theta+2\beta^4+2C_4(k+1)T\big]^{-1}.
\end{align*}
We deduce from the above implications that
\begin{align} \label{ineq:ell(k+1).ell(k)|ell(l)=l}
& \big\{\ell_{\theta,\beta}^N(k+1)\neq 0,\ell_{\theta,\beta}^N(k)=0|\ell_{\theta,\beta}^N(0)=0 \big\} \\
&\subset \{\taut_1<T|\ell_{\theta,\beta}^N(0)=0\}\cup \{\taut_2<T|\ell_{\theta,\beta}^N(0)=0\} \nonumber \\
&\qquad \cup \big\{\exists t\in [0,T\mi\taut_1\mi\taut_2\mi\tau_3]: \Xg_N(kT+t;u_1)\neq \Xg_N(kT+t;u_2), \nonumber \\
&\hspace{2cm}\Xg_N(kT;u_1)= \Xg_N(kT;u_2) | \ell_{\theta,\beta}^N(0)=0  \big\}\nonumber \\
&\qquad \cup  \Big\{\int_{kT}^{(k+1)T}\hspace{-0.5cm}\exp\Big\{ \alpha (t\mi\tau_N)-\frac{c_*}{\alpha_N^{1/8}}\int_0^{t\mi\tau_N}\hspace{-0.5cm}2\alpha\Phi(u_1(s))^4+2\alpha\Phi(u_2(s))^4+1 \d s\Big\} J(t\mi\tau_N) \d t  \nonumber  \\
& \hspace{2cm}\ge \rho_2\exp\Big\{\frac{1}{4}\alpha kT -2c_*\frac{\theta+\beta^4}{\alpha_N^{1/8}}\Big\}\big[1+2\theta+2\beta^4+2C_4(k+1)T\big]^{-1}\Big|\ell_{\theta,\beta}^N(0)=0 \Big\}.\notag 
\end{align}

With regard to $\taut_i$, $i=1,2$, we have
\begin{align*}
\P(\taut_i<T|\ell_{\theta,\beta}^N(0)=0)& = \P\Big( \sup_{t\in[0,T]}E_4(kT+t;u_i)-C_4 (kT+t)\ge \theta+\beta^4 |\ell_{\theta,\beta}^N(0)=0\Big)\\
 &\le \P\Big( \sup_{t\in[0,T] }E_4(kT+t;u_i)-C_4 (kT+t)\ge \theta+\Phi(u_i)^4|\ell_{\theta,\beta}^N(0)=0\Big).
\end{align*}
On the one hand, when $k=0$, we employ Lemma \ref{lem:moment:P[E_k(t)>Phi(0)^n+rho(Phi(0)^2n+t)]}, cf.~\eqref{ineq:moment:P[sup_[0,T]E_k(t)>Phi(0)^n+rho(Phi(0)^2n+t)} with $n=4$ and $\rho = \theta/\sqrt{T}$ to infer
\begin{align*}
   \P(\taut_i<T|\ell_{\theta,\beta}^N(0)=0) &\le  \P\Big( \sup_{t\in[0,T]}E_4(t;u_i)-(C_4-1) t\ge \frac{\theta}{\sqrt{T}}\cdot\sqrt{T}+\Phi(u_i)^4 |\ell_{\theta,\beta}^N(0)=0\Big)\\
   &\le  K_{4,q}\frac{T^q}{\theta^{2q}} \E\big[ \Phi(u_i)^{8q}+1|\ell_{\theta,\beta}^N(0)=0\big]  \\
&\le  K_{4,q}\frac{T^q}{\theta^{2q}}(\beta^{8q}+1) \le 2K_{4,q}\frac{T^q}{\theta^{2q}},
\end{align*}
where in the last implication, we simply invoked the choice $\beta<1$. Picking $\theta>T=T_1$, we deduce in this case that
\begin{align*}
    \P(\taut_i<T|\ell_{\theta,\beta}^N(0)=0) \le K_{4,q}\frac{2}{\theta^{q}}.
\end{align*}
On the other hand, when $k\ge 1$, we have
\begin{align*}
   \P(\taut_i<T|\ell_{\theta,\beta}^N(0)=0) &\le \P\Big( \sup_{t\in[0,\infty) }E_4(kT+t;u_i)-C_4 (kT+t)\ge \theta+ \Phi(u_i)^4|\ell_{\theta,\beta}^N(0)=0\Big).
\end{align*}
Recalling from Proposition \ref{prop:ergodicity:P(ell(k+1)=k+1|ell(k)=infinity)} that $T=T_1>1$, Lemma \ref{lem:moment:P[E_k(t)>Phi(0)^n+rho(Phi(0)^2n+t)]}, cf.~\eqref{ineq:moment:P[sup_[T,infty)E_k(t)>Phi(0)^n+Phi(0)^2n+1+rho)} (with $n=4$ and $\rho=\theta$) implies that
\begin{align*}
\P(\taut_i<T|\ell_{\theta,\beta}^N(0)=0)&\le K_{4,q}\frac{1}{(\theta+kT)^{q}} \E\big[ \Phi(u_i)^{8q}+1|\ell_{\theta,\beta}^N(0)=0\big] \nonumber \\
&\le  K_{4,q}\frac{\beta^{8q}+1}{(\theta+ kT)^{q}}\le K_{4,q}\frac{2}{(\theta+ kT)^{q}},\quad i=1,2.
\end{align*}
So, from both cases, we obtain
\begin{align} \label{ineq:ell(k+1).ell(k)|ell(l)=l:tau_i<T}
\P(\taut_i<T|\ell_{\theta,\beta}^N(0)=0)&\le K_{4,q}\frac{2}{(\theta+ kT)^{q}},\quad i=1,2.
\end{align}

Concerning the third conditional event on the right-hand side of \eqref{ineq:ell(k+1).ell(k)|ell(l)=l}, we will employ an argument similar to the proof of Lemma \ref{lem:ergodicity:P(ell(k+1)=k+1|ell(k)=infinity):Girsanov} to deduce an upper bound in probability. To see this, note that the equation \eqref{eqn:Ginzburg-Landau:X_N} for $\Xg_N$ can be recast as
\begin{align*}
\d \Xg_N(t;u_1)&=  \big(-\gamma A -\i A- \alpha\big)\Xg_N(t;u_1)  \d t+Q\d \tilde{W}(t)\\
&\qquad +\i P_N\big[\big|\Xg_N(t;u_1)+\Yg_N(t;u_2)\big|^2\big(\Xg_N(t;u_1)+\Yg_N(t;u_2) \big)\big]\d t ,
\end{align*}
where
\begin{align*}
\d \tilde{W}(t) =\d W(t)+ Q^{-1}F_2(t)\d t,
\end{align*}
and
\begin{align*}
F_2(t)&=\i P_N\Big[ \big|\Xg_N(t;u_1)+\Yg_N(t;u_1)\big|^2\big(\Xg_N(t;u_1)+\Yg_N(t;u_1) \big)\\
&\hspace{2cm}-\big|\Xg_N(t;u_1)+\Yg_N(t;u_2)\big|^2\big(\Xg_N(t;u_1)+\Yg_N(t;u_2) \big)\Big].
\end{align*}
Observe that 
\begin{align*}
\|F_2(t)\|^2_H\le \big(1+\Phi(u_1(t))^4+\Phi(u_2(t))^4\big)\|u_1(t)-u_2(t)\|^2_{H^1},
\end{align*}
whence
\begin{align*}
\E\exp\Big\{\frac{1}{2}\int_{kT}^{kT+T\mi\taut_1\mi\taut_2\mi\tau_3} \hspace{-1cm}\|Q^{-1}F_2(s)\|^2_H\d s\Big\} \le \exp\Big\{\frac{1}{2} \|Q^{-1}\|^2_{L(H)}\rho_2 e^{-\frac{1}{4}\alpha kT}\Big\}.
\end{align*}
This verifies Novikov's condition, and thus establishes the equivalence in law between $(\Xg_N(\,\cdot\,;u_1),\tilde{W})$ and $(\Xg_N(\,\cdot\,;u_2,W)$ on $[kT,kT+T\mi\taut_1\mi\taut_2\mi\tau_3]$. Now, by the uniqueness of weak solutions, it holds that
\begin{align*}
&\P\Big(\exists t\in [0,T\mi\taut_1\mi\taut_2\mi\tau_3]: \Xg_N(kT+t;u_1)\neq \Xg_N(kT+t;u_2), \nonumber \\
&\hspace{2cm}\Xg_N(kT;u_1)= \Xg_N(kT;u_2) | \ell_{\theta,\beta}^N(0)=0  \Big)\\
&\le \P\Big( W|_{[kT,kT+T\mi\taut_1\mi\taut_2\mi\tau_3]}\neq \tilde{W}|_{[kT,kT+T\mi\taut_1\mi\taut_2\mi\tau_3]}\big| \ell_{\theta,\beta}^N(0)=0  \Big).
\end{align*}
We note that up to this point, we have not chosen a coupling of $\big(\ug(\cdot;u_1),\ug(\cdot;u_2), W )$ so as to establish ~\eqref{ineq:ergodicity:P(ell(k+1)=l|ell(k)=l)}. Instead of doing so directly, we will pick an optimal coupling on $[kT,kT+T\mi\taut_1\mi\taut_2\mi\tau_3]$ for $\big((W|\ell_{\theta,\beta}^N(0)=0),(\tilde{W}|\ell_{\theta,\beta}^N(0)=0)\big)$. In turn, similar to the proof of Lemma \ref{lem:ergodicity:P(ell(k+1)=k+1|ell(k)=infinity):Girsanov}, this allows for employing Pinsker inequality and \cite[Theorem A.2]{butkovsky2020generalized} to deduce
\begin{align*}
&\P\Big( W_{[kT,kT+T\mi\taut_1\mi\taut_2\mi\tau_3]}\neq \tilde{W}_{[kT,kT+T\mi\taut_1\mi\taut_2\mi\tau_3]}\big| \ell_{\theta,\beta}^N(0)=0  \Big)\\
& =\W_{\TV}\Big( \Law\big(W_{[kT,kT+T\mi\taut_1\mi\taut_2\mi\tau_3]}|\ell_{\theta,\beta}^N(0)=0\big) ,\Law\big( \tilde{W}_{[kT,kT+T\mi\taut_1\mi\taut_2\mi\tau_3]}|\ell_{\theta,\beta}^N(0)=0\big) \Big) \\
&\le \sqrt{\frac{1}{2}D_{\KL}\Big( \Law\big(W_{[kT,kT+T\mi\taut_1\mi\taut_2\mi\tau_3]}|\ell_{\theta,\beta}^N(0)=0\big) ,\Law\big( \tilde{W}_{[kT,kT+T\mi\taut_1\mi\taut_2\mi\tau_3]}|\ell_{\theta,\beta}^N(0)=0\big) \Big) }\\
 & \le \frac{1}{2}\sqrt{\E\Big[\int_{kT}^{kT+T\mi\taut_1\mi\taut_2\mi\tau_3} \hspace{-1cm}\|Q^{-1}F_2(s)\|^2_H\d s\Big|\ell_{\theta,\beta}^N(0)=0\Big]}\\
 & \le \frac{1}{2}\sqrt{ \|Q^{-1}\|^2_{L(H)}\rho_2 e^{-\frac{1}{4}\alpha kT}}\\
 & \le C(N)\sqrt{\rho_2}e^{-\frac{1}{8}\alpha kT}.
\end{align*}
As a consequence, we obtain
\begin{align} \label{ineq:ell(k+1).ell(k)|ell(l)=l:b}
&\P\Big(\exists t\in [0,T\mi\taut_1\mi\taut_2\mi\tau_3]: \Xg_N(kT+t;u_1)\neq \Xg_N(kT+t;u_2), \nonumber \\
&\hspace{2cm}\Xg_N(kT;u_1)= \Xg_N(kT;u_2) | \ell_{\theta,\beta}^N(0)=0  \Big)\le C(N)\sqrt{\rho_2}e^{-\frac{1}{8}\alpha kT}.
\end{align}

Turning to the last event on the right-hand side of \eqref{ineq:ell(k+1).ell(k)|ell(l)=l}, we invoke Lemma \ref{lem:Foias-Prodi:Ginzburg-Landau} to infer
\begin{align*}
&\P\Big(\int_{kT}^{(k+1)T}\hspace{-0.5cm}\exp\Big\{ \alpha (t\mi\tau_N)-\frac{c_*}{\alpha_N^{1/8}}\int_0^{t\mi\tau_N}\hspace{-0.5cm}2\alpha\Phi(u_1(s))^4+2\alpha\Phi(u_2(s))^4+1 \d s\Big\} J(t\mi\tau_N) \d t  \nonumber  \\
& \hspace{1cm}\ge \rho_2\exp\Big\{\frac{1}{4}\alpha kT -2c_*\frac{\beta^4}{\alpha_N^{1/8}}\Big\}\big[1+2\theta+2\beta^4+2C_4(k+1)T\big]^{-1}\Big|\ell_{\theta,\beta}^N(0)=0 \Big)\\
&\le \int_{kT}^{(k+1)T}\hspace{-0.5cm}\E[J(0)|\ell_{\theta,\beta}^N(0)=0]\d t\cdot\frac{1}{\rho_2}\exp\Big\{-\frac{1}{4}\alpha kT +2c_*\frac{\beta^4}{\alpha_N^{1/8}}\Big\}\big[1+2\theta+2\beta^4+2C_4(k+1)T\big]\\
&\le \tilde{c}(\beta+\beta^2)\frac{T}{\rho_2}\exp\Big\{-\frac{1}{4}\alpha kT +2c_*\frac{\beta^4}{\alpha_N^{1/8}}\Big\}\big[1+2\theta+2\beta^4+2C_4(k+1)T\big].
\end{align*}
Taking into account of \eqref{ineq:ell(k+1).ell(k)|ell(l)=l}, \eqref{ineq:ell(k+1).ell(k)|ell(l)=l:tau_i<T} and \eqref{ineq:ell(k+1).ell(k)|ell(l)=l:b}, the above estimate implies the bound for all $q\ge 2$
\begin{align*}
&\P\big(\ell_{\theta,\beta}^N(k+1)\neq 0,\ell_{\theta,\beta}^N(k)=0|\ell_{\theta,\beta}^N(0)=0 \big)\\
&\le \tilde{c}(\beta+\beta^2)\frac{T}{\rho_2}\exp\Big\{-\frac{1}{4}\alpha kT + 2c_*\frac{\beta^4}{\alpha_N^{1/8}}\Big\}\big[1+2\theta+2\beta^4+2C_4(k+1)T\big]\\
&\qquad+\frac{c}{(\theta+kT)^{2q}} + C(N)\sqrt{\rho_2}e^{-\frac{1}{8}\alpha kT}.
\end{align*}
For the sake of simplicity, we pick $\rho_2=\sqrt{\beta}$ and recall $\beta\in(0,1)$ to deduce
\begin{align} \label{ineq:ell(k+1).ell(k)|ell(l)}
&\P\big(\ell_{\theta,\beta}^N(k+1)\neq 0,\ell_{\theta,\beta}^N(k)=0|\ell_{\theta,\beta}^N(0)=0 \big) \notag \\
&\le c\sqrt{\beta}e^{-\frac{1}{4}\alpha kT}\big[\theta+(k+1)T\big]T  +\frac{c}{(\theta+kT)^{2q}} + C(N)\beta^{\frac{1}{4}} e^{-\frac{1}{8}\alpha kT}.
\end{align}
In the above, we emphasize that $c$ is an absolute constant independent of $\beta,\theta, T,k$ and $C(N)$. Now, we proceed to tune $\theta$, $\beta$ and $T$ appropriately so as to produce
\begin{align} \label{ineq:ell(k+1).ell(k)|ell(l)=l<1/4(1+kT)^q}
\P\big(\ell_{\theta,\beta}^N(k+1)\neq 0,\ell_{\theta,\beta}^N(k)=0|\ell_{\theta,\beta}^N(0)=0 \big)\le \frac{1}{4}(1+kT)^{-q}.
\end{align}
To see this, on the right hand side of \eqref{ineq:ell(k+1).ell(k)|ell(l)}, we first pick $\theta$ sufficiently large such that
\begin{align*}
\frac{c}{(\theta+kT)^{2q}}<\frac{c}{\theta^{q-1}}\cdot \frac{1}{(\theta+kT)^{q+1}} <\frac{1}{10(1+kT)^{q+1}}.
\end{align*}
Next, letting $T$ be the same constant $T_1$ from Proposition \ref{prop:ergodicity:P(ell(k+1)=k+1|ell(k)=infinity)}, in view of Remark \ref{rem:T_1}, $T$ has the order of $|\log\beta|$. Since \begin{align*}
e^{-\frac{1}{8}\alpha kT}\le \frac{c}{(1+kT)^{q+2}}, 
\end{align*}
we pick $\beta$ sufficiently small such that
\begin{align*}
& \sqrt{\beta}e^{-\frac{1}{4}\alpha kT}\big[\theta+(k+1)T\big]T  + C(N)\beta^{\frac{1}{4}} e^{-\frac{1}{8}\alpha kT}\\
&\le  \Big(\sqrt{\beta}\big[\theta+|\log \beta|\big]|\log \beta|  + C(N)\beta^{\frac{1}{4}}\Big)\frac{c}{(1+kT)^{q+1}} < \frac{1}{10(1+kT)^{q+1}}.
\end{align*}
Altogether with \eqref{ineq:ell(k+1).ell(k)|ell(l)}, we establish \eqref{ineq:ell(k+1).ell(k)|ell(l)=l<1/4(1+kT)^q}, as claimed.

Turning back to \eqref{ineq:ergodicity:P(ell(k+1)=l|ell(k)=l)}, we invoke \eqref{ineq:ell(k+1).ell(k)|ell(l)=l<1/4(1+kT)^q} to infer
\begin{align*}
\P\big(\ell_{\theta,\beta}^N(k)\neq 0|\ell_{\theta,\beta}^N(0)=0 \big)& \le \sum_{j=0}^{k-1}\P\big(\ell_{\theta,\beta}^N(j+1)\neq 0,\ell_{\theta,\beta}^N(j)=0|\ell_{\theta,\beta}^N(0)=0 \big)\\
&\le \frac{1}{4}+\frac{1}{4}\sum_{k=1}^\infty\frac{1}{(1+kT)^{q+1} }\le \frac{1}{4}+ \frac{C(q)}{T^q} . 
\end{align*}
We once again invoke the fact that $T$ has the order of $|\log \beta|$, and shrink $\beta$ further to zero if necessary to infer that $C(q)/T^q<1/4$, whence
\begin{align*}
\P\big(\ell_{\theta,\beta}^N(k)\neq 0|\ell_{\theta,\beta}^N(0)=0 \big)& \le \frac{1}{2}. 
\end{align*}
As a consequence, we combine the above estimate with \eqref{ineq:ell(k+1).ell(k)|ell(l)=l<1/4(1+kT)^q} to arrive at the bound
\begin{align*}
\P\big(\ell_{\theta,\beta}^N(k+1)\neq 0|\ell_{\theta,\beta}^N(k)=0 \big)& = \P\big(\ell_{\theta,\beta}^N(k+1)\neq 0|\ell_{\theta,\beta}^N(k)=0,\ell_{\theta,\beta}^N(0)=0 \big)\\
&=\frac{\P\big(\ell_{\theta,\beta}^N(k+1)\neq 0,\ell_{\theta,\beta}^N(k)=0|\ell_{\theta,\beta}^N(0)=0 \big)}{\P\big(\ell_{\theta,\beta}^N(k)=0|\ell_{\theta,\beta}^N(0)=0 \big)}\\
&\le \frac{1}{2}(1+kT)^{-q}.
\end{align*}
This produces \eqref{ineq:ergodicity:P(ell(k+1)=l|ell(k)=l)}, thereby finishing the proof.

\end{proof}



\section{Inviscid limit as $\gamma \to 0$} \label{sec:gamma->0}

In this section, we establish the validity of the inviscid limit $\gamma\to 0$ for the invariant probability measures as well as for the infinite time horizon. 

\subsection{Proof of Theorem \ref{thm:gamma->0:Wasserstein:nu^gamma-nu^0}} \label{sec:gamma->0:nu^gamma->nu^0}

Concerning the convergence of $\nu^\gamma$ toward $\nu^0$, we will draw upon the framework of \cite{
cerrai2006smoluchowski,cerrai2006smoluchowski2, cerrai2020convergence,nguyen2023small} dealing with similar issue for stochastic wave equations. For the convenience of the reader, we briefly review the argument, which essentially consists of two main steps. We first establish a convergence of $\ug(t)$ toward $u(t)$ on any finite time window. This is presented in Proposition \ref{prop:gamma->0:|u^gamma-u|:[0,T]}. Then, we combine with the fact that $\nu^0$ satisfies a polynomial mixing rate, cf. Theorem \ref{thm:ergodicity:Schrodinger}, to conclude Theorem \ref{thm:gamma->0:Wasserstein:nu^gamma-nu^0}.

We start by considering the solutions $\ug(t)$ and $u(t)$ and establish the inviscid limit on finite time windows. We state the result now, but defer its proof to the end of this subsection.

\begin{proposition} \label{prop:gamma->0:|u^gamma-u|:[0,T]}
For all $u_0\in L^2(\Omega; H^1)$, let $\ug(t)$ and $u(t)$ respectively be the solutions of \eqref{eqn:Ginzburg_Landau} and \eqref{eqn:Schrodinger} with initial condition $u_0$. Then, the following holds for all $n\ge 1$
\begin{align} \label{lim:gamma->0:|u^gamma-u|}
\E\Big[\sup_{t\in[0,T]}\|\ug(t)-u(t)\|_H\Big] \le \frac{ C}{|\log \gamma|^{\frac{n}{8}}}\big( \E[\Phi(u_0)^{2n}] +T+1\big)e^{2\textup{Tr}(AQQ^*)T},\quad T>0,
\end{align}
 for some positive constant $C=C(n)$ independent of $\gamma, T,Q$ and $u_0$.
\end{proposition}

We note that the exponent $n/8$ on the right-hand side of \eqref{lim:gamma->0:|u^gamma-u|} originates from the $\alpha^{1/8}$ factor in the Foias-Prodi estimate established in Lemma \ref{lem:Foias-Prodi:Ginzburg-Landau}. This factor, in turn, follows from the Sobolev exponents used in the proof of Lemma \ref{lem:Foias-Prodi:Ginzburg-Landau}. Next, we provide the uniform moment bound on $\nu^\gamma$ in the following auxiliary result, whose proof is also deferred to the end of this subsection.

\begin{lemma}  \label{lem:moment:nu_gamma}
Under the same hypotheses of Theorem \ref{thm:ergodicity:Ginzburg_Landau}, let $\nu^\gamma$ be the unique invariant probability measure of \eqref{eqn:Ginzburg_Landau}. Then, the followings hold
\begin{align} \label{ineq:moment:mu_gamma}
\sup_{\gamma\in[0,1]}\int_{H} e^{\xi\|u\|^2_H}\nu^\gamma(\textup{d} u) <\infty \quad\textup{and}\quad \sup_{\gamma\in[0,1]}\int_{H} \Phi(u)^n \nu^\gamma(\textup{d} u) <\infty.
\end{align}
for all $\xi>0$ sufficiently small and $n\ge 1$.
\end{lemma}

We note that the supremum over $\gamma\in[0,1]$ in Lemma \ref{lem:moment:nu_gamma} includes the end point $\gamma=0$ for consistency with Lemma \ref{lem:moment:nu^0}, which furnishes analogous moment bounds on $\nu^0$. Assuming that Proposition \ref{prop:gamma->0:|u^gamma-u|:[0,T]} and Lemma \ref{lem:moment:nu_gamma} hold, let us conclude Theorem \ref{thm:gamma->0:Wasserstein:nu^gamma-nu^0}. In turn, the result of Theorem \ref{thm:gamma->0:Wasserstein:nu^gamma-nu^0} will be combined with Theorem \ref{thm:ergodicity:Ginzburg_Landau} to establish Theorem \ref{thm:gamma->0:phi} in Section \ref{sec:gamma->0:phi}.

\begin{proof}[Proof of Theorem \ref{thm:gamma->0:Wasserstein:nu^gamma-nu^0}] Firstly, we prove the following analogue of \eqref{lim:gamma->0:Wasserstein:nu^gamma-nu^0} with respect to $\W_{d_0}$
\begin{align} \label{lim:gamma->0:Wasserstein:d_0}
\W_{d_0}(\nu^\gamma,\nu^0) \le \frac{C}{(\log |\log \gamma|)^q}, \quad \textup{as}\, \gamma\to 0.
\end{align}
To see this, we invoke the triangle inequality and obtain
\begin{align}
\W_{d_0}\big(\nu^\gamma,\nu^0  \big) &\le  \W_{d_0}\big(\nu^\gamma,P_t^0\nu^\gamma   \big)+ \W_{d_0}\big(P_t^0\nu^\gamma,\nu^0  \big) \nonumber \\
& = \W_{d_0}\big(P_t^\gamma\nu^\gamma,P_t^0\nu^\gamma   \big)+ \W_{d_0}\big(P_t^0\nu^\gamma,P_t^0\nu^0  \big), \label{ineq:W_H(nu^gamma,nu^0)}
\end{align}
where the last implication follows from the invariance property. 

With regard to the second term on the right-hand side of \eqref{ineq:W_H(nu^gamma,nu^0)}, we note that
\begin{align*}
d_0(u_1,u_2) =\|u_1-u_2\|_H\mi 1\le c\big( \|u_1-u_2\|_{H^1}\mi 1\big)=c\,d_1(u_1,u_2),
\end{align*}
implying $\W_{d_0}\le c\W_{d_1}$ by virtue of definition \eqref{form:W_d}. In turn, Lemma \ref{lem:ergodicity:Schrodinger} can be employed to deduce
\begin{align*}
 \W_{d_0}\big(P_t^0\nu^\gamma,P_t^0\nu^0  \big)&\le c\,\W_{d_1}\big(P_t^0\nu^\gamma,P_t^0\nu^0  \big)\\
& \le \frac{C}{(1+t)^q }\Big(1+\int_{H^1}\Psi(u)\nu^\gamma(\textup{d} u)+\int_{H^1}\Psi(u)\nu^0(\textup{d} u)\Big)\\
&\le \frac{C}{(1+t)^q }\Big(1+\int_{H^1}\Phi(u)\nu^\gamma(\textup{d} u)+\int_{H^1}\Psi(u)\nu^0(\textup{d} u)\Big).
\end{align*}
where we recall $\Psi$ and $\Phi$ are defined in \eqref{form:Psi} and \eqref{form:Phi}, respectively. Taking into account of Lemma \ref{lem:moment:nu_gamma} and Lemma \ref{lem:moment:nu^0}, we obtain the bound
\begin{align} \label{ineq:W_H(P_t.nu^gamma,P_t.nu^0)<(1+t)^-q}
\W_{d_0}\big(P_t^0\nu^\gamma,P_t^0\nu^0  \big)  & \le \frac{C}{(1+t)^q },
\end{align}
for some positive constant $C$ independent of $\gamma$ and $t$. 

Concerning the first term on the right-hand side of \eqref{ineq:W_H(nu^gamma,nu^0)}, we employ definition \eqref{form:W_d} again to see that
\begin{align*}
\W_{d_0}\big(P_t^\gamma\nu^\gamma,P_t^0\nu^\gamma   \big) & \le \E\big[\|u^\gamma(t)-u(t)\|_H\mi 1\big],
\end{align*}
where $u^\gamma(0)=u(0)\sim \nu^\gamma$. In view of Proposition \ref{prop:gamma->0:|u^gamma-u|:[0,T]}, we infer
\begin{align}
\W_{d_0 }\big(P_t^\gamma\nu^\gamma,P_t^0\nu^\gamma   \big) & \le \E\big[\|u^\gamma(t)-u(t)\|_H\big] \nonumber\\
&\le \frac{C}{|\log \gamma|^{\frac{n}{8}}}\Big( \int_{H^1}\Phi(u)^{2n}\nu^\gamma(\d u)+t+1\Big)e^{Ct} \nonumber \\
&\le \frac{e^{Ct}}{|\log \gamma|^{\frac{n}{8}}}, \label{ineq:W_H(P_t^gamma.nu^gamma,P_t^0.nu^gamma)}
\end{align}
where the last implication follows from \eqref{ineq:moment:mu_gamma}. In the above, we emphasize that $C=C(n)$ is independent of $\gamma$ and $t$.

Now, from \eqref{ineq:W_H(nu^gamma,nu^0)} together with \eqref{ineq:W_H(P_t.nu^gamma,P_t.nu^0)<(1+t)^-q} and \eqref{ineq:W_H(P_t^gamma.nu^gamma,P_t^0.nu^gamma)}, we arrive at the bound
\begin{align*}
\W_{d_0}\big(\nu^\gamma,\nu^0  \big) \le \frac{C}{(1+t)^q }+ \frac{e^{Ct}}{|\log \gamma|^{\frac{n}{8}}}.
\end{align*}
Picking $t$ satisfying
\begin{align*}
t= \frac{1}{C}\log |\log \gamma|,\quad \gamma\in (0,1/4),
\end{align*}
implies
\begin{align*}
\W_{d_0}\big(\nu^\gamma,\nu^0  \big) \le \frac{C}{(1+\log |\log \gamma|)^q }+ \frac{1}{|\log \gamma|^{\frac{n}{8}-1}}.
\end{align*} 
Since $q$ and $n$ are arbitrarily chosen, sending $\gamma$ to 0 produces \eqref{lim:gamma->0:Wasserstein:d_0}, as claimed.

Turning back to \eqref{lim:gamma->0:Wasserstein:nu^gamma-nu^0}, for any bivariate random variable $(X,Y)$ such that $X\sim\nu^\gamma$ and $Y\sim \nu^0$, we invoke Holder inequality to see that
\begin{align*}
&\E\Big[\sqrt{d_0(X,Y)\big(1+e^{\xi\|X\|^2_H}+e^{\xi\|Y\|^2_H}\big)}\Big] \\
&\le \sqrt{\E d_0(X,Y)\big(1+\E e^{\xi\|X\|^2_H}+\E e^{\xi\|Y\|^2_H}\big)}\\
&= \sqrt{\E d_0(X,Y)\Big(1+\int_H e^{\xi\|u\|^2_H}\nu^\gamma(\d u)+\int_H e^{\xi\|u\|^2_H}\nu^0(\d u)\Big)}.
\end{align*}
Since $(X,Y)$ is arbitrary, from \eqref{form:W_d}, we obtain
\begin{align} \label{ineq:W(d_0^xi)<W_(d_0)(1+e^u)}
\W_{d_0^\xi}(\nu^\gamma,\nu^0)& \le \sqrt{\W_{d_0}(\nu^\gamma,\nu^0)\Big(1+\int_H e^{\xi\|X\|^2_H}\nu^\gamma(\d u)+\int_H e^{\xi\|u\|^2_H}\nu^0(\d u)\Big)}\nonumber \\
&\le \frac{C}{(\log |\log \gamma|)^q}\Big(1+\int_H e^{\xi\|u\|^2_H}\nu^\gamma(\d u)+\int_H e^{\xi\|u\|^2_H}\nu^0(\d u)\Big).
\end{align}
Furthermore,  taking into account of Lemma \ref{lem:moment:nu_gamma} and Lemma \ref{lem:moment:nu^0}, we arrive at \eqref{lim:gamma->0:Wasserstein:nu^gamma-nu^0}, i.e.,
\begin{align*}
\W_{d_0^\xi}(\nu^\gamma,\nu^0)&\le \frac{C}{(\log |\log \gamma|)^q}.
\end{align*}
This completes the proof.

\end{proof}

We now turn to the proof of the auxiliary results. In order to prove Proposition \ref{prop:gamma->0:|u^gamma-u|:[0,T]}, we will not compare $\ug(t)$ and $u(t)$ directly, due to the difficulty induced by the cubic nonlinearity. Instead, we tackle the issue by modifying \eqref{eqn:Ginzburg_Landau} and \eqref{eqn:Schrodinger} as follows: for $R>0$, we introduce a smooth cut-off function $\f_R:[0,\infty)\to [0,1]$ defined as
\begin{align} \label{form:phi_R}
\f_R(x)= \begin{cases}
1, & 0\le x \le R,\\
\text{decreasing},& R\le x\le R+1,\\
0, & R+1\le x. 
\end{cases}
\end{align}
Given $\f_R$ above, consider the following truncating version of \eqref{eqn:Ginzburg_Landau}
\begin{align} \label{eqn:Ginzburg_Landau:truncated}
\d \ug_R (t) & = -(\gamma +\i) A \ug_R(t)\d t + \i |\ug_R(t)|^2\ug_R (t)\f_R(|\ug_R(t)|^2)\d t -\alpha \ug_R(t)\d t +Q\d W(t),
\end{align}
as well as of \eqref{eqn:Schrodinger}
\begin{align} \label{eqn:Schrodinger:truncated}
\d u_R (t) & = -\i A u_R (t)\d t + \i |u_R (t)|^2u_R  (t)\f_R(|u_R (t)|^2)\d t -\alpha u_R (t)\d t +Q\d W(t).
\end{align}
Observe that \eqref{eqn:Ginzburg_Landau:truncated} and \eqref{eqn:Schrodinger:truncated} are both Lipschitz systems, which allows for proving that $\ug_R$ can be approximated by $u_R$ on $[0,T]$. Then, we will remove the Lipschitz restriction by exploiting the uniform moment bounds in Lemma \ref{lem:moment:H_1} and Lemma \ref{lem:moment:H_1:Schrodinger}.

In Lemma \ref{lem:moment:u^gamma_R}, stated and proven next, we provide a uniform moment bound on $\ug_R$. In turn, this will be invoked to establish the convergence of $\ug_R$ toward $u_R$.

\begin{lemma} \label{lem:moment:u^gamma_R} For all $R>0$ and $u_0\in L^2(\Omega; H^1)$, let $\ug_R(t)$ be the solution of \eqref{eqn:Ginzburg_Landau:truncated} with initial condition $u_0$. Then, the following holds
\begin{align} \label{ineq:moment:u^gamma_R} 
\sup_{\gamma\in(0,1)}\gamma\E\int_0^T\|\ug_R(t)\|^2_{H^1}\textup{d} t \le \big(\E\|u_0\|^2_{H^1}+1\big)e^{
[CR^4+2\textup{Tr}(AQQ^*)]T}, \quad T>0,
\end{align}
for some positive constant $C$ independent of $\gamma, R,T,Q$ and $u_0$.
\end{lemma}
\begin{proof} From \eqref{eqn:Ginzburg_Landau:truncated}, we apply It\^o's formula to $\ug_R$ and obtain
\begin{align} \label{eqn:d.|u^gamma_R|^2_H1}
\d \|\ug_R(t)\|^2_{H^1} 
&= -2\gamma \|A\ug_R(t)\|^2_H\d t-2\alpha \|\ug_R(t)\|^2_{H^1}\d t+2\Tr(AQQ^*)\d t  \notag \\
&\qquad + \la \ug_R(t),\overline{Q\d W(t)}\ra_{H^1} + \la \ubar^\gamma_R(t),Q\d W(t)\ra_{H^1} \notag \\
&\qquad +\la \grad \ubar^\gamma_R(t),\i \grad \big(|\ug_R(t)|^2\ug_R(t)\f_R(|\ug_R(t)|^2)  \big)\ra_H\d t \notag\\
&\qquad+ \la \grad \ug_R(t),-\i \grad \big(|\ug_R(t)|^2\ubar^\gamma_R(t)\f_R(|\ug_R(t)|^2)  \big)\ra_H\d t .
\end{align} 
With regard to the last two terms on the above right-hand side, we note that
\begin{align*}
&\i\la \grad \ubar, \grad \big(|u|^2u\f_R(|u|^2)  \big)\ra_H \\
& = 2 \la |\grad u|^2, |u|^2\f_R(|u|^2) \ra_H +\la (\grad \ubar)^2,u^2\f_R(|u|^2)\ra_H\\
&\qquad + \la |\grad u|^2, |u|^4\f_R'(|u|^2) \ra_H + \la (\grad \ubar)^2,u^2|u|^2\f_R'(|u|^2)\ra_H \\
&\le C R^4\|u\|^2_{H^1}.
\end{align*}
Likewise, 
\begin{align*}
-\i\la \grad u, \grad \big(|u|^2\ubar\f_R(|u|^2)  \big)\ra_H \le C R^4\|u\|^2_{H^1}.
\end{align*}
It follows from \eqref{eqn:d.|u^gamma_R|^2_H1} that
\begin{align*} 
\E \|\ug_R(t)\|^2_{H^1}\le \E \|u_0\|^2_{H^1}+2\Tr(AQQ^*)t + CR^4\int_0^t \E \|\ug_R(s)\|^2_{H^1}\d s,
\end{align*} 
whence
\begin{align*}
\E \|\ug_R(t)\|^2_{H^1} \le \big(\E \|u_0\|^2_{H^1}+2\Tr(AQQ^*)t\big) e^{CR^4 t} \le   \big(\E \|u_0\|^2_{H^1}+1\big) e^{[CR^4+2\Tr(AQQ^*)] t} .
\end{align*}
Taking \eqref{eqn:d.|u^gamma_R|^2_H1} into account once again produces
\begin{align*}
\gamma\int_0^t\E\|A\ug_R(s)\|^2_H\d s\le  \big(\E \|u_0\|^2_{H^1}+1\big) e^{[CR^4+2\Tr(AQQ^*)] t} ,
\end{align*}
for some positive constant independent of $\gamma, R,t,Q$ and $u_0$, thereby establishing \eqref{ineq:moment:u^gamma_R} as claimed.

\end{proof}

Having obtained the moment bound \eqref{ineq:moment:u^gamma_R}, we now show that $\ug_R$ converges toward $u_R$ on $[0,T]$ in $H$. This is summarized in the following result.

\begin{lemma} \label{lem:gamma->0:truncated}
For all $R>0$ and $u_0\in L^2(\Omega; H^1)$, let $\ug_R(t)$ and $u_R(t)$ respectively be the solutions of \eqref{eqn:Ginzburg_Landau:truncated} and \eqref{eqn:Schrodinger:truncated} with initial condition $u_0$. Then, the following holds 
\begin{align} \label{lim:gamma->0:truncated}
\E\Big[\sup_{t\in[0,T]}\|\ug_R(t)-u_R(t)\|^2_H\Big] \le \gamma \big(\E\|u_0\|^2_{H^1}+1\big)e^{
[CR^4+\textup{Tr}(AQQ^*)]T},
\end{align}
 for some positive constant $C$ independent of $\gamma, R,T,Q$ and $u_0$.
 
\end{lemma}
\begin{proof}

Setting $v=\ug_R-u_R$, we subtract \eqref{eqn:Schrodinger:truncated} from \eqref{eqn:Ginzburg_Landau:truncated} and observe that $v$ satisfies the following equation
\begin{align*}
&\frac{\d}{\d t} v(t) = -\gamma A\ug_R(t)-\i Av(t)-\alpha v(t)\\
&+ \i \Big(  |\ug_R(t)|^2\ug_R (t)\f_R(|\ug_R(t)|^2)- |u_R(t)|^2u_R (t)\f_R(|u_R(t)|^2) \Big) ,
\end{align*}
with initial condition $v(0)=0$. Recalling $\f_R$ as in \eqref{form:phi_R}, we have
\begin{align*}
&\Big|\i \Big\la  |\ug_R(t)|^2\ug_R (t)\f_R(|\ug_R(t)|^2)- |u_R(t)|^2u_R (t)\f_R(|u_R(t)|^2), v(t) \Big\ra_H \Big|\\ 
&\quad \le  CR^4\|v(t)\|^2_H,
\end{align*}
whence
\begin{align*}
\frac{\d}{\d t} \|v(t)\|^2_H \le  -\gamma\la  A\ug_R(t),v(t)\ra_H-\alpha \|v(t)\|^2_H+CR^4\|v(t)\|^2_H.
\end{align*}
We invoke Cauchy-Schwarz inequality to further estimate
\begin{align*}
\frac{\d}{\d t} \|v(t)\|^2_H \le \gamma^2\|  A\ug_R(t)\|^2_H +CR^4\|v(t)\|^2_H.
\end{align*}
In turn, this implies the bound
\begin{align*}
\sup_{s\in[0,T]}\|v(s)\|^2_H  \le \gamma^2 \int_0^T \|  A\ug_R(t)\|^2_H\d t+ CR^4\int_0^T \sup_{s\in[0,t]} \|v(s)\|^2_H\d t.
\end{align*}
In view of Lemma \ref{lem:moment:u^gamma_R}, we deduce while making use of Gronwall's inequality
\begin{align*}
\E \sup_{s\in[0,T]}\|v(s)\|^2_H  &\le \gamma^2 \int_0^T \|  A\ug_R(t)\|^2_H\d t e^{CR^4 T}\\
&\le \gamma \big(\E\|u_0\|^2_{H^1}+1\big)e^{
[CR^4+2\textup{Tr}(AQQ^*)]T}.
\end{align*}
We emphasize that the positive constant $C$ is independent of $\gamma,R,T,Q$ and $u_0$. This produces \eqref{lim:gamma->0:truncated}, thereby finishing the proof.
\end{proof}

We now provide the proof of Proposition \ref{prop:gamma->0:|u^gamma-u|:[0,T]} by combining Lemma \ref{lem:gamma->0:truncated} together with Lemma \ref{lem:moment:H_1} and Lemma \ref{lem:moment:H_1:Schrodinger}.

\begin{proof}[Proof of Proposition \ref{prop:gamma->0:|u^gamma-u|:[0,T]}]

To remove the Lipschitz truncation as in Lemma~\ref{lem:gamma->0:truncated}, we shall employ an argument similarly to those in \cite[Section 6.2]{cerrai2020convergence} and \cite[Proposition 5.1]{nguyen2023small} tailored to our settings. 

Since the initial condition $u_0$ takes values in $L^2(\Omega;H^1)$, from Lemma \ref{lem:moment:H_1}, we note that $\ug(t)$ and $u(t)$ both belong to $H^1$. As a consequence, they are elements in $L^\infty$, thanks to Sobolev embedding in dimension $d=1$. So, for $R>0$, we introduce the stopping times $\tau^\gamma_R$ and $\tau_R$ given by
\begin{equation*}
\tau^\gamma_R=\inf\{t\ge 0:\|\ug(t)\|_{L^\infty}^2> R   \} ,
\end{equation*} 
and
\begin{align*}
\tau_R=\inf\{t\ge 0:\|u(t)\|_{L^\infty}^2>R\}.
\end{align*}
With the above stopping times, we have
\begin{align} \label{eqn:I^gamma_R1+I^gamma_R2}
&\E\Big[\sup_{t\in[0,T]}\|\ug(t)-u(t)\|_{H}\Big] \nonumber \\
&=\E\Big[ \mathbf{1}\{\tau^\gamma_R\mi \tau_R>T\}\sup_{t\in[0,T]}\|\ug(t)-u(t)\|_H\Big]+\E\Big[ \mathbf{1}\{\tau^\gamma_R\mi \tau_R\le T\}\sup_{t\in[0,T]}\|\ug(t)-u(t)\|_H\Big] \nonumber \\
&=I_{1,R}^{\gamma}(T)+I^\gamma_{2,R}(T).
\end{align}

Observe for $0\le t\le \tau^\gamma_R$, $ \ug(t) =\ug_R(t)$ where $\ug_R(t)$ is the solution of the truncating system \eqref{eqn:Ginzburg_Landau:truncated}. Likewise, for $0\le t\le \tau_R$, $u(t)=u_R(t)$ which solves \eqref{eqn:Schrodinger:truncated}. Taking into account of these facts, we recast $I^\gamma_{1,R}$ as
\begin{align} \label{ineq:I^gamma_R1}
I_{1,R}^{\gamma}(T) 
& = \E\Big[ \mathbf{1}\{\tau^\gamma_R\mi \tau_R>T\}\sup_{t\in[0,T]}\|\ug_R(t)-u_R(t)\|_H\Big]  \nonumber \\
&\le \E\Big[ \sup_{t\in[0,T]}\|\ug_R(t)-u_R(t)\|_H\Big]  \nonumber \\
&\le \sqrt{\gamma} \cdot \sqrt{\big(\E\|u_0\|^2_{H^1}+1\big)}\cdot e^{
[CR^4+\textup{Tr}(AQQ^*)]T},
\end{align}
where the last implication follows from Lemma \ref{lem:gamma->0:truncated}.

Concerning $I^\gamma_{2,R}$, we invoke Holder's inequality to infer
\begin{align*}
I^\gamma_{2,R}(t) &\le \Big(2 \E\Big[ \sup_{t\in[0,T]}\|\ug(t)\|^2_H\Big] + 2\E\Big[ \sup_{t\in[0,T]}\|u(t)\|^2_{H}\Big]\Big)^{1/2}\Big(\P\big(\tau^\gamma_R\le T\big)+\P\big(\tau_R\le T\big)\Big)^{1/2}.
\end{align*}
For $n\ge 1$, using Markov's inequality and $H^1\subset L^\infty$, we further estimate
\begin{align*}
\P\big(\tau^\gamma_R\le T\big)= \P\Big(\sup_{t\in[0,T]}\|\ug(t)\|^2_{L^\infty} \ge R\Big)
&\le \frac{1}{R^n}\E \Big[\sup_{t\in[0,T]}\|\ug(t)\|^{2n}_{L^\infty} \Big]\\
&\le  \frac{C}{R^n}\E \Big[\sup_{t\in[0,T]}\|\ug(t)\|^{2n}_{H^1} \Big].
\end{align*}
Recalling $\Phi$ defined in \eqref{form:Phi}, in view of \eqref{cond:Phi} and \eqref{ineq:moment:sup_[0,T]Phi(t)^n}, we obtain
\begin{align*}
\P\big(\tau^\gamma_R\le T\big) \le \frac{C}{R^n}\big( \E[\Phi(u_0)^n]+\E[\Phi(u_0)^{2n}] +T\big)\le \frac{C}{R^n}\big( \E[\Phi(u_0)^{2n}] +T+1\big).
\end{align*}
Likewise, from \eqref{ineq:moment:sup_[0,T]Phi(u(t))^n:Schrodinger}, we have
\begin{align*}
\P\big(\tau_R\le T\big) \le \frac{C}{R^n}\big( \E[\Phi(u_0)^{2n}] +T+1\big).
\end{align*}
Making use of \eqref{ineq:moment:sup_[0,T]Phi(t)^n} and \eqref{ineq:moment:sup_[0,T]Phi(u(t))^n:Schrodinger} again implies
\begin{align} \label{ineq:I^gamma_R2}
I^\gamma_{2,R}(t) \le \frac{C}{R^{\frac{n}{2}}} \big( \E[\Phi(u_0)^{2n}] +T+1\big).
\end{align}

Turning back to \eqref{eqn:I^gamma_R1+I^gamma_R2}, we combine \eqref{ineq:I^gamma_R1} and \eqref{ineq:I^gamma_R2} to deduce
\begin{align} \label{ineq:E|sup_[0,T].|u^gamma-u|_H<C(gamma,R)}
&\E\Big[\sup_{t\in[0,T]}\|\ug(t)-u(t)\|_{H}\Big] \nonumber \\
&\le \sqrt{\gamma} \cdot \sqrt{\big(\E\|u_0\|^2_{H^1}+1\big)}\cdot e^{
[CR^4+2\textup{Tr}(AQQ^*)]T} + \frac{C}{R^{\frac{n}{2}}} \big( \E[\Phi(u_0)^{2n}] +T+1\big).
\end{align}
We emphasize that the positive constant $C$ on the above right-hand side is independent of $\gamma,R,T,Q$ and $u_0$. For $\varepsilon\in(0,1/2)$, we pick $R>0$ satisfying
\begin{align*}
R^4 = \frac{-\log\big( \gamma^{\frac{1}{2}-\varepsilon}\big)}{CT}.
\end{align*}
Plugging into \eqref{ineq:E|sup_[0,T].|u^gamma-u|_H<C(gamma,R)} yields
\begin{align*}
&\E\Big[\sup_{t\in[0,T]}\|\ug(t)-u(t)\|_{H}\Big]\\
&\le \gamma^{\varepsilon}\sqrt{\big(\E\|u_0\|^2_{H^1}+1\big)}\cdot e^{2\textup{Tr}(AQQ^*)T} + C\Big(\frac{ T}{|\log \gamma|}\Big)^{\frac{n}{8}}\big( \E[\Phi(u_0)^{2n}] +T+1\big)\\
&\le \frac{ C}{|\log \gamma|^{\frac{n}{8}}}\big( \E[\Phi(u_0)^{2n}] +T+1\big)e^{2\textup{Tr}(AQQ^*)T}, \quad \gamma\to 0.
\end{align*}
Since $n$ is arbitrarily large, this produces \eqref{lim:gamma->0:|u^gamma-u|}, thereby completing the proof.

\end{proof}

Finally, we supply the proof of Lemma \ref{lem:moment:nu_gamma}, which together with Proposition \ref{prop:gamma->0:|u^gamma-u|:[0,T]} was employed to conclude Theorem \ref{thm:gamma->0:Wasserstein:nu^gamma-nu^0}.

\begin{proof}[Proof of Lemma \ref{lem:moment:nu_gamma}]
To avoid repetition, we provide the proof for the former exponential moment bound in \eqref{ineq:moment:mu_gamma}. The latter polynomial bound in $H^1$ can be established using an analogous argument.

Consider the set $B_R\subset H$ defined as
\begin{align*}
B_R = \{u\in H^1: \|u\|^2_H\le R\}.
\end{align*}
Since $\nu^\gamma(H)=1$, for $\varepsilon\in (0,1)$, there exists $R=R(\varepsilon,\gamma)>0$ such that 
\begin{align*}
\nu^\gamma(B_R^c)<\varepsilon.
\end{align*}
Next, given $m>0$, we set $\f_m(u)=e^{\xi\|u\|^2_H}\mi m$. On the one hand, by the invariance of $\nu^\gamma$, since $\f_m$ is bounded, 
\begin{align*}
\int_{H^1}P^\gamma_t \f_m(u)\nu^\gamma(\d u) =\int_{H^1} \f_m(u)\nu^\gamma(\d u).
\end{align*}
On the other hand, given the choice of $B_R$, we have
\begin{align*}
\int_{H^1}P^\gamma_t \f_m(u)\nu^\gamma(\d u)& =\Big\{\int_{B_R}+ \int_{B_R^c}\Big\}P^\gamma_t \f_m(u)\nu^\gamma(\d u)\\
&\le \int_{B_R}P^\gamma_t \f_m(u)\nu^\gamma(\d u)+m\varepsilon.
\end{align*}
With regard to the first term on the above right-hand side, for all $\xi$ sufficiently small, we invoke \eqref{ineq:moment:exp(|u|^2_H)} to see that for $u\in B_R$
\begin{align*}
P^\gamma_t \f_m(u) = \E \big[ e^{\xi\|\ug(t)\|^2_H}\mi m\big] \le e^{-c_\xi t}e^{\xi\|u\|^2_H} +C_\xi \le  e^{-c_\xi  t} e^{\xi R^2}+C_\xi,
\end{align*}
where $c_\xi$ and $C_\xi$ are the constants as in \eqref{ineq:moment:exp(|u|^2_H)}. It follows that
\begin{align*}
\int_{H^1}P^\gamma_t \f_m(u)\nu^\gamma(\d u)& \le \int_{B_R}P^\gamma_t \f_m(u)\nu^\gamma(\d u)+m\varepsilon\\
&\le e^{-c_\xi  t} e^{\xi R^2}+C_\xi +m\varepsilon.
\end{align*}
We may take $\varepsilon$ small and then take $t$ large enough, e.g., $\varepsilon<m^{-1}$ and $t>\xi R^2/c_\xi$, to deduce further that
\begin{align*}
\int_{H^1} \f_m(u)\nu^\gamma(\d u) \le C_\xi+2. 
\end{align*}
Since the above right-hand side is independent of $\gamma$ and $m$, by sending $m$ to infinity, we establish \eqref{ineq:moment:mu_gamma} by virtue of the Monotone Convergence Theorem. The proof is thus finished.

\end{proof}

\subsection{Proof of Theorem \ref{thm:gamma->0:phi}}

\label{sec:gamma->0:phi}

In this subsection, we will draw upon the framework of  \cite{
glatt2022short,glatt2021mixing,nguyen2023small} tailored to our settings, so as to establish the validity of the inviscid regime $\gamma\to 0$ on the infinite time horizon. More specifically, the proof of Theorem \ref{thm:gamma->0:phi} relies on the following intermediate result giving the convergence with respect to Wasserstein distances.

\begin{proposition} \label{prop:gamma->0:Wasserstein:[0,infty)}
Under the same hypotheses of Theorem \ref{thm:ergodicity:Ginzburg_Landau}, let $\{\ug_0\}_{\gamma\in(0,1)}$ be a sequence of deterministic initial conditions such that 
\begin{align*}
\sup_{\gamma\in(0,1)} \|\ug_0\|_{H^1} < R.
\end{align*}
Then, for all $R>0$, $q\ge 1$ and $\xi$ sufficiently small, the following holds 
\begin{align} \label{lim:gamma->0:Wasserstein:[0,infty)}
\sup_{t\ge 0}\W_{d_0^\xi}\big( P^\gamma_t \delta_{\ug_0}, P^0_t \delta_{\ug_0}\big) \le \frac{C}{(\log |\log \gamma|)^{q}}  ,\quad\textup{as} \,\gamma\to 0,
\end{align}
for some positive constant $C=C(R,\xi,q)$ independent of $\gamma$.
\end{proposition}

For the sake of clarity, the proof of Proposition \ref{prop:gamma->0:Wasserstein:[0,infty)} will be deferred to the end of this subsection. Assuming Proposition \ref{prop:gamma->0:Wasserstein:[0,infty)}, we are now in a position to conclude Theorem \ref{thm:gamma->0:phi}, whose argument is relatively short making use of observation \eqref{ineq:W_d(nu_1,nu_2):dual}.

\begin{proof}[Proof of Theorem \ref{thm:gamma->0:phi}] Let $\f:H^1\to \rbb$ satisfy $\|\f\|_{\textup{Lip},d_0^\xi}<\infty$. Since $d_0^\xi$ is only a distance-like function, in view of \eqref{ineq:W_d(nu_1,nu_2):dual}, it holds that
\begin{align*}
\W_{d_0^\xi}\big( P^\gamma_t\delta_{\ug_0},P^0_t\delta_{\ug_0}  \big) &\ge \frac{1}{\|\f\|_{\textup{Lip},d_0^\xi} } \Big|\int_H \f(u)P^\gamma_t\delta_{\ug_0}(\d u)-\int_H \f(u)P^0_t\delta_{\ug_0}(\d u) \Big|\\
&= \frac{1}{\|\f\|_{\textup{Lip},d_0^\xi} } \Big| \E\f\big(\ug(t;\ug_0)\big)-\E\f\big(u(t;\ug_0)\big)\Big|.
\end{align*}
In light of Proposition \ref{prop:gamma->0:Wasserstein:[0,infty)}, we immediately obtain
\begin{align*}
\sup_{t\ge 0} \Big| \E\f\big(\ug(t;\ug_0)\big)-\E\f\big(u(t;\ug_0)\big)\Big|& \le \|\f\|_{\textup{Lip},d_0^\xi} \sup_{t\ge 0} \W_{d_0^\xi}\big( P^\gamma_t\delta_{\ug_0},P^0_t\delta_{\ug_0}  \big) \\
&\le \|\f\|_{\textup{Lip},d_0^\xi} \frac{C(q,R,\xi) }{(\log |\log \gamma|)^{q}}. 
\end{align*}
This establishes \eqref{lim:gamma->0:phi}, thereby finishing the proof.

\end{proof}

Turning back to Proposition \ref{prop:gamma->0:Wasserstein:[0,infty)}, we will employ the following auxiliary result, which is considered as an analogue of Theorem \ref{thm:ergodicity:Ginzburg_Landau} in terms of Wasserstein distances.

\begin{lemma} \label{lem:ergodicity:Ginzburg-Landau:W_(d_1)}
For all $q\ge 1$, and probability measures $\mu_1,\mu_2\in \Pcal r(H^1)$, it holds that
\begin{align} \label{ineq:ergodicity:Ginzburg-Landau:W_(d_1)}
\W_{d_1}\big(P_t^\gamma\mu_1,P_t^\gamma\mu_2\big)  \le C(1+t)^{-q}\Big(1+\int_{H}\Phi(u)\mu_1(\textup{d} u)+\int_{H}\Phi(u)\mu_2(\textup{d} u)\Big),\quad t\ge 0,
\end{align}
where $C=C(q)$ is a positive constant independent of $\mu_1,\mu_2$ and $t$.
\end{lemma}
\begin{proof} Since $\W_{d_0}$ is convex, cf. \cite[Theorem 4.8]{villani2008optimal}, it holds that
\begin{align*}
\W_{d_1}\big(P_t^\gamma\mu_1,P_t^\gamma\mu_2\big) & \le \int_{H\times H}\close \W_{d_1}\big(P_t^\gamma(u_1,\cdot),P_t^\gamma(u_2,\cdot)\big) \pi(\d u_1,\d u_2), 
\end{align*}
where for a slight abuse of notation, we denote by $\pi$ a generic coupling of $(\mu_1,\mu_2)$. From the estimate \eqref{ineq:ergodicity:Ginzburg_landau}, we have
\begin{align*}
\sup_{\|\f\|_{\textup{Lip},d_1}\le 1}\big| \E \f\big(u(t;u_1)\big) - \E \f\big(u(t;u_2)\big) \big|\le C(1+t)^{-q}\big(1+\Phi(u_1)+\Phi(u_2)\big),\quad t\ge 0.
\end{align*} 
In light of the dual Kantorovich identity \eqref{form:W_d:dual-Kantorovich}, we deduce that
\begin{align*}
\W_{d_1}\big(P_t^\gamma(u_1,\cdot),P_t^\gamma(u_2,\cdot)\big) \le C(1+t)^{-q}\big(1+\Phi(u_1)+\Phi(u_2)\big),
\end{align*}
whence
\begin{align*}
\W_{d_1}\big(P_t^0\mu_1,P_t^0\mu_2\big) & \le C(1+t)^{-q} \int_{H\times H}\close\close\close\big(1+\Phi(u_1)+\Phi(u_2)\big) \pi(\d u_1,\d u_2)\\
& = C(1+t)^{-q} \Big(1+\int_{H}\Phi(u)\mu_1(\textup{d} u)+\int_{H}\Phi(u)\mu_2(\textup{d} u)\Big).
\end{align*}
In the above, the last implication follows from the fact that $\pi$ is a coupling for $(\mu_1,\mu_2)$. This produces \eqref{ineq:ergodicity:Ginzburg-Landau:W_(d_1)}, as claimed.

\end{proof}

Finally, we provide the proof of Proposition \ref{prop:gamma->0:Wasserstein:[0,infty)}.

\begin{proof}[Proof of Proposition \ref{prop:gamma->0:Wasserstein:[0,infty)}] 
Firstly, we proceed to establish an analogue of \eqref{lim:gamma->0:Wasserstein:[0,infty)} for $\W_{d_0}$. Then, we will upgrade to $\W_{d_0^\xi}$ by exploiting the exponential moment bounds in Lemma \ref{lem:moment:L^2} and Lemma \ref{lem:moment:H_1:Schrodinger}. 

Let $T$ be given and be chosen later. Since $\ug_0\in H^1$, from definition \eqref{form:W_d} and Proposition \ref{prop:gamma->0:|u^gamma-u|:[0,T]}, for all $t\in[0,T]$ and $n\ge 1$, we have
\begin{align} \label{ineq:W(u^gamma-u):[0,T]}
\sup_{t\in[0,T]}\W_{d_0}\big( P^\gamma_t \delta_{\ug_0}, P^0_t \delta_{\ug_0}\big)&\le \E\Big[\sup_{t\in[0,T]}\|\ug(t;\ug_0)-u(t;\ug_0)\|_H\Big] \nonumber \\
&\le \frac{ C}{|\log \gamma|^{\frac{n}{8}}}\big( \Phi(\ug_0)^{2n} +T+1\big)e^{2\textup{Tr}(AQQ^*)T}.
\end{align}
Now, for $t\ge T$, by the triangle inequality, it holds that
\begin{align}
\W_{d_0}\big( P^\gamma_t \delta_{\ug_0}, P^0_t \delta_{\ug_0}\big)&\le \W_{d_0}\big( P^\gamma_t \delta_{\ug_0}, \nu^\gamma\big)+ \W_{d_0}\big( \nu^\gamma, \nu^0\big)+ \W_{d_0}\big( \nu^0, P^0_t \delta_{\ug_0}\big).
\end{align}
On the one hand, from Theorem \ref{thm:gamma->0:Wasserstein:nu^gamma-nu^0}, we readily have for all $q\ge 1$
\begin{align*}
 \W_{d_0}\big( \nu^\gamma, \nu^0\big) \le \frac{C}{(\log |\log \gamma|)^q},
\end{align*}
holds for some positive constant $C=C(q)$ independent of $\gamma$. On the other hand, we employ invariance property together with Lemma \ref{lem:ergodicity:Ginzburg-Landau:W_(d_1)} and Lemma \ref{lem:ergodicity:Schrodinger} to obtain for all $q\ge 1$
\begin{align*}
&\W_{d_0}\big( P^\gamma_t \delta_{\ug_0}, \nu^\gamma\big)+ \W_{d_0}\big( \nu^0, P^0_t \delta_{\ug_0}\big)\\
&\le \W_{d_1}\big( P^\gamma_t \delta_{\ug_0}, P^\gamma_t\nu^\gamma\big)+ \W_{d_1}\big( P^0_t\nu^0, P^0_t \delta_{\ug_0}\big)\\
&\le C(1+t)^{-q} \Big(1+\Phi(\ug_0) + \int_H \Phi(u)\nu^\gamma(\d u)+ \int_H \Psi(u)\nu^0(\d u)  \Big)\\
&\le C(1+t)^{-q}\big(1+\Phi(\ug_0)\big).
\end{align*}
In the last implication above, we invoked the uniform moment bound on $\nu^\gamma$ and $\nu^0$ from Lemma \ref{lem:moment:nu_gamma} and Lemma \ref{lem:moment:nu^0}, respectively. It follows that for all $q_1,q_2\ge 1$ and $t\ge T$
\begin{align*}
\W_{d_0}\big( P^\gamma_t \delta_{\ug_0}, P^0_t \delta_{\ug_0}\big)
&\le \frac{C}{(\log |\log \gamma|)^{q_1}}+C(1+T)^{-q_2}\big(1+\Phi(\ug_0)\big).
\end{align*}
This together with \eqref{ineq:W(u^gamma-u):[0,T]} implies
\begin{align*}
\sup_{t\ge 0}\W_{d_0}\big( P^\gamma_t \delta_{\ug_0}, P^0_t \delta_{\ug_0}\big)
&\le \frac{ C}{|\log \gamma|^{\frac{n}{8}}}\big( \Phi(\ug_0)^{2n} +T+1\big)e^{2\textup{Tr}(AQQ^*)T}\\
&\qquad+\frac{C}{(\log |\log \gamma|)^{q_1}}+C(1+T)^{-q_2}\big(1+\Phi(\ug_0)\big).
\end{align*}
We emphasize once again that the above constant $C$ does not depend on $\gamma,T$ and $\ug_0$. Recall that $\sup_{\gamma\in(0,1)}\|\ug_0\|_{H^1}<R$, for all $\gamma$ sufficiently small, we pick 
\begin{align*}
T=\frac{n\log |\log\gamma|}{32\Tr(AQQ^*)},
\end{align*}
and obtain 
\begin{align*}
&\sup_{t\ge 0}\W_{d_0}\big( P^\gamma_t \delta_{\ug_0}, P^0_t \delta_{\ug_0}\big)\\
&\le \frac{ C}{|\log \gamma|^{\frac{n}{16}}}\big(\log |\log\gamma| +1\big)+\frac{C}{(\log |\log \gamma|)^{q_1}}+\frac{C}{(\log |\log \gamma|)^{q_2}}.
\end{align*}
Sending $\gamma$ to 0 implies the bound for all $q\ge 1$
\begin{align} \label{lim:gamma->0:W_d0:[0,infty)}
\sup_{t\ge 0}\W_{d_0}\big( P^\gamma_t \delta_{\ug_0}, P^0_t \delta_{\ug_0}\big)&\le \frac{C}{(\log |\log \gamma|)^{q}},
\end{align}
for some positive constant $C=C(R,q)$ independent of $\gamma$. 

Turning back to \eqref{lim:gamma->0:Wasserstein:[0,infty)}, similar to the proof of \eqref{ineq:W(d_0^xi)<W_(d_0)(1+e^u)}, we have that
\begin{align*}
&\W_{d_0^\xi}\big( P^\gamma_t\delta_{\ug_0},P^0_t\delta_{\ug_0}  \big) \\
& \le \sqrt{\W_{d_0}\big( P^\gamma_t\delta_{\ug_0},P^0_t\delta_{\ug_0} \big) \Big( 1+ \int_H e^{\xi \|u\|^2_H} P^\gamma_t\delta_{\ug_0}(\d u)+ \int_H e^{\xi \|u\|^2_H} P^0_t\delta_{\ug_0}(\d u) \Big)}\\
&= \sqrt{\W_{d_0}\big( P^\gamma_t\delta_{\ug_0},P^0_t\delta_{\ug_0} \big) \Big( 1+ \E e^{\xi \|\ug(t;\ug_0)\|^2_H}+ \E e^{\xi \|u(t;\ug_0)\|^2_H} \Big)}.
\end{align*}
From Lemma \ref{lem:moment:L^2} and Lemma \ref{lem:moment:H_1:Schrodinger}, we see that
\begin{align*}
\E e^{\xi \|\ug(t;\ug_0)\|^2_H}+ \E e^{\xi \|u(t;\ug_0)\|^2_H}
 &\le C\big(1+ e^{\xi \|\ug_0\|^2_H} \big)\le C(\xi,R).
\end{align*}
Combining with \eqref{lim:gamma->0:W_d0:[0,infty)}, we obtain
\begin{align*}
\W_{d_0^\xi}\big( P^\gamma_t\delta_{\ug_0},P^0_t\delta_{\ug_0}  \big)
&\le  \sqrt{\W_{d_0}\big( P^\gamma_t\delta_{\ug_0},P^0_t\delta_{\ug_0} \big) \Big( 1+ \E e^{\xi \|\ug(t;\ug_0)\|^2_H}+ \E e^{\xi \|u(t;\ug_0)\|^2_H} \Big)}\\
&\le \frac{C(q,R)}{(\log |\log \gamma|)^{q}} \cdot C(\xi,R).
\end{align*}
In turn, this establishes \eqref{lim:gamma->0:Wasserstein:[0,infty)}. The proof is thus complete.
\end{proof}

\section*{Acknowledgement}
The author would like to thank the anonymous referees for their providing a thorough review of this work. We appreciate their careful reading and insightful comments, which have improved the manuscript. The author also would like to thank Tuoc Phan and Diogo Pinheiro for fruitful discussion on the topics of the paper.





\appendix

\section{Estimates on the Schr\"odinger equation \eqref{eqn:Schrodinger}} \label{sec:Schrodinger}

In this section, we collect useful estimates on the system \eqref{eqn:Schrodinger} that were employed to prove the main results. We start with the polynomial bounds in $H^1$ and exponential moment bounds in $H$ subject to random initial conditions. The former was employed in the proof of Proposition \ref{prop:gamma->0:|u^gamma-u|:[0,T]} whereas the latter appeared in the proof of Proposition \ref{prop:gamma->0:Wasserstein:[0,infty)}. The proof of Lemma \ref{lem:moment:H_1:Schrodinger} is similar to those of Lemma \ref{lem:moment:L^2} and Lemma \ref{lem:moment:H_1}, and thus is omitted.

\begin{lemma} \label{lem:moment:H_1:Schrodinger}
Let $u_0\in L^2(\Omega;H^1)$ be given and $u(t)$ be the solution of \eqref{eqn:Schrodinger} with initial condition $u_0$.  

1. For all $n\ge 1$ and $T>0$, the following holds
\begin{align} \label{ineq:moment:sup_[0,T]Phi(u(t))^n:Schrodinger}
\E \big[\sup_{[0,T]}\Psi(u(t))^n\big] &\le \E\big[\Psi(u_0)^{n}\big]+ \E \big[\Psi(u_0)^{2n}\big] + C_{2,n}T ,
\end{align}
for some positive constant $C_{2,n}$ independent of $T$ and $u_0$. In the above, $\Psi$ is the function defined in \eqref{form:Psi}.

2. For all $\xi>0$ sufficiently small, the following holds
\begin{align}\label{ineq:moment:exp(xi|u|^2_H):Schrodinger}
\E\big[ e^{\xi\|u(t)\|^2_H}\big] \le e^{-c_\xi t}\E \big[ e^{\xi\|u_0\|^2_H}\big]+C_\xi,\quad t\ge 0,
\end{align}
for some positive constants $c_\xi$ and $C_\xi$ independent of $t$ and $u_0$.
\end{lemma}

Next, we recall the following ergodicity result from \cite{debussche2005ergodicity} giving the polynomial mixing rate of \eqref{eqn:Schrodinger}. Theorem \ref{thm:ergodicity:Schrodinger} was directly invoked in the proof of Theorem \ref{thm:gamma->0:Wasserstein:nu^gamma-nu^0}.

\begin{theorem}{\cite[Theorem 2.9]{debussche2005ergodicity}} \label{thm:ergodicity:Schrodinger} Let $u_1,u_2\in H^1$ be given and $\f:H^1\to\rbb$ be such that $\|\f\|_{\textup{Lip},{d_1}}<\infty$ where $d_1$ and $\|\f\|_{\textup{Lip},{d_1}}$ are defined in \eqref{form:d_k} and \eqref{form:Lipschitz}, respectively. Then, the following holds for all $q\ge 1$
\begin{align} \label{ineq:ergodicity:Schrodinger}
\big| \E \f\big(u(t;u_1)\big) - \E \f\big(u(t;u_2)\big) \big|\le C(1+t)^{-q}\|\f\|_{\textup{Lip},d_1}\big(1+\Psi(u_1)+\Psi(u_2)\big),\quad t\ge 0,
\end{align}
for some positive constant $C=C(q)$ independent of $u_1,u_2,t$ and $\f$. In the above, $\Psi$ is the function defined in \eqref{form:Psi}.
\end{theorem}

As a corollary, we obtain the following mixing result in terms of probability measures in $\Pcal r(H^1)$. We opt for neglecting the proof of Lemma \ref{lem:ergodicity:Schrodinger} and refer the reader to Lemma \ref{lem:ergodicity:Ginzburg-Landau:W_(d_1)} for a detailed argument. 
\begin{lemma} \label{lem:ergodicity:Schrodinger}
For all $q\ge 1$, and probability measures $\mu_1,\mu_2\in \Pcal r(H^1)$, it holds that
\begin{align} \label{ineq:ergodicity:Schrodinger:Wasserstein}
\W_{d_1}\big(P_t^0\mu_1,P_t^0\mu_2\big)  \le C(1+t)^{-q}\Big(1+\int_{H^1}\Psi(u)\mu_1(\textup{d} u)+\int_{H^1}\Psi(u)\mu_2(\textup{d} u)\Big),\quad t\ge 0,
\end{align}
where $C=C(q)$ is a positive constant independent of $\mu_1,\mu_2$ and $t$.
\end{lemma}

Lastly, we consider the invariant measure $\nu^0$ and establish two moment bounds on $\nu^0$. By exploiting the estimates in Lemma \ref{lem:moment:H_1:Schrodinger}, the proof of Lemma \ref{lem:moment:nu^0} below is similar to that of Lemma \ref{lem:moment:nu_gamma} and is thus omitted. 

\begin{lemma}  \label{lem:moment:nu^0}
Let $\nu^0$ be the unique invariant probability measure of \eqref{eqn:Schrodinger}. Then, for all $\xi>0$ sufficiently small and $n\ge 1$, the followings hold
\begin{align} \label{ineq:moment:nu^0}
\int_{H} e^{\xi\|u\|^2_H}\nu^0(\textup{d} u) <\infty \quad\textup{and}\quad \sup_{\gamma\in[0,1]}\int_{H} \Psi(u)^n \nu^0(\textup{d} u) <\infty.
\end{align}
\end{lemma}

\section{Auxiliary irreducibility conditions} \label{sec:auxiliary-results}

In this section, we consider the process $\eta(t)$ solving the following equation
\begin{align} \label{eqn:eta}
\d \eta(t) & = -\i A \eta(t)\d t-\alpha\eta(t)\d t +Q\d W(t),\quad\eta(0)= 0.
\end{align}
Observe that \eqref{eqn:eta} is essentially the linear part of \eqref{eqn:Schrodinger} without the cubic potential. In Lemma \ref{lem:irreducibility:eta}, we discuss an irreducibility condition that was invoked to establish Lemma \ref{lem:irreducibility:Ginzburg-Landau}.

\begin{lemma} \label{lem:irreducibility:eta}
Let $N$ be the constant as in Assumption \ref{cond:Q} and $\eta$ be the solution of the linear equation \eqref{eqn:eta}. Then, for all $r>0$ and $T>1$,
\begin{align} \label{ineq:irreducibility:eta}
\P\Big(\sup_{t\in[0,T]}\|\eta(t)\|_{H^2}\le r\Big)>\varepsilon,
\end{align}
for some positive constant $\varepsilon=\varepsilon(r,T)$ independent of the choice of $N$.
\end{lemma}
\begin{proof} We note that $\eta$ is explicitly given by
\begin{align*}
\eta(t)& = \int_0^t e^{-\alpha(t-r)}\sum_{k\ge 1}\sin(\alpha_k (t-r))\lambda_k \d B_k^1(r)e_k\\
&\qquad+ \int_0^t e^{-\alpha(t-r)}\sum_{k\ge 1}\cos(\alpha_k (t-r))\lambda_k \d B_k^2(r)e_k. 
\end{align*}
Concerning the terms involving sine functions, for $0\le s\le t$, we employ It\^o's Isometry to compute
\begin{align*}
& \E\Big\|\int_0^t e^{-\alpha(t-r)}\sin(\alpha_k (t-r))\lambda_k \d B_k^1(r)e_k-\int_0^s e^{-\alpha(s-r)}\sin(\alpha_k (s-r))\lambda_k \d B_k^1(r)e_k\Big\|^2_{H^2}\\
&=\E\Big\|\int_s^t e^{-\alpha(t-r)}\sin(\alpha_k (t-r))\lambda_k \d B_k^1(r)e_k\Big\|^2_{H^2}\\
&\quad+\E\Big\| \int_0^s \Big[e^{-\alpha(t-r)}\sin(\alpha_k (t-r))-e^{-\alpha(s-r)}\sin(\alpha_k (s-r))\Big]\lambda_k \d B_k^1(r)e_k\Big\|^2_{H^2}\\
&= \alpha_k^2\lambda_k^2\int_s^t e^{-2\alpha(t-r)}|\sin(\alpha_k (t-r))|^2\d r \\
&\qquad + \alpha_k^2\lambda_k^2\int_0^s \Big[e^{-\alpha(t-r)}\sin(\alpha_k (t-r))-e^{-\alpha(s-r)}\sin(\alpha_k (s-r))\Big]^2\d r .
\end{align*}
Together with the analogous computation for cosine functions, it follows that
\begin{align*}
&\E\|\la \eta(t)-\eta(s),e_k\ra_H e_k\|^2_{H^2}\\
&=  \alpha_k^2\lambda_k^2 \Big[\int_s^t \big| e^{-(\alpha+\i \alpha_k)(t-r)}\big|^2\d r+ \int_0^s \big|e^{-(\alpha+\i \alpha_k)(t-r)}-e^{-(\alpha+\i \alpha_k)(s-r)}\big|^2\d r\Big]\\
& =\alpha_k^2\lambda_k^2 \Big[\int_s^t e^{-2\alpha(t-r)}\d r+ \big|e^{-(\alpha+\i \alpha_k)t}-e^{-(\alpha+\i \alpha_k)s}\big|^2\int_0^s e^{2\alpha r}\d r\Big]\\
& = \frac{\alpha_k^2\lambda_k^2}{2\alpha} \Big[1-e^{-2\alpha(t-s)}+ \big|1-e^{-(\alpha+\i \alpha_k)(t-s)}\big|^2\big(1-e^{-2\alpha s}\big)\Big].
\end{align*}
Using the elementary inequality $|1-e^{-(a+\i b)x}|\le (a+b)x$ for $a,b,x\ge 0$, we deduce further that
\begin{align*}
&\E\|\la \eta(t)-\eta(s),e_k\ra_H e_k\|^2_{H^2}\\
&\le \frac{\alpha_k^2\lambda_k^2}{2\alpha}\big[ 2\alpha(t-s)+4(\alpha+\alpha_k)(t-s)\big]\le \Big(3+\frac{2}{\alpha}\Big)\alpha_k^3\lambda_k^2(t-s),
\end{align*}
whence
\begin{align*}
\E\|\eta(t)-\eta(s)\|^2_{H^2}\le \Big(3+\frac{2}{\alpha}\Big)\Tr(A^3QQ^*)(t-s).
\end{align*}
In light of \cite[Lemma 6.2]{csorgo1994almost} applying to the Gaussian process $\eta(t)$, for $r\in(0,1)$ and $T>1$,
\begin{align*}
\P\Big(\|\eta(t)\|^2_{H^2} \le \Big(3+\frac{2}{\alpha}\Big)\Tr(A^3QQ^*)r \Big) \ge  e^{-cT/r},
\end{align*}
holds for some positive constant $c$ independent of $T,r$ and $\Tr(A^3QQ^*)$. Recalling Assumption \ref{cond:Q}, cf. condition \eqref{cond:Q:Tr(A^(3/2)QQ)<infinity}, we deduce
\begin{align*}
\P\Big(\|\eta(t)\|^2_{H^2} \le \Big(3+\frac{2}{\alpha}\Big)C_Q r \Big) \ge  e^{-cT/r},
\end{align*}
which does not depend on the size of $N$. In turn, this produces \eqref{ineq:irreducibility:eta}, as claimed.

\end{proof}

\bibliographystyle{abbrv}
{\footnotesize\bibliography{wave-bib.bib}}

@book{ablowitz1981solitons,
  title={{Solitons and the Inverse Scattering Transform}},
  author={Ablowitz, Mark J and Segur, Harvey},
  year={1981},
  publisher={SIAM}
}

@article{agrawal2011nonlinear,
  title={Nonlinear fiber optics: its history and recent progress},
  author={Agrawal, Govind P},
  journal={J. Opt. Soc. Am. B},
  volume={28},
  number={12},
  pages={A1--A10},
  year={2011},
  publisher={Optica Publishing Group}
}

@book{albeverio2008spde,
  title={{SPDE in Hydrodynamics: Recent Progress and Prospects: Lectures given at the CIME Summer School held in Cetraro, Italy, August 29-September 3, 2005}},
  author={Albeverio, Sergio and Flandoli, Franco and Sinai, Yakov G},
  year={2008},
  publisher={Springer}
}

@article{arecchi1993transition,
  title={Transition from boundary-to bulk-controlled regimes in optical pattern formation},
  author={Arecchi, FT and Boccaletti, S and Ramazza, PL and Residori, S},
  journal={Phys. Rev. Lett.},
  volume={70},
  number={15},
  pages={2277},
  year={1993},
  publisher={APS}
}

@article{barton2004invariant,
  title={{Invariant measure for the stochastic Ginzburg Landau equation}},
  author={Barton-Smith, Marc},
  journal={Nonlinear Differ. Equ. Appl.},
  volume={11},
  pages={29--52},
  year={2004},
  publisher={Springer}
}

@article{barton2004global,
  title={{Global solution for a stochastic Ginzburg-Landau equation with multiplicative noise}},
  author={Barton-Smith, Marc},
  journal={Stoch. Anal. Appl.},
  volume={22},
  number={1},
  pages={1--18},
  year={2004},
  publisher={Taylor \& Francis}
}

@article{bechouche2000inviscid,
  title={{Inviscid limits of the complex Ginzburg--Landau equation}},
  author={Bechouche, Philippe and J{\"u}ngel, Ansgar},
  journal={Commun. Math. Phys.},
  volume={214},
  pages={201--226},
  year={2000},
  publisher={Springer}
}

@article{blennerhassett1980generation,
  title={On the generation of waves by wind},
  author={Blennerhassett, PJ},
  journal={	Philos. Trans. R. Soc. A},
  volume={298},
  number={1441},
  pages={451--494},
  year={1980},
  publisher={The Royal Society London}
}

@article{boritchev2013sharp,
  title={{Sharp estimates for turbulence in white-forced generalised Burgers equation}},
  author={Boritchev, Alexandre},
  journal={Geom. Funct. Anal.},
  volume={23},
  pages={1730--1771},
  year={2013},
  publisher={Springer}
}

@article{boulvard2023controllability,
  title={Controllability and ergodicity of three dimensional primitive equations driven by a finite-dimensional force},
  author={Boulvard, Pierre-Marie and Gao, Peng and Nersesyan, Vahagn},
  journal={Arch. Ration. Mech. Anal.},
  volume={247},
  number={1},
  pages={2},
  year={2023},
  publisher={Springer}
}

@article{bricmont2002exponential,
  title={{Exponential mixing of the 2D stochastic Navier-Stokes dynamics}},
  author={Bricmont, Jean and Kupiainen, Antti and Lefevere, Rapha{\"e}l},
  journal={Commun. Math. Phys.},
  volume={230},
  pages={87--132},
  year={2002},
  publisher={Springer}
}

@article{brzezniak2023invariant,
  title={{Invariant measures for a stochastic nonlinear and damped 2D Schr{\"o}dinger equation}},
  author={Brze{\'z}niak, Zdzis{\l}aw and Ferrario, Benedetta and Zanella, Margherita},
  journal={Nonlinearity},
  volume={37},
  number={1},
  pages={015001},
  year={2023},
  publisher={IOP Publishing}
}

@article{brzezniak2023ergodic,
  title={{Ergodic results for the stochastic nonlinear Schr{\"o}dinger equation with large damping}},
  author={Brze{\'z}niak, Zdzislaw and Ferrario, Benedetta and Zanella, Margherita},
  journal={J. Evol. Equ.},
  volume={23},
  number={1},
  pages={19},
  year={2023},
  publisher={Springer}
}

@article{brzezniak2024global,
  title={{Global well posedness and ergodic results in regular Sobolev spaces for the nonlinear Schr{\"o}dinger equation with multiplicative noise and arbitrary power of the nonlinearity}},
  author={Brze{\'z}niak, Zdzis{\l}aw and Ferrario, Benedetta and Maurelli, Mario and Zanella, Margherita},
  journal={Discrete Contin. Dyn. Syst.},
  volume={45},
  number={9},
  pages={3217--3257},
  year={2025}
}

@article{butkovsky2014subgeometric,
  title={{Subgeometric rates of convergence of Markov processes in the Wasserstein metric}},
  author={Butkovsky, Oleg},
  journal={Ann. Appl. Probab.},
  volume={24},
  number={2},
  pages={526--552},
  year={2014},
  publisher={Institute of Mathematical Statistics}
}

@article{butkovsky2020generalized,
  title={{Generalized couplings and ergodic rates for SPDEs and other Markov models}},
  author={Butkovsky, Oleg and Kulik, Alexei and Scheutzow, Michael},
  journal={Ann. Appl. Probab.},
  volume={30},
  number={1},
  pages={1--39},
  year={2020},
  publisher={Institute of Mathematical Statistics}
}

@article{cerrai2006smoluchowski,
  title={{On the Smoluchowski-Kramers approximation for a system with an infinite number of degrees of freedom}},
  author={Cerrai, Sandra and Freidlin, Mark},
  journal={Probab. Theory Relat. Fields},
  volume={135},
  number={3},
  pages={363--394},
  year={2006},
  publisher={Springer}
}

@article{cerrai2006smoluchowski2,
  title={{Smoluchowski-Kramers approximation for a general class of SPDEs}},
  author={Cerrai, Sandra and Freidlin, Mark},
  journal={J. Evol. Equ.},
  volume={6},
  number={4},
  pages={657--689},
  year={2006},
  publisher={Springer}
}

@article{cerrai2020convergence,
  title={On the convergence of stationary solutions in the Smoluchowski-Kramers approximation of infinite dimensional systems},
  author={Cerrai, Sandra and Glatt-Holtz, Nathan},
  journal={J. Funct. Anal.},
  volume={278},
  number={8},
  pages={108421},
  year={2020},
  publisher={Elsevier}
}

@article{cerrai2022smoluchowski,
  title={{A Smoluchowski--Kramers approximation for an infinite dimensional system with state--dependent damping}},
  author={Cerrai, Sandra and Xi, Guangyu},
  journal={Ann. Probab.},
  volume={50},
  number={3},
  pages={874--904},
  year={2022},
  publisher={Institute of Mathematical Statistics}
}

@article{cerrai2022small,
  title={{On the small noise limit in the Smoluchowski-Kramers approximation of nonlinear wave equations with variable friction}},
  author={Cerrai, Sandra and Xie, Mengzi},
  journal={Trans. Am. Math. Soc.},
  volume={376},
  number={11},
  pages={7651--7689},
  year={2023}
}

@article{cerrai2023small,
  title={On the small-mass limit for stationary solutions of stochastic wave equations with state dependent friction},
  author={Cerrai, Sandra and Xie, Mengzi},
  journal={Appl. Math. Optim.},
  volume={90},
  number={1},
  pages={7},
  year={2024},
  publisher={Springer}
}

@article{csorgo1994almost,
  title={{On almost sure limit inferior for B-valued stochastic processes and applications}},
  author={Cs{\"o}rgo, Mikl{\'o}s and Shao, Qi-Man},
  journal={Probab. Theory Relat. Fields},
  volume={99},
  number={1},
  pages={29--54},
  year={1994},
  publisher={Springer}
}

@article{debussche2005ergodicity,
  title={{Ergodicity for a weakly damped stochastic non-linear Schr{\"o}dinger equation}},
  author={Debussche, Arnaud and Odasso, Cyril},
  journal={J. Evol. Equ.},
  volume={5},
  number={3},
  pages={317--356},
  year={2005},
  publisher={Springer}
}

@article{dias1999nonlinear,
  title={Nonlinear gravity and capillary-gravity waves},
  author={Dias, Fr{\'e}d{\'e}ric and Kharif, Christian},
  journal={Annu. Rev. Fluid Mech.},
  volume={31},
  number={1},
  pages={301--346},
  year={1999},
  publisher={Annual Reviews 4139 El Camino Way, PO Box 10139, Palo Alto, CA 94303-0139, USA}
}

@article{dong2023ergodicity,
  title={Ergodicity for stochastic conservation laws with multiplicative noise},
  author={Dong, Zhao and Zhang, Rangrang and Zhang, Tusheng},
  journal={Commun. Math. Phys.},
  volume={400},
  number={3},
  pages={1739--1789},
  year={2023},
  publisher={Springer}
}

@article{weinan2001gibbsian,
  title={{Gibbsian Dynamics and Ergodicity for the Stochastically Forced Navier--Stokes Equation}},
  author={E, Weinan and Mattingly, Jonanthan C and Sinai, Yakov},
  journal={Commun. Math. Phys.},
  volume={224},
  number={1},
  pages={83--106},
  year={2001},
  publisher={Springer}
}

@article{ekren2017existence,
  title={{Existence of invariant measures for the stochastic damped Schr{\"o}dinger equation}},
  author={Ekren, Ibrahim and Kukavica, Igor and Ziane, Mohammed},
  journal={Stoch. Partial Differ. Equ.: Anal. Comput.},
  volume={5},
  pages={343--367},
  year={2017},
  publisher={Springer}
}

@book{fibich2015nonlinear,
  title={{The Nonlinear Schr{\"o}dinger Equation}},
  author={Fibich, Gadi},
  volume={192},
  year={2015},
  publisher={Springer}
}

@article{foldes2015ergodic,
  title={{Ergodic and mixing properties of the Boussinesq equations with a degenerate random forcing}},
  author={F{\"o}ldes, Juraj and Glatt-Holtz, Nathan and Richards, Geordie and Thomann, Enrique},
  journal={J. Funct. Anal.},
  volume={269},
  number={8},
  pages={2427--2504},
  year={2015},
  publisher={Elsevier}
}

@article{foldes2016ergodicity,
  title={{Ergodicity in randomly forced Rayleigh--B{\'e}nard convection}},
  author={F{\"o}ldes, J and Glatt-Holtz, NE and Richards, G and Whitehead, JP},
  journal={Nonlinearity},
  volume={29},
  number={11},
  pages={3309},
  year={2016},
  publisher={IOP Publishing}
}

@article{foldes2017asymptotic,
  title={{Asymptotic analysis for randomly forced MHD}},
  author={Foldes, Juraj and Friedlander, Susan and Glatt-Holtz, Nathan and Richards, Geordie},
  journal={SIAM J. Math. Anal.},
  volume={49},
  number={6},
  pages={4440--4469},
  year={2017},
  publisher={SIAM}
}

@article{foldes2019large,
  title={{Large Prandtl number asymptotics in randomly forced turbulent convection}},
  author={F{\"o}ldes, Juraj and Glatt-Holtz, Nathan E and Richards, Geordie},
  journal={Nonlinear Differ. Equ. Appl. NoDEA},
  volume={26},
  number={6},
  pages={43},
  year={2019},
  publisher={Springer}
}

@article{gao2024polynomial,
  title={{Polynomial mixing for white-forced Kuramoto-Sivashinsky equation on the whole line}},
  author={Gao, Peng},
  journal={arXiv preprint arXiv:2408.00592},
  year={2024}
}

@article{gao2024polynomialWAVE,
  title={Polynomial mixing for the white-forced wave equation on the whole line},
  author={Gao, Peng},
  journal={arXiv preprint arXiv:2412.13230},
  year={2024}
}

@article{ginzburg1950theory,
  title={On the theory of superconductivity},
  author={Ginzburg, V.L. and Landau, L.D.},
  journal={Zh. Eksp. Teor. Fiz.},
  volume={20},
  pages={1064--1082},
  year={1950}
  
}

@article{glatt2008stochastic,
  title={The stochastic primitive equations in two space dimensions with multiplicative noise},
  author={Glatt-Holtz, Nathan and Ziane, Mohammed},
  journal={Discrete Contin. Dyn. Syst. - B},
  volume={10},
  number={4},
  pages={801},
  year={2008},
  publisher={American Institute of Mathematical Sciences}
}

@article{hairer2002exponential,
  title={{Exponential mixing properties of stochastic PDEs through asymptotic coupling}},
  author={Hairer, Martin},
  journal={Probab. Theory Relat. Fields},
  volume={124},
  number={3},
  pages={345--380},
  year={2002},
  publisher={Springer}
}

@article{hairer2006ergodicity,
  title={{Ergodicity of the 2D Navier-Stokes equations with degenerate stochastic forcing}},
  author={Hairer, Martin and Mattingly, Jonathan C},
  journal={Ann. Math.},
  pages={993--1032},
  year={2006},
  publisher={JSTOR}
}

@article{hairer2008spectral,
  title={{Spectral gaps in Wasserstein distances and the 2D stochastic Navier--Stokes equations}},
  author={Hairer, Martin and Mattingly, Jonathan C},
  journal={Ann. Probab.},
  volume={36},
  number={6},
  pages={2050--2091},
  year={2008},
  publisher={Institute of Mathematical Statistics}
}

@article{hairer2011asymptotic,
  title={{Asymptotic coupling and a general form of Harris' theorem with applications to stochastic delay equations}},
  author={Hairer, Martin and Mattingly, Jonathan C and Scheutzow, Michael},
  journal={Probab. Theory Relat. Fields},
  volume={149},
  number={1},
  pages={223--259},
  year={2011},
  publisher={Springer}
}

@article{hairer2011theory,
  title={{A theory of hypoellipticity and unique ergodicity for semilinear stochastic PDEs}},
  author={Hairer, Martin and Mattingly, Jonathan C},
  journal={Electron. J. Probab.},
  volume={16},
  pages={658--738},
  year={2011},
  publisher={The Institute of Mathematical Statistics and the Bernoulli Society}
}

@article{glatt2021mixing,
  title={{Mixing rates for Hamiltonian Monte Carlo algorithms in finite and infinite dimensions}},
  author={Glatt-Holtz, Nathan E and Mondaini, Cecilia F},
  journal={Stoch. Partial Differ. Equ.: Anal. Comput.},
  volume={10},
  number={4},
  pages={1318--1391},
  year={2022},
  publisher={Springer}
}

@article{glatt2022short,
  title={{The short memory limit for long time statistics in a stochastic Coleman--Gurtin model of heat conduction}},
  author={Glatt-Holtz, Nathan E and Martinez, Vincent R and Nguyen, Hung D},
  journal={arXiv preprint arXiv:2212.05646},
  year={2022}
}

@article{huang2008inviscid,
  title={{Inviscid limit for the energy-critical complex Ginzburg--Landau equation}},
  author={Huang, Chunyan and Wang, Baoxiang},
  journal={J. Funct. Anal.},
  volume={255},
  number={3},
  pages={681--725},
  year={2008},
  publisher={Elsevier}
}

@article{huber1993universal,
  title={Universal Decay of vortex density in two dimensions},
  author={Huber, Greg and Alstr, Preben and others},
  journal={Physica A},
  volume={195},
  number={3-4},
  pages={448--456},
  year={1993},
  publisher={Elsevier}
}

@book{karatzas2012brownian,
  title={{Brownian Motion and Stochastic Calculus}},
  author={Karatzas, Ioannis and Shreve, Steven},
  volume={113},
  year={2012},
  publisher={Springer Science \& Business Media}
}

@article{kuksin2000stochastic,
  title={{Stochastic dissipative PDE's and Gibbs measures}},
  author={Kuksin, Sergei and Shirikyan, Armen},
  journal={Commun. Math. Phys.},
  volume={213},
  pages={291--330},
  year={2000},
  publisher={Springer}
}

@article{kuksin2001coupling,
  title={{A Coupling Approach to Randomly Forced Nonlinear PDE's. I}},
  author={Kuksin, Sergei and Shirikyan, Armen},
  journal={Commun. Math. Phys.},
  volume={221},
  pages={351--366},
  year={2001},
  publisher={Springer-Verlag}
}

@article{kuksin2002coupling,
  title={{A coupling approach to randomly forced nonlinear PDE's. II}},
  author={Kuksin, Sergei and Piatniski, Andrey and Shirikyan, Armen},
  journal={Commun. Math. Phys.},
  volume={230},
  pages={81--85},
  year={2002},
  publisher={Springer}
}

@article{kuksin2002couplingb,
  title={{Coupling approach to white-forced nonlinear PDEs}},
  author={Kuksin, Sergei and Shirikyan, Armen},
  journal={J. Math. Pures Appl.},
  volume={81},
  number={6},
  pages={567--602},
  year={2002},
  publisher={Elsevier}
}

@article{kuksin2004randomly,
  title={{Randomly forced CGL equation: stationary measures and the inviscid limit}},
  author={Kuksin, Sergei and Shirikyan, Armen},
  journal={J. Phys. A},
  volume={37},
  number={12},
  pages={3805},
  year={2004},
  publisher={IOP Publishing}
}

@book{kuksin2012mathematics,
  title={{Mathematics of Two-dimensional Turbulence}},
  author={Kuksin, Sergei and Shirikyan, Armen},
  volume={194},
  year={2012},
  publisher={Cambridge University Press}
}

@article{kuksin2013weakly,
  title={{Weakly nonlinear stochastic CGL equations}},
  author={Kuksin, Sergei B},
  journal={Ann. inst. Henri Poincare (B) Probab. Stat.},
  volume={49},
  number={4},
  pages={1033--1056},
  year={2013}
}

@article{kuramoto1975formation,
  title={{On the formation of dissipative structures in reaction-diffusion systems: Reductive perturbation approach}},
  author={Kuramoto, Yoshiki and Tsuzuki, Toshio},
  journal={Prog. Theor. Phys.},
  volume={54},
  number={3},
  pages={687--699},
  year={1975},
  publisher={Oxford University Press}
}

@article{machihara2003inviscid,
  title={{The inviscid limit for the complex Ginzburg--Landau equation}},
  author={Machihara, Shuji and Nakamura, Yoshihisa},
  journal={J. Math. Anal. Appl.},
  volume={281},
  number={2},
  pages={552--564},
  year={2003},
  publisher={Elsevier}
}

@article{mattingly2002exponential,
  title={{Exponential convergence for the stochastically forced Navier-Stokes equations and other partially dissipative dynamics}},
  author={Mattingly, Jonathan C},
  journal={Commun. Math. Phys.},
  volume={230},
  number={3},
  pages={421--462},
  year={2002},
  publisher={Springer}
}

@article{mattingly2002ergodicity,
  title={{Ergodicity for SDEs and approximations: locally Lipschitz vector fields and degenerate noise}},
  author={Mattingly, Jonathan C and Stuart, Andrew M and Higham, Desmond J},
  journal={Stoch. Process. Their Appl.},
  volume={101},
  number={2},
  pages={185--232},
  year={2002},
  publisher={Elsevier}
}

@article{moon1982three,
  title={{Three-frequency motion and chaos in the Ginzburg-Landau equation}},
  author={Moon, Hie-Tae and Huerre, P and Redekopp, LG},
  journal={Phys. Rev. Lett.},
  volume={49},
  number={7},
  pages={458},
  year={1982},
  publisher={APS}
}

@article{moon1983transitions,
  title={{Transitions to chaos in the Ginzburg-Landau equation}},
  author={Moon, Hie Tae and Huerre, P and Redekopp, LG},
  journal={Physica D},
  volume={7},
  number={1-3},
  pages={135--150},
  year={1983},
  publisher={Elsevier}
}

@article{nersesyan2008polynomial,
  title={{Polynomial mixing for the complex Ginzburg--Landau equation perturbed by a random force at random times}},
  author={Nersesyan, Vahagn},
  journal={J. Evol. Equ.},
  volume={8},
  number={1},
  pages={1--29},
  year={2008},
  publisher={Springer}
}

@article{nersesyan2024exponential,
  title={{Exponential mixing for the white-forced complex Ginzburg--Landau equation in the whole space}},
  author={Nersesyan, Vahagn and Zhao, Meng},
  journal={SIAM J. Math. Anal.},
  volume={56},
  number={3},
  pages={3646--3678},
  year={2024},
  publisher={SIAM}
}

@article{nersesyan2024polynomial,
  title={{Polynomial mixing for the white-forced Navier-Stokes system in the whole space}},
  author={Nersesyan, Vahagn and Zhao, Meng},
  journal={arXiv preprint arXiv:2410.15727},
  year={2024}
}

@article{newell1969finite,
  title={Finite bandwidth, finite amplitude convection},
  author={Newell, Alan C and Whitehead, John A},
  journal={J. Fluid Mech.},
  volume={38},
  number={2},
  pages={279--303},
  year={1969},
  publisher={Cambridge University Press}
}

@book{newell1985solitons,
  title={{Solitons in Mathematics and Physics}},
  author={Newell, Alan C},
  year={1985},
  publisher={SIAM}
}

@article{nguyen2018small,
  title={{The small-mass limit and white-noise limit of an infinite dimensional generalized Langevin equation}},
  author={Nguyen, Hung D},
  journal={J. Stat. Phys.},
  volume={173},
  number={2},
  pages={411--437},
  year={2018},
  publisher={Springer}
}

@article{nguyen2023small,
  title={The small mass limit for long time statistics of a stochastic nonlinear damped wave equation},
  author={Nguyen, Hung D},
  journal={J. Diff. Equ.},
  volume={371},
  pages={481--548},
  year={2023},
  publisher={Elsevier}
}

@article{nguyen2024polynomial,
  title={Polynomial mixing of a stochastic wave equation with dissipative damping},
  author={Nguyen, Hung D},
  journal={Appl. Math. Opt.},
  volume={89},
  pages={1--31},
  year={2024},
  publisher={Springer}
}

@article{nguyen2025polynomial,
  title={{Polynomial mixing for the stochastic Schr\"odinger equation with large damping in the whole space}},
  author={Nguyen, Hung D and Seong, Kihoon},
  journal={arXiv preprint arXiv:2512.24599},
  year={2025}
}

@article{odasso2006ergodicity,
  title={{Ergodicity for the stochastic complex Ginzburg-Landau equations}},
  author={Odasso, Cyril},
  journal={Ann. Inst. Henri Poincare (B) Probab. Stat.},
  volume={42},
  number={4},
  pages={417--454},
  year={2006}
}

@article{shirikyan2004exponential,
  title={{Exponential mixing for 2D Navier-Stokes equations perturbed by an unbounded noise}},
  author={Shirikyan, Armen},
  journal={J. Math. Fluid Mech.},
  volume={6},
  pages={169--193},
  year={2004},
  publisher={Springer}
}

@article{shirikyan2011local,
  title={{Local times for solutions of the complex Ginzburg--Landau equation and the inviscid limit}},
  author={Shirikyan, Armen},
  journal={J. Math. Anal. Appl.},
  volume={384},
  number={1},
  pages={130--137},
  year={2011},
  publisher={Elsevier}
}

@article{stuart1978frs,
  title={{The Eckhaus and Benjamin-Feir resonance mechanisms}},
  author={Stuart, J. T. and Di Prima, D. C.},
  journal={Proc. R. Soc. London Ser. A},
  volume={362},
  pages={27--41},
  year={1978}
}

@book{temam2012infinite,
  title={{Infinite-dimensional Dynamical Systems in Mechanics and Physics}},
  author={Temam, Roger},
  volume={68},
  year={2012},
  publisher={Springer Science \& Business Media}
}

@book{villani2008optimal,
  title={{Optimal Transport: Old and New}},
  author={Villani, C{\'e}dric},
  volume={338},
  year={2008},
  publisher={Springer Science \& Business Media}
}

@article{wang2002limit,
  title={{The limit behavior of solutions for the Cauchy problem of the complex Ginzburg-Landau equation}},
  author={Wang, Baoxiang},
  journal={Commun. Pure Appl. Math.},
  volume={55},
  number={4},
  pages={481--508},
  year={2002},
  publisher={Wiley Online Library}
}

@article{wang2004inviscid,
  title={{The inviscid limit of the derivative complex Ginzburg--Landau equation}},
  author={Baoxiang, Wang and Youde, Wang},
  journal={J. Math. Pures Appl.},
  volume={83},
  number={4},
  pages={477--502},
  year={2004},
  publisher={Elsevier}
}

@article{wu1998inviscid,
  title={{The inviscid limit of the complex Ginzburg--Landau equation}},
  author={Wu, Jiahong},
  journal={J. Differ. Equ.},
  volume={142},
  number={2},
  pages={413--433},
  year={1998},
  publisher={Academic Press}
}

@article{zine2022inviscid,
  title={{On the inviscid limit of the singular stochastic complex Ginzburg-Landau equation at statistical equilibrium}},
  author={Zine, Younes},
  journal={arXiv preprint arXiv:2212.00604},
  year={2022}
}

\end{document}